\RequirePackage{fix-cm}
\documentclass[smallextended,numbook,envcountsame]{svjour3}    
\smartqed  
\usepackage{graphicx}
 \usepackage{mathptmx}      
 \journalname{Probability Theory and Related Fields}
\usepackage{amsmath,amssymb,amsfonts,mathrsfs}
\RequirePackage[numbers,comma,square,sort&compress]{natbib}

\usepackage{graphicx,tikz,subfigure,color}
\usetikzlibrary{decorations.pathreplacing}

\usepackage{nicefrac}
\usepackage{enumerate}

\RequirePackage{color}
\RequirePackage[colorlinks, urlcolor=my-blue,linkcolor=my-link,citecolor=my-green]{hyperref}
\definecolor{my-link}{rgb}{0.5,0.0,0.0}
\definecolor{my-blue}{rgb}{0.0,0.0,0.6}
\definecolor{my-red}{rgb}{0.5,0.0,0.0}
\definecolor{my-green}{rgb}{0.0,0.5,0.0}
\definecolor{nicos-red}{rgb}{0.75,0.0,0.0}
\definecolor{light-gray}{gray}{0.6}
\definecolor{really-light-gray}{gray}{0.8}

\newcommand{\be}{\begin{equation}}
\newcommand{\ee}{\end{equation}}

\providecommand{\abs}[1]{\vert#1\vert}

\providecommand{\pp}[1]{\langle#1\rangle}
\newcommand{\fl}[1]{\lfloor{#1}\rfloor}

\def\Coc{\mathscr K}
\def\Cor{\Coc_0}
\def\cC{\mathcal{C}}  
   \def\cE{\mathcal{E}}
\def\cG{\mathcal{G}}

\def\esssup{\mathop{\mathrm{ess\,sup}}}

\def\Omhat{\widehat\Omega}

\def\what{\widehat\w}

\newcommand{\one}{\mbox{\mymathbb{1}}}

\def\kS{\mathfrak{S}}

\def\bE{\mathbb{E}}
\def\bN{\mathbb{N}}
\def\bP{\mathbb{P}}

\def\bR{\mathbb{R}}
\def\bZ{\mathbb{Z}}

 \def\Z{\bZ}  \def\R{\bR}\def\N{\bN}

\def\w{\omega}
\def\om{\omega}

\def\e{\varepsilon}

\def\ddd{\displaystyle} 
\def\ri{{\mathrm{ri\,}}}

\def\m1{\mathbf{1}}

\def\kS{\mathfrak S}

 \def\Vvv{{\rm\mathbb{V}ar}}

 \def\wt{\widetilde}  \def\wh{\widehat} \def\wb{\overline}
 
\def\E{\bE}
\def\P{\bP} 

\def\funct lp{L} 
\def\funct lpbar{\bar L} 

\font \mymathbb = bbold10 at 11pt

\def\range{\mathcal R}
\def\Uset{\mathcal U}
\def\Diff{\mathcal D}
\def\EP{\mathcal E}

\def\gpp{{g_{\text{pp}}}}
\def\gpl{{g_{\text{pl}}}}
\def\gppa{\gamma}

\def\Gsw{G}

\def\f{f}

\def\cA{{\mathcal A}}

\def\Gpl{G}
\def\Gpp{G}
\def\gpp{g_{\text{\rm pp}}}
\def\gpl{g_{\text{\rm pl}}}

\def\ri{{\mathrm{ri\,}}}
\def\f{{f}} 
\def\gppa{{\gamma}}
\def\amin{{\underline\alpha}}
\def\amax{{\overline\alpha}}

\def\ximin{{\underline\xi}}
\def\ximax{{\overline\xi}}
\def\zetamin{{\underline\zeta}}
\def\zetamax{{\overline\zeta}}

\def\B{{B}}

\def\cE{{\mathcal E}}

\def\W{{W}} 
\def\S{{S}}
\def\A{{A}}
\def\D{{D}}
\def\Wp{{\wt W}}
\def\Sp{{\wt S}}
\def\Ap{{\wt A}}
\def\M{{\mathcal M}}
\def\U{{U}}
\def\Ombig{{\widehat\Omega}}
\def\Pbig{{\widehat\P}}
\def\Ebig{{\widehat\E}}
\def\kSbig{{\widehat\kS}}
\def\OBP{(\Omega, \kS, \P)} 
\def\OBPbig{(\Ombig, \kSbig, \Pbig)}

\def\Gsw{G^{\rm SW}}
\def\Gne{G^{\rm NE}}

\def\Gn{G^{\rm N}}

\def\pp{p}  
\def\Ew{m_0}   
\def\opc{\vec p_c} 

\def\cla{\cO} 
 
\def\etamin{{\underline\eta}}
\def\etamax{{\overline\eta}}

\def\pc{\vec p_c}
\def\cO{\mathcal O}

\definecolor{darkgreen}{rgb}{0.0,0.5,0.0}
\definecolor{darkblue}{rgb}{0.0,0.0,0.3}
\definecolor{nicosred}{rgb}{0.65,0.1,0.1}
\definecolor{light-gray}{gray}{0.7}
\allowdisplaybreaks[1]

\allowdisplaybreaks[1]
\newenvironment{proofof}[2]{\removelastskip\vspace{6pt}\noindent
 {\it Proof  #1.}~\rm#2}{\par\vspace{6pt}}

\begin{document}

\title{Stationary cocycles and Busemann functions\\  for the corner growth model
}

\titlerunning{Corner growth model}        

\author{N.\ Georgiou         \and
        F.\ Rassoul-Agha \and
        T.\ Sepp\"al\"ainen 
}

\authorrunning{N.\ Georgiou, F.\ Rassoul-Agha, and T.\ Sepp\"al\"ainen} 

\institute{N.\ Georgiou \at
              School of Mathematical and Physical Sciences, Department of Mathematics, University of Sussex, Falmer Campus, Brighton BN1 9QH, UK \\
              \email{n.georgiou@sussex.ac.uk}           
           \and
           F.\ Rassoul-Agha \at
              Mathematics Department, University of Utah, 155 South 1400 East, Salt Lake City, UT 84109, USA \\
              \email{firas@math.utah.edu}
           \and
           T.\ Sepp\"al\"ainen \at
              Mathematics Department, University of Wisconsin-Madison, Van Vleck Hall, 480 Lincoln Dr., Madison, WI 53706, USA \\
              \email{seppalai@math.wisc.edu}
}

\date{Received: August 2015 / Accepted: July 2016}

\maketitle

\begin{abstract}
 We study  the directed last-passage percolation model on the planar square  lattice  with nearest-neighbor steps and general i.i.d.\ weights on the vertices,  outside of  the class of exactly solvable models.   Stationary cocycles are constructed for this percolation model from queueing fixed points.  These cocycles serve as boundary conditions for stationary last-passage percolation,  solve  variational formulas that characterize limit shapes,  and   yield existence of  Busemann functions in directions where the shape has some regularity. 
In a sequel to this paper the cocycles are used to prove results about semi-infinite geodesics and the competition interface.

\end{abstract}

\tableofcontents

\section{Introduction}

  We study nearest-neighbor  directed  last-passage percolation (LPP)  on the lattice $\Z^2$, also called the {\it corner growth model}.  Random i.i.d.\ weights $\{\w_x\}_{x\in\Z^2}$ are  used to define {\sl last-passage times}  $\Gpp_{x,y}$ between  lattice points $x\le y$ (coordinatewise ordering) in $\Z^2$ by 
 \be \Gpp_{x,y}=\max_{x_\centerdot}\sum_{k=0}^{n-1}\w_{x_k} 
	\label{Gxy1}\ee
where the maximum is over paths $x_\centerdot=\{x=x_0, x_1, \dotsc, x_n=y\}$  that satisfy  $x_{k+1}-x_k\in\{e_1,e_2\}$ (up-right paths whose increments are   standard basis vectors).  

 When $\w_x\ge 0$   this defines  a growth model in the first quadrant $\Z_+^2$.  Initially  the growing cluster    is   empty.   The origin joins the cluster at time $\w_0$.  After both  $x-e_1$ and $x-e_2$   have joined the cluster,  point $x$ waits time $\w_x$ to  join.  (However,  if  $x$ is on the boundary of $\Z_+^2$, only  one of $x-e_1$ and $x-e_2$ is required to  have joined.)  The cluster    at  time $t$  is   $\cC_t=\{x\in\Z_+^2:  \Gpp_{0,x}+ \w_x \le t\}$.       Our convention to exclude   $\w_{x_n}$ in \eqref{Gxy1} forces the  addition of $\w_x$ in the definition of $\cC_t$.  

The interest is in the large-scale behavior of the model.  This begins with  the deterministic limit  $\gpp(\xi)=\lim_{n\to\infty}  n^{-1}\Gpp_{0, \fl{n\xi}}$  for  $\xi\in\R_+^2$ under a moment assumption.     It is expected but not yet proved that   the shape function  $\gpp$ is differentiable.      Another  natural expectation is that  the increment process  $\{\Gpp_{0, v_n+x} -  \Gpp_{0, v_n} : x\in\Z^2\}$ converges in distribution as  $v_n\to\infty$ in a particular direction.  
Distributionally equivalent is to look at the limit of 
$ \Gpp_{0, v_n } - \Gpp_{x, v_n}$.   For this last  quantity  we can  expect even almost sure convergence.  The limit is called a  {\sl Busemann function}.    This is one  type of result developed in the paper.

Here are some particulars of what follows, in relation to past work.

In \cite{Geo-Ras-Sep-13a-} we derived variational formulas that characterize the limiting free energy and limit shape  for a general class of  polymer models. The results cover both positive temperature polymers and zero temperature percolation models, both point-to-point  and  point-to-line  models in all dimensions, and general admissible paths. 
Part of the theory of \cite{Geo-Ras-Sep-13a-}  is a solution  approach to one kind of  variational formula  in terms of stationary cocycles. 

 In the present paper we construct these cocycles that minimize in the variational formula for the  particular case of the planar corner growth model with general i.i.d.\ weights.  The weights are assumed  bounded from below and subject to a $2+\e$ moment bound.  The  construction of the cocycles comes from the fixed points of  the associated queueing operator whose existence   was proved by Mairesse and Prabhakar \cite{Mai-Pra-03}.      A Markov process analogy of the cocycle  construction  is a simultaneous construction  of a given interacting Markov process for  all invariant distributions, coupled through   common Poisson clocks that drive the dynamics.   The i.i.d.\ weights $\w$ are the analogue of the clocks and the cocycles are the analogues of the initial state variables in stationary distribution.

     The  construction  of the cocycles 
 does not require any regularity assumptions on $\gpp$.    But the cocycles are  indexed by the gradients $\nabla\gpp(\xi)$   as $\xi$ varies across directions in the first quadrant.    Consequently at a corner  $\xi$ of the shape  we get  two cocycles  that correspond to  $\nabla\gpp(\xi\pm)$.  We can prove the existence of the Busemann function  in directions $\xi$ where  $\nabla\gpp(\xi\pm)$ coincide and can be approximated by gradients from either side.   Such directions are either (i)  exposed points of differentiability  of $\gpp$ or  (ii)  lie on a linear segment of $\gpp$ whose endpoints are points of differentiability.

The companion paper  \cite{Geo-Ras-Sep-15b-}   uses the cocycles constructed here to prove results about the geodesics and the competition interface  of the corner growth model.  



  Under some moment and regularity assumptions on the weights, the corner growth  model   is expected to lie in the Kardar-Parisi-Zhang (KPZ)  universality class. (For a review of KPZ universality see \cite{Cor-12}.)   The fluctuations of $\Gpp_{0, \fl{n\xi}}$ are expected to have order of magnitude $n^{1/3}$ and  limit distributions from random matrix theory.     When the weights have exponential or geometric distribution the model is exactly solvable, and it is possible to derive 
   exact fluctuation exponents and limit distributions \cite{Bal-Cat-Sep-06,Joh-00,Joh-06}. 
  In these cases the  cocycles we construct  have explicit product form distributions.  
  The present paper can be seen as an attempt to begin development of techniques for studying the corner growth model beyond the exactly solvable cases.   As observed in  Remark \ref{rk:pc-wdo} in Section \ref{subsec:perc} below, the case of a percolation cone is exceptional  
  as it bears some markers of weak disorder and hence cannot obey   KPZ universality in all aspects.




Let us briefly touch upon   the last 20 years of research  on Busemann functions in percolation.  
The only    widely applied technique for proving the existence of Busemann functions  has come from the work of  Newman et al.\  \cite{How-New-01, Lic-New-96, New-95}.  The approach uses uniqueness and coalescence of directional geodesics and relies on a uniform curvature assumption on the limit shape.  This assumption cannot as yet be checked for general first or last-passage percolation models.  Still  there are many interesting processes  with  enough special  features so that  this approach can be carried out.  The examples are    exactly solvable directed lattice  percolation   and models in Euclidean space built on homogeneous Poisson processes.  

 W\"utrich \cite{Wut-02} gave an early application of the ideas of Newman et al.\ to construct Busemann functions for directed   last-passage  percolation on Poisson points.    Then came  the use of Busemann functions to study competition and coexistence  by Cator, Ferrari, Martin and Pimentel \cite{Cat-Pim-13,Fer-Pim-05, Fer-Mar-Pim-09, Pim-07},  existence, uniqueness and spatial mixing  of  invariant distributions of an interacting system by Cator and Pimentel \cite{Cat-Pim-12},  the  construction of solutions to a randomly forced Burgers equation by Bakhtin, Cator, and Khanin  \cite{Bak-Cat-Kha-14,Bak-15-}, 
 and   Pimentel's \cite{Pim-13-} bounds  on the speed of coalescence that improve those of  W\"utrich \cite{Wut-02}.   Distinct from this line  of work is Hoffman's use  of 
Busemann functions to study competition under more general weight distributions   \cite{Hof-05, Hof-08}.   


Our approach to Busemann functions (and to geodesics in the sequel  \cite{Geo-Ras-Sep-15b-})  is  the very opposite.     The limits are constructed a priori in the form of the   cocycles.  The cocycles define stationary   percolation models that can be  coupled with the original one.    The coupling,   ergodicity,  and local regularity of the limit shape  give the control that proves the Busemann limits.   

 The  use of Busemann functions or stationary cocycles  to create auxiliary  stationary percolation processes  has been very fruitful.    Hoffman applied this idea to study geodesics and coexistence  \cite{Hof-05}.  The stationary process became a tool for proving fluctuation exponents in the seminal work of Cator and Groeneboom \cite{Cat-Gro-06} on the planar Poisson last-passage percolation model.  The idea was adapted to the lattice case by Bal\'azs, Cator and Sepp\"al\"ainen \cite{Bal-Cat-Sep-06}, whose estimates have then been used by other works.  The stationary process of Busemann functions    
  was also profitably utilized in    \cite{Cat-Pim-12}.    In an earlier version of this technique the stationary process coupled to the  percolation model was an associated particle system \cite{Sep-98-mprf-2, Sep-98-mprf-1}.  
  
  The idea of deducing existence and uniqueness of stationary processes by studying   geodesic-like objects    has also been used in random dynamical systems.
For example, \cite{E-etal-00} and its extensions \cite{Hoa-Kha-03,Itu-Kha-03,Bak-07,Bak-Kha-10,Bak-13} take this approach to produce invariant measures for the Burgers equation with random forcing. 
This line of work has treated situations where the space is compact or essentially compact.  To make progress in the non-compact case,  the approach of Newman et al.\ was adopted again in \cite{Bak-Cat-Kha-14,Bak-15-}, as mentioned above.

The seminal results of   Newman et al.\  on geodesics  have recently been extended by Damron and Hanson   \cite{Dam-Han-14} by taking as starting point a   weak subsequential limit of Busemann functions.     These  weak Busemann  limits of  \cite{Dam-Han-14}  can  be regarded as a   counterpart of our stationary cocycles. 

To generalize our results beyond i.i.d.\ weights and potentially to higher dimensions, a  possible strategy that avoids the reliance on queueing theory would be to develop 
 sufficient   control on the gradients 
$\Gpp_{x,\fl{n\xi}}-\Gpp_{y,\fl{n\xi}}$ (or their point-to-line counterparts)  to construct cocycles through weak limits as $n\to\infty$.  This  worked well for undirected first-passage percolation in   \cite{Dam-Han-14} because  the gradients are uniformly integrable.   Note however that when $\{\w_x\}$ are only   ergodic, the limiting 
shape can  have corners and linear segments, and can even be a finite polygon.

 \smallskip

{\bf Organization of the paper.}  
Section \ref{sec:results} introduces the corner growth model.  
Section \ref{sec:bus-thm} states the existence theorem for Busemann functions   under   local  regularity   assumptions on  the limit function  $\gpp(\xi)$.   The remainder of the paper does not rely on properties of $\gpp$ beyond those known and stated in   Section \ref{sec:results}.  
Section \ref{sec:duality} develops a convex duality between directions or velocities $\xi$ and tilts or external fields $h$ that comes from  the relationship of the point-to-point and point-to-line percolation models.  
Section \ref{sec:cocycles}  states the existence and properties of the cocycles on which all the results of the paper are based.  The proof comes at the end of the paper in  Section  \ref{app:q}.   The cocycles are used to construct   a stationary last-passage process  (Section \ref{sec:stat-lpp})  and they  yield  minimizers for the variational formulas that characterize the limit shapes (Section \ref{subsec:varsol}). 
Section \ref{sec:busemann} proves the existence of  Busemann functions.

 Two particular cases of the corner growth model are addressed separately: the exactly solvable cases   with geometric or  exponential   weights $\{\w_x\}$ (Sections \ref{subsec:solv} and \ref{sec:solv}), and the percolation cone  (Section \ref{subsec:perc}). In the case of the percolation cone we can give an alternative  formula for the limiting   Busemann function and we observe that the centered Busemann function is a gradient.  This is consistent with the percolation cone situation being in weak disorder, as also indicated by the vanishing of the fluctuation exponent for the point-to-line last-passage time.  
 
 We have organized the paper so that the results on Busemann functions lead and the technicalities of cocycle construction are delayed towards the end.   The reader who wants to see each step fully proved before progressing further can go through  the material in the following order.   (i)  Basic definitions from  Sections \ref{sec:results} and  \ref{sec:duality}.   (ii)    The cocycle construction:   Theorem \ref{th:cocycles}  from Section  \ref{sec:exist} and its proof from Section  \ref{app:q}.   (iii)  The stationary LPP from Section  \ref{sec:stat-lpp}, followed by Theorem \ref{th:construction} for Busemann functions  and its proof in Section \ref{sec:busemann}.   (iv)  For minimizers of the variational formulas,  Theorem \ref{th:var-sol} and its proof from Section \ref{subsec:varsol}.   (v) The rest is specialization under  assumptions on the weights or the limit shape.


A short 
Appendix \ref{app:aux} states an ergodic theorem for cocycles proved in \cite{Geo-etal-15-}. 

\smallskip

{\bf Notation and conventions.}   $\R_+=[0,\infty)$,  $\Z_+=\{0,1,2,3, \dotsc\}$,  $\N=\{1,2,3,\dotsc\}$. The standard basis vectors of $\R^2$ are  $e_1=(1,0)$ and $e_2=(0,1)$ and  the $\ell^1$-norm of $x\in\R^2$  is   $\abs{x}_1=\abs{x\cdot e_1} + \abs{x\cdot e_2}$.   For $u,v\in\R^2$ a  closed line segment on $\R^2$ is denoted by   $[u,v]=\{tu+(1-t)v:  t\in[0,1]\}$, and an open line segment by  $]u,v[=\{tu+(1-t)v:  t\in(0,1)\}$.    Coordinatewise ordering $x\le y$ means that $x\cdot e_i\le y\cdot e_i$ for both $i=1$ and $2$.  Its negation $x\not\le y$ means that   $x\cdot e_1> y\cdot e_1$ or   $x\cdot e_2> y\cdot e_2$.    An admissible or up-right path $x_{0,n}=(x_k)_{k=0}^n$  on $\Z^2$ satisfies $x_k-x_{k-1}\in\{e_1,e_2\}$.  

The  environment space is $\Omega=\R^{\Z^2}$ whose elements are denoted by $\w$.  Elements of a larger product space $\Ombig=\Omega\times\Omega'$   are denoted by $\what=(\w, \w')$.   
Parameter  $\pp>2$ appears in a  moment hypothesis  $\E[|\w_0|^{\pp}]<\infty$,  while $p_1$ is the  probability of an open site in an oriented site percolation process.   $X\sim\mu$ means that random variable $X$ has distribution $\mu$. 
 
$\Diff$ is the set of points of differentiability of $\gpp$ and $\EP$ the set of exposed points of differentiability.   
A statement that  contains $\pm$ or $\mp$ is a combination of two statements: one for the top choice of the sign and another one for the bottom choice.  


\section{Preliminaries on the corner growth model} 
\label{sec:results}

This section presents assumptions,  definitions, and notation used throughout the paper, and states some known results.  

\subsection{Assumptions} 
The two-dimensional  corner growth model is the last-passage percolation model on the planar square lattice $\Z^2$ with admissible steps   $\range=\{e_1,e_2\}$.  
    $\Omega=\R^{\Z^2}$ is the   space of   environments  or weight configurations $\w=(\w_x)_{x\in\Z^2}$.   The group of spatial translations   $\{T_x\}_{x\in\Z^2}$  acts on $\Omega$ by $(T_x\w)_y=\w_{x+y}$ for $x,y\in\Z^2$.   Let $\kS$ denote the Borel $\sigma$-algebra of $\Omega$.  $\P$ is  a Borel probability measure    on $\Omega$  under which the weights  $\{\w_x\}$  are  independent, identically distributed (i.i.d.) nondegenerate random variables with a $2+\e$ moment.      Expectation under $\P$ is denoted by $\E$.   For a technical reason we also assume $\P(\w_0\ge c)=1$ for some finite constant $c$.
    
For future reference we summarize our standing assumptions in this statement: 
	\be\begin{aligned} \label{2d-ass}
		&\text{$\P$ is i.i.d., \, $\E[|\w_0|^{\pp}]<\infty$ for some $\pp>2$,\,     $\sigma^2=\Vvv(\w_0)>0$, and }\\
		 &\text{$\P(\w_0\ge c)=1$ for some   $c>-\infty$.} 
	\end{aligned} 
	\ee  
Assumption \eqref{2d-ass} is valid throughout the paper and will not be repeated in  every statement.  	The constant 
\[\Ew=\E(\w_0)\]  will appear  frequently.  	The symbol   $\w$ is reserved for   these $\P$-distributed  i.i.d.\ weights, also later when they are embedded in a larger configuration $\what=(\w, \w')$.  

  Assumption $\P(\w_0\ge c)=1$ is used explicitly   only   in Section  \ref{app:q} where we rely on  queueing theory.   In that context $\w_x$ is a service time, and the results we quote  have been proved only for $\w_x\ge 0$.  (The extension to $\w_x\ge c$  is immediate.)   The key point   is that if the queueing results are extended to general real-valued  i.i.d.\ weights $\w_x$ subject to the moment assumption in \eqref{2d-ass},  everything in this paper is true  for these general real-valued weights.  
  
  
  \subsection{Last-passage percolation} 
Given an environment $\w$ and two points $x,y\in\Z^2$ with $x\le y$ coordinatewise,
define the {\sl point-to-point last-passage time} by
	\[ \Gpp_{x,y}=\max_{x_{0,n}}\sum_{k=0}^{n-1}\w_{x_k}.\]
The maximum is over paths $x_{0,n}=(x_k)_{k=0}^n$  that start at  $x_0=x$,  end at $x_n=y$ with $n=\abs{y-x}_1$,  and have increments $x_{k+1}-x_k\in\{e_1,e_2\}$.  Call such paths {\sl admissible} or {\sl up-right}.

Given a vector $h\in\R^2$, an environment $\w$, and an integer $n\ge0$, define the  $n$-step  {\sl point-to-line last-passage time} with {\sl tilt} (or {\sl external field}) $h$ by
 	\be\label{Gnh} \Gpl_n(h)=\max_{x_{0,n}}\Bigl\{\,\sum_{k=0}^{n-1}\w_{x_k}+h\cdot x_n\Bigr\}. \ee 
The maximum is over all admissible $n$-step   paths that start at $x_0=0$.

 It is standard (see for example \cite{Mar-04} or \cite{Ras-Sep-14})   that under assumption \eqref{2d-ass},    for $\P$-almost every $\w$, simultaneously  for every $\xi\in\R_+^2$ and every $h\in\R^2$, 
the following  limits exist:
		\begin{align}
		 &\gpp(\xi)=\lim_{n\to\infty}n^{-1}\Gpp_{0,\lfloor{n\xi}\rfloor},\label{eq:g:p2p}\\
		 &\gpl(h)=\lim_{n\to\infty}n^{-1}\Gpl_n(h).\label{eq:g:p2l}
		\end{align}
In the definition above  integer parts are taken coordinatewise: $\fl v=(\fl{a},\fl{b}) \in\Z^2$ for $v=(a,b) \in\R^2$.   A stronger result is also true: the shape theorem gives a uniform limit (Theorem 5.1(i) of \cite{Mar-04}):  
\begin{align}\label{lln5}
\lim_{n\to\infty} n^{-1}\max_{x\in\Z_2^+:\abs{x}_1=n}\abs{\Gpp_{0,x}-\gpp(x)}=0\quad\text{$\P$-almost surely}.
\end{align}

 Under assumption \eqref{2d-ass} $\gpp$  and $\gpl$  are finite nonrandom continuous functions.  In particular, $\gpp$  is continuous up to the boundary of $\R_+^2$.   Furthermore,   $\gpp$ is a symmetric, concave,  $1$-homogeneous function on $\R_+^2$  and $\gpl$ is a convex Lipschitz function on $\R^2$. Homogeneity means that    $\gpp(c\xi)=c\gpp(\xi)$ for $\xi\in\R_+^2$ and $c\in\R_+$.    By homogeneity, for most purposes it suffices to consider $\gpp$ as a function  on  the convex hull    $\Uset=[e_1,e_2]=\{te_1+(1-t)e_2: 0\le t\le1\}$ of $\range$.    The relative interior $\ri\Uset$ is the open line segment $]e_1,e_2[=\{te_1+(1-t)e_2:  0<t<1\}$.  


Decomposing according to the endpoint of the path and some estimation (Theorem 2.2 in \cite{Ras-Sep-14}) give
 	\be\label{h-xi} \gpl(h)=\sup_{\xi\in\Uset}\{\gpp(\xi)+h\cdot\xi\}.\ee
By convex duality   for $\xi\in\ri\Uset$
	\[\gpp(\xi)=\inf_{h\in\R^2}\{\gpl(h)-h\cdot\xi\}.\] 
Let us  say $\xi\in\ri\Uset$ and $h\in\R^2$ are {\sl dual} if  
	\begin{align}\label{eq:duality} \gpp(\xi)=\gpl(h)-h\cdot\xi.\end{align}

Very little  is known in general about $\gpp$ beyond the soft properties mentioned above.  In the    exactly solvable case, with $\w_x$ either exponential or geometric,    
we have $\gpp(s,t)=(s+t)\Ew+2\sigma\sqrt{st}$. 
The Durrett-Liggett flat edge result  (\cite{Dur-Lig-81}, Theorem \ref{th:flat-edge} below) tells us that  this formula is   not true for all i.i.d.\ weights.   
It does  hold for general weights asymptotically at the boundary \cite{Mar-04}:  $\gpp(1,t)=\Ew+2\sigma\sqrt{t} + o(\sqrt t\,)$
as $t\searrow 0$.  

\subsection{Gradients and convexity} 
Regularity properties of $\gpp$ play a role in  our results,  so we   introduce notation for that purpose.   
Let 
\[  \Diff=\{\xi\in\ri\Uset:\text{ }\gpp\text{ is differentiable at $\xi$}\}.\]
To be   clear,  $\xi\in\Diff$ means that the gradient $\nabla\gpp(\xi)$
 exists in the usual sense of differentiability of functions of several variables.  
 At $\xi\in\ri\Uset$ this is equivalent to the differentiability of the single variable 
 function $r\mapsto \gpp(r,1-r)$ at $r={\xi\cdot e_1}/{\abs{\xi}_1}$.  
 By concavity the set $(\ri\Uset)\smallsetminus\Diff$ is at most countable.
 
 A point $\xi\in\ri\Uset$ is an {\sl exposed point} if 
 	\be \label{eq:epod}	\begin{aligned}
\exists v\in\R^2 \; \text{ such that } \;	   \forall \zeta\in\ri\Uset\smallsetminus \{\xi\}: \; \gpp(\zeta)\; <\; \gpp(\xi) + v\cdot(\zeta - \xi) .   \end{aligned}\ee
The set  of {\sl exposed points of differentiability} of $\gpp$ is  
$\EP=  \{ \xi\in\Diff:  \text{\eqref{eq:epod} holds}\}$.    For $\xi\in\EP$, \eqref{eq:epod} is uniquely satisfied by $v=\nabla\gpp(\xi)$.  
The condition for an exposed point is formulated entirely in terms of $\Uset$ because  $\gpp$  is 
a homogeneous function and therefore cannot have exposed points as a function on $\R_+^2$. 	
		
It is expected that $\gpp$ is differentiable on all of $\ri\Uset$.  
But since this is not known, our development must handle 
possible points of nondifferentiability.   For this purpose we take left and right
limits on $\Uset$. Our convention is that a {\sl left limit} $\xi\to\zeta$ on $\Uset$ means that 
$\xi\cdot e_1$ increases to $\zeta\cdot e_1$, while in a {\sl right limit}   
$\xi\cdot e_1$ decreases to $\zeta\cdot e_1$, with $\xi\ne\zeta$.  
  
For $\zeta\in\ri\Uset$ define one-sided gradient vectors $\nabla\gpp(\zeta\pm)$ by
\begin{align*}
\nabla\gpp(\zeta\pm)\cdot e_1&=\lim_{\e\searrow0}\frac{\gpp(\zeta\pm\e e_1)-\gpp(\zeta)}{\pm\e} \\  \text{and}\quad 
\nabla\gpp(\zeta\pm)\cdot e_2&=\lim_{\e\searrow0}\frac{\gpp(\zeta\mp\e e_2)-\gpp(\zeta)}{\mp\e}.
\end{align*}
Concavity of $\gpp$ ensures  the  limits   exist.   $\nabla\gpp(\xi\pm)$ coincide (and equal $\nabla\gpp(\xi)$) 
 if and only if $\xi\in\Diff$.  
Furthermore,  on $\ri\Uset$, 
\begin{align}\label{nabla-g-lim}  
\nabla\gpp(\zeta-)=\lim_{\xi\cdot e_1\nearrow\zeta\cdot e_1} \nabla\gpp(\xi\pm)
\quad\text{and}\quad 
\nabla\gpp(\zeta+)=\lim_{\xi\cdot e_1\searrow\zeta\cdot e_1} \nabla\gpp(\xi\pm). 
\end{align}
  
For $\xi\in\ri\Uset$ define maximal  line segments on which $\gpp$ is linear,  $\Uset_{\xi-}$ for the left gradient at $\xi$ 
and $\Uset_{\xi+}$ for the right gradient at $\xi$,  
by 
	\begin{align}\label{eq:sector1}
\Uset_{\xi\pm}=\{\zeta\in\ri\Uset: 	\gpp(\zeta)-\gpp(\xi)=\nabla g(\xi\pm)\cdot(\zeta-\xi)\}.  
	\end{align}
Either or both segments can degenerate to a point. 
Let \be\label{eq:sector2} \Uset_\xi=\Uset_{\xi-}\cup\,\Uset_{\xi+}=[\ximin, \ximax]
\qquad\text{with $\ximin\cdot e_1\le \ximax\cdot e_1$.}
\ee
  If $\xi\in\Diff$ then $\Uset_{\xi+}=\Uset_{\xi-}=\Uset_\xi$, while if $\xi\notin\Diff$ then $\Uset_{\xi+}\cap\Uset_{\xi-}=\{\xi\}$.  If $\xi\in\EP$ then $\Uset_{\xi}= \{\xi\}$.     Figure \ref{fig:defs} illustrates.
  
For $\zeta\cdot e_1<\eta\cdot e_1$ in $\ri\Uset$,  $[\zeta, \eta]$ is a     
  {\sl maximal linear segment} of  $\gpp$   if $\nabla\gpp$ exists and  is constant in $]\zeta, \eta[$  but not on any strictly larger open line segment in $\ri\Uset$.  
Then    $[\zeta, \eta]=\Uset_{\zeta+}=\Uset_{\eta-}=\Uset_\xi$ for any $\xi\in\;]\zeta, \eta[$.   If $\zeta, \eta\in\Diff$ we say that $\gpp$ is differentiable at the endpoints of this maximal linear segment.  This  hypothesis will be invoked several times.

  \begin{figure}[h]	
	\begin{center}
		\begin{tikzpicture}[>=latex,scale=0.74]

	\draw[-](0,0)--(11,0)node[below]{$\Uset$};
		
	\draw[ line width = 1.3 pt, color=nicosred] (4, 7)node[right,color=black]{$ $}--(6,5.5)--(8,4)--(10,2.5);
		
\draw[ line width = 1.2 pt, color=nicosred] (3, 5.5)--(5,5.5)--(7,5.5) node[color=black, below right]{$ $}--(10,5.5);

		\draw[line width = 1.3pt] (0,.9)--(0,1) .. controls (0.2,1.5)..(2,3)..controls (5,5.5)..(6,5.5) -- (8,4)..controls (9.5,2)..(10,.9); 
		
		\draw[dashed, line width = 1.1pt] (6,5.5)--(6,0)node[below]{$\xi = \ximin$};
		\draw[dashed, line width = 1.1pt] (8,4)--(8,0)node[below]{$\ximax$};
		
		\draw[line width=3pt, color= nicosred] (6,0)--(8,0);
		
		\shade[ball color=red](8,4)circle(1.5mm);
		\shade[ball color=red](6,5.5)circle(1.5mm); 
		
		\draw [decorate,decoration={brace, amplitude = 0.3cm},xshift=0pt,yshift=5, line width=1.2pt]
						(6.1,0) --(7,0)-- (7.9,0);
		\draw(7,.44)node[above]{$\Uset_{\xi+}$};
		
				
		\draw[dashed, line width = 1.1pt] (.5,1.7)--(.5,0)node[below]{$ \zetamin$};
		\draw[dashed, line width = 1.1pt] (4.6,5.1)--(4.6,0)node[below]{$\zetamax$};
		\draw[dashed, line width = 1.1pt] (2.4,3.3)--(2.4,1);
		\draw[dashed, line width = 1.1pt] (2.4,.3)--(2.4,0)node[below]{$\zeta$};
		
		\shade[ball color=red](.5,1.7)circle(1.5mm);
		\shade[ball color=red](4.6,5.1 )circle(1.5mm); 
		\shade[ball color=red](2.4, 3.3 )circle(1.5mm); 
		
		\draw [decorate,decoration={brace, amplitude = 0.3cm},xshift=0pt,yshift=5, line width=1.2pt]
						(.5,0) -- (4.6,0);
		\draw(2.5,.44)node[above]{$\Uset_{\zeta+} = \Uset_{\zeta} = \Uset_{\zeta-}  $};
		
		\draw[line width=3pt, color= nicosred] (.5,0)--(4.6,0);	
		
		\end{tikzpicture}
	\end{center}
	\caption{\small    A graph of a concave function over $\Uset$ to illustrate the definitions.  
	   $\zetamin$, $\zeta$ and $\zetamax$ are points of differentiability while $\ximin = \xi$ and  $\ximax$ are not.   $\Uset_{\zetamin}=\Uset_\zeta=\Uset_\zetamax=[\zetamin,\zetamax]$.
	    The red lines represent supporting hyperplanes at $\xi$.   The slope from the left at $\xi$ is zero, and the  horizontal red line touches the graph only at $\xi$.  Hence  $ \Uset_{\xi-} = \{\xi\}$.  	Points on  the line segments $[\zetamin, \zetamax]$ and $]\xi, \ximax[$ are   not exposed. 
  $\cE=\ri\Uset \smallsetminus \bigl([\zetamin, \zetamax] \cup[\xi, \ximax]\bigr).$}
  \label{fig:defs}
\end{figure}



\subsection{Cocycles and variational formulas} 
The next definition is central to the paper.  

\begin{definition}[Cocycle]
\label{def:cK}
	A measurable function $B:\Omega\times\Z^2\times\Z^2\to\R$ is   a {\rm stationary $L^1(\P)$  cocycle}
	if it satisfies the following three conditions. 
		\begin{enumerate}[\ \ \rm(a)]
			\item\label{def:cK:int} Integrability: for each $z\in\{e_1,e_2\}$, $\E\abs{B(0,z)} <\infty$.
			\item\label{def:cK:stat} Stationarity:  for $\P$-a.e.\ $\w$ and  all $x,y,z\in\Z^2$, 
					$B(\w, z+x,z+y)=B(T_z\w, x,y)$.
			\item\label{def:cK:coc} Additivity: for $\P$-a.e.\ $\w$ and  all $x,y,z\in\Z^2$, $B(\w, x,y)+B(\w, y,z)=B(\w, x,z)$.
		\end{enumerate}
	The space  of stationary $L^1(\P)$  cocycles on $\OBP$  is denoted by $\Coc(\Omega)$.  
	The subspace   $\Coc_0(\Omega)$ of {\rm centered} stationary $L^1(\P)$  cocycles   consists of $F\in\Coc(\Omega)$ such that 
  $\E[F(x,y)]=0$ for all $x,y\in\Z^2$.
\end{definition}
$\Coc_0(\Omega)$ is the $L^1(\P)$ closure of gradients 
$F(\w,x,y)=\varphi(T_y\w)-\varphi(T_x\w)$, $\varphi\in L^1(\P)$ (see 
\cite[Lemma~C.3]{Ras-Sep-Yil-13}).     Our convention for  centering a 
  stationary $L^1$ cocycle $B$  is to let   
 $h(\B)\in\R^2$ denote the vector  that   satisfies 
 	\be 	
		\E[\B(0,e_i)]=-h(\B)\cdot e_i \qquad\text{for  $i\in\{1,2\}$ }  \label{eq:EB}
	\ee
and then define $F\in\Coc_0(\Omega)$ by 
	\be 
		F(\w,x,y)=  h(\B)\cdot (x-y)-\B(\w, x,y).  \label{FF}
	\ee
 
 Cocycles appear as minimizers  in variational formulas that describe the limits of last-passage percolation models.    
 In Theorems 3.2 and 4.3 in \cite{Geo-Ras-Sep-13a-} we proved these variational formulas:  for $h\in\R^2$ 
	\begin{align}
		\gpl(h)&=\inf_{F\in\Coc_0(\Omega)}\; \P\text{-}\esssup_\w\;  \max_{i\in\{1,2\}} \{\w_0+h\cdot e_i+F(\w, 0,e_i)\}\label{eq:g:K-var}\\
		\intertext{and for $\xi\in\ri\Uset$}  
		\gpp(\xi)&=\inf_{B\in\Coc(\Omega)} \;\P\text{-}\esssup_\w\;  \max_{i\in\{1,2\}} \{\w_0 -B(\w, 0,e_i)-h(B)\cdot\xi\}.\label{eq:gpp:K-var}
	\end{align}


\medskip

\section{Results on Busemann functions} \label{sec:bus-thm} 
This section utilizes specialized assumptions either on the shape function $\gpp$ or on the weights, in addition to the basic assumption \eqref{2d-ass}.  
We  state the theorem on the existence of Busemann functions, both point-to-point and point-to-line.   This theorem is proved in Section \ref{sec:busemann}.   After the theorem we discuss two examples: the exactly solvable case and the percolation cone.   

\begin{theorem}\label{thm:buse}  
 Fix two points $\zeta, \eta\in\Diff$ such that $\zeta\cdot e_1\le\eta\cdot e_1$.  Assume that  either 
 
 {\rm(i)}  $\zeta=\eta=\xi\in\EP$ in which case  $\zeta=\eta=\xi=\ximin= \ximax$,   or that 
 
 {\rm(ii)}  $[\zeta, \eta]$ is a maximal linear segment of $\gpp$ in which case  
 $[\zeta, \eta]=[\ximin, \ximax]$ for all $\xi\in[\zeta, \eta]$.  
 
Then there exists a stationary $L^1(\P)$ cocycle $\{B(x,y):x,y\in\Z^2\}$ and an event  $\Omega_0$ with $\P(\Omega_0)=1$ such that  the following holds for each 
$\w\in\Omega_0$.    For each sequence $v_n\in\Z_+^2$ such that 
\begin{align}\label{eq:vn}
\abs{v_n}_1\to\infty\quad\text{and}\quad\zeta\cdot e_1\le\varliminf_{n\to\infty}\frac{v_n\cdot e_1}{\abs{v_n}_1}\le\varlimsup_{n\to\infty}\frac{v_n\cdot e_1}{\abs{v_n}_1}\le\eta\cdot e_1,
\end{align}
we have the limit
\be \label{eq:grad:coc1} 
							B(\w, x,y) = \lim_{n\to \infty} \big( \Gpp_{x, v_n}(\w) - \Gpp_{y, v_n}(\w) \big)   \qquad  \text{ for  $x,y\in\Z^2$. } 
						\ee 
 Furthermore,   if  $h=t(e_1+e_2)-\nabla\gpp(\xi)$  for some  $t\in\R$ and any {\rm(}and hence all{\rm)} $\xi\in[\zeta, \eta]$,  we have the limit 
 \be \label{p2lB-1} 
	B(\w, 0, z) +h\cdot z= \lim_{n\to \infty} \big( \Gpl_{n}(h)-\Gpl_{n-1}(h)\circ T_z \big)  
	\qquad  \text{for $z\in\{e_1,e_2\}$.}  
						\ee 
						
 The mean of the limit is given by 
\be\label{EB=Dg}     \nabla \gpp(\xi) = \bigl( \, \E[B(x,x+e_1)]\,,\,   \E[B(x,x+e_2)] \,\bigr)\quad\text{for all $\xi\in[\zeta, \eta]$}.
\ee 
\end{theorem}

To paraphrase the theorem,   suppose  $\xi$ is an exposed point of differentiability of $\gpp$,  or $\xi$ lies on a maximal linear segment of $\gpp$ whose endpoints are points of differentiability.  Then a Busemann function $B^\xi$ exists in direction $\xi$ in the sense that   $B^\xi(\w,x,y)$ equals  the a.s.\ limit  in \eqref{eq:grad:coc1}   for any sequence $v_n/\abs{v_n}_1\to\xi$ with $\abs{v_n}_1\to\infty$.    Furthermore, the $B^\xi$'s match for points $\xi$ on  maximal linear segments of $\gpp$ with endpoints in $\Diff$.  
The condition on $h$ in the theorem is exactly that $h$ and $\xi$ are dual in the sense of \eqref{eq:duality}.  
 


   We shall not derive the  cocycle properties of $B$  from the limit \eqref{eq:grad:coc1}.  Instead   we construct a family of cocycles on an extended   space $\Ombig=\Omega\times\Omega'$ and show that  one  of these   cocycles equals the limit on the right of \eqref{eq:grad:coc1}.

The Busemann limits \eqref{eq:grad:coc1}  can also be interpreted as convergence of the last-passage process  to a  stationary last-passage process, described in Section \ref{sec:stat-lpp}.

	Equation \eqref{EB=Dg}  
  was anticipated in \cite{How-New-01} (see paragraph after the proof of Theorem 1.13)   for  Euclidean first passage percolation (FPP) where $\gpp(x,y)=c\sqrt{x^2+y^2}$. 
    	A version of this formula appears also in Theorem 3.5 of \cite{Dam-Han-14} for lattice  FPP.

The next theorem states that the Busemann functions found in Theorem \ref{thm:buse} give minimizing cocycles.  	
	
	\begin{theorem}\label{thm:var-buse}  Let  $\xi\in\ri\Uset$ with  $\Uset_\xi =[\ximin, \ximax]$   defined in \eqref{eq:sector2}.
 Assume that $\ximin, \xi, \ximax\in\Diff$.  Let $B^\xi\in\Coc(\Omega)$ be the limit in  \eqref{eq:grad:coc1} for any sequence $v_n$ that satisfies \eqref{eq:vn} for $\zeta=\ximin$ and $\eta= \ximax$.  We have  $h(B^\xi)=-\nabla\gpp(\xi)$ by \eqref{EB=Dg} and \eqref{eq:EB}.  Define   $F(x,y)=h(B^\xi)\cdot(x-y)-B^\xi(x,y)$ as in  \eqref{FF}.  
 
 \begin{enumerate}[\ \ \rm(i)]
 \item\label{thm:var-buse:i}   Let $h=h(B^\xi)+(t,t)$ for some $t\in\R$.  Then for $\P$-a.e.~$\w$ 
\be\label{var15}  \gpl(h)=  \max_{i\in\{1,2\}} \{\w_0+h\cdot e_i+F(\w, 0,e_i)\} = t. \ee
In other words,  $F$ is a minimizer in \eqref{eq:g:K-var} and the essential supremum vanishes.  
  
 \item\label{thm:var-buse:ii}   For $\P$-a.e.~$\w$ 
\be\label{var151}\gpp(\xi)=  \max_{i\in\{1,2\}} \{\w_0 -B^\xi(\w, 0,e_i)-h(B^\xi)\cdot\xi\} = -h(B^\xi)\cdot \xi.  
\ee
In other words,  $B^\xi$ is a minimizer in \eqref{eq:gpp:K-var} and the essential supremum vanishes.  
\end{enumerate}
\end{theorem}

The condition $h=h(B^\xi)+(t,t)$ for some $t\in\R$ is equivalent to $h$ dual to $\xi$. Every $h$ has a dual $\xi\in\ri\Uset$ as we show in Section \ref{sec:duality}.  Consequently, if $\gpp$ is differentiable everywhere on $\ri\Uset$,  each $h$ has a minimizing  cocycle $F$ that satisfies \eqref{var15} and is obtained by centering a Busemann function.    Theorem \ref{thm:var-buse}  is proved in Section \ref{sec:busemann}.  

The choice of $i\in\{1,2\}$ in \eqref{var15} and \eqref{var151}  depends on $\w$ and is related to the notion of    competition interface.  This issue is addressed in the  companion paper  \cite{Geo-Ras-Sep-15b-}.  

Borrowing from  homogenization literature (see e.g.\ page 468 of \cite{Arm-Sou-12}), 
a minimizer of \eqref{eq:g:K-var}  
that  removes the essential supremum, that is,  a mean zero cocycle that satisfies \eqref{var15},  could be called a {\sl corrector}.


\subsection{Exactly solvable models}\label{subsec:solv}
We illustrate the  results in the two exactly solvable cases: the distribution of the  mean $\Ew$ weights $\w_x$ is 
\be\label{cases7} \begin{aligned}
\text{exponential:  } \P\{\w_x\ge t\}&=  e^{-t/\Ew} \text{ for  $t\ge0$ with $\sigma^2=\Ew^2$,  }\\
\text{or geometric: } \P\{\w_x\ge k\}&=(1-\Ew^{-1})^{k} \text{ for  $k\in\N$ with  $\sigma^2=\Ew(\Ew-1)$.}
\end{aligned}\ee
Calculations behind the claims below are sketched in Section \ref{sec:solv} after the connection with queueing theory has been established.

For both cases in \eqref{cases7}  the point-to-point limit   function is 
 \begin{align}\label{eg:gpp}
 \gpp(\xi)=\Ew(\xi\cdot e_1+\xi\cdot e_2)+2\sigma\sqrt{(\xi\cdot e_1)(\xi\cdot e_2)}\,. 
 \end{align}
 In the exponential case this formula was first derived by Rost \cite{Ros-81} (who presented the model in its coupling with TASEP without the last-passage formulation)  while early derivations of the geometric case appeared in  \cite{Coh-Elk-Pro-96, Joc-Pro-Sho-98, Sep-98-mprf-1}.  
  Convex duality \eqref{eq:duality} becomes 
\begin{align*}
&\xi\in\ri\Uset\text{ is dual to $h$ if and only if }\\
&\qquad  \exists t\in\R: \quad  h=\bigl(\Ew+\sigma\sqrt{{\xi\cdot e_1}/{\xi\cdot e_2}}+t \,,\, \Ew+\sigma\sqrt{{\xi\cdot e_2}/{\xi\cdot e_1}}+t\,\bigr).  
\end{align*}
This in turn gives an explicit formula for $\gpl(h)$.

Since $\gpp$ above  is  differentiable and strictly concave, all points of $\ri\Uset$ are exposed points of differentiability.   
 Theorem \ref{thm:buse} implies that Busemann functions \eqref{eq:grad:coc1} exist in all directions  $\xi\in\ri\Uset$. 
They  minimize formulas \eqref{eq:g:K-var} and \eqref{eq:gpp:K-var}
as given in \eqref{var15} and \eqref{var151}. 
For each $\xi\in\ri\Uset$ the processes $\{ B^\xi(ke_1, (k+1)e_1): k\in\Z_+\}  $
and $\{ B^\xi(ke_2, (k+1)e_2): k\in\Z_+\}  $ are i.i.d.\ processes  independent of each other, exponential or geometric depending on the case,  with means 
\be \label{geom:B}  
\E[B^\xi(ke_i, (k+1)e_i)]=\Ew+\sigma  \sqrt{\xi\cdot e_{3-i}/\xi\cdot e_i}, \qquad i\in\{1,2\}. 
 \ee
 For the distribution of  $B^\xi$ see  Theorem 8.1 in \cite{Cat-Pim-12},  Section 3.3 in \cite{Cat-Pim-13},  and Section \ref{sec:solv} below.

\subsection{Flat edge in the percolation cone}\label{subsec:perc}

In this section we assume that the LPP weights satisfy $\w_x\le 1$ and $p_1=\P\{\w_0=1\}>0$. The classic Durrett-Liggett flat edge result, sharpened by Marchand,  implies that if $p_1$ is large enough,  $\gpp$ is linear on the percolation cone.  By the more recent work of Auffinger-Damron, $\gpp$  is  differentiable on the edges.    We make a precise statement about this below, after  a  short detour into oriented percolation.  

In {\sl oriented site percolation} vertices of $\Z^2$ are assigned i.i.d.\ $\{0,1\}$-valued random variables $\{\sigma_z\}_{z\in\Z^2}$ with $p_1=\P\{\sigma_0=1\}$.
For  points  $u\le v$ in $\Z^2$  we write   $u\to v$ (there is an open path from $u$   to $v$) if there exists an   up-right path $u=x_0, x_1,\dotsc, x_m=v$ with $x_{i}-x_{i-1}\in\{e_1,e_2\}$, $m=\abs{v-u}_1$, and such that $\sigma_{x_i}=1$ for $i=1,\dotsc,m$.  (The openness of a path does not depend on the weight at the initial point of the path.)  
  The {\sl percolation}  event  $\{u\to\infty\}$ is the existence of  an infinite open up-right path from point $u$.  
There exists a critical threshold $\pc \in (0,1)$ such that if $p_1 < \pc$ then $\P\{ 0\to{\infty}\} = 0$ and if $p_1 > \pc$
then $\P\{ 0\to\infty \} > 0$.   (The facts we need about  oriented site percolation  are proved in article  \cite{Dur-84}  for  oriented edge percolation.  The proofs apply to site percolation just as well.)    

   Let  $\cla_n = \{ u \in \Z^2_+: \abs{u}_1=n, \, 0 \to u\}$ denote the set of vertices on level $n$ that can be reached from the origin along open paths.   The  right edge   $a_n = \max_{u \in \cla_n}\{ u \cdot e_1 \}$ is  defined on the event $\{\cla_n \neq \varnothing\}$.    When $p_1\in(\pc,1)$ there exists a  constant $\beta_{p_1} \in (1/2, 1)$ such that   \cite[eqn.~(7) on p.~1005]{Dur-84}
		\[
			\lim_{n\to \infty} \frac{a_n}{n}\one\{ 0\to{\infty}\} = \beta_{p_1} \one\{ 0\to{\infty}\} \qquad
			\P\textrm{-a.s. }  
		\]
Let  $\etamax=(\beta_{p_1}, 1-\beta_{p_1})$ and $\etamin=(1-\beta_{p_1}, \beta_{p_1})$. The {\sl percolation cone} is  the set $\{\xi\in\R_+^2:  \xi/\abs{\xi}_1\in [\etamin, \etamax]\}$.  
The next theorem is proved by associating 
  an oriented site percolation process to the LPP process by defining $\sigma_x=\one\{\w_x=1\}$.  
 


%

\begin{theorem} \label{th:flat-edge}  Assume that  $\{\w_x\}_{x\in\Z^2}$ are i.i.d., $\E\abs{\w_0}^p<\infty$ for some $p>2$ and  $ \w_x\le 1$.  Suppose  $\opc<p_1=\P\{\w_0=1\}<1$.     Let $\xi\in\Uset$.  Then $\gpp(\xi)\le 1$,  and $\gpp(\xi)=1$  if and only if $\xi\in [\,\etamin, \etamax\,]$.   The endpoints  $\etamin$ and  $\etamax$ are points of differentiability of $\gpp$.  

\end{theorem}

Theorem \ref{th:flat-edge}  above summarizes a development  carried out for undirected  first-passage percolation in articles   \cite{Auf-Dam-13,Dur-Lig-81,Mar-02}.    
A proof of Theorem \ref{th:flat-edge} for  the corner growth model, adapted  from   the earlier arguments,  is  in Appendix D of \cite{Geo-Ras-Sep-13b-}.  
As a corollary,  our results that assume differentiable endpoints of a maximal linear segment are valid for the percolation cone.  

\begin{theorem}\label{th:flat-B}
Assume \eqref{2d-ass}, $ \w_x\le 1$ and   $\opc<p_1=\P\{\w_0=1\}<1$.  
There exists a stationary $L^1(\P)$ cocycle $\{B(x,y):x,y\in\Z^2\}$ and an event  $\Omega_0$ with $\P(\Omega_0)=1$ such that  the following statements hold for each 
$\w\in\Omega_0$.  Let   $v_n\in\Z_+^2$ be a sequence  such that 
\begin{align}\label{eq:vn-perc}
\abs{v_n}_1\to\infty\quad\text{and}\quad 1-\beta_{p_1}\le\varliminf_{n\to\infty}\frac{v_n\cdot e_1}{\abs{v_n}_1}\le\varlimsup_{n\to\infty}\frac{v_n\cdot e_1}{\abs{v_n}_1}\le\beta_{p_1}.  \end{align}
Then 
	\be\label{flat-B9}   B(\w, x,y) = \lim_{n\to \infty} \big( \Gpp_{x, v_n}(\w) - \Gpp_{y, v_n}(\w) \big)  \ee
						 for all  $x,y\in\Z^2$.     Furthermore,   \ 
$\E[B(x,x+e_1)]=  \E[B(x,x+e_2)]=1. $ 
\end{theorem}

Continuing with the assumptions of Theorem \ref{th:flat-B}, we   develop a more explicit formula for the Busemann function.  Let 
\be\label{B-psi}  \psi(\w)= \inf_{x_{0,\infty}: \, x_0=0}  \sum_{k=0}^\infty(1-\w_{x_k}) \ee
where the infimum is over all infinite up-right paths that start at $x_0=0$.     $\psi$ is measurable because it is the nondecreasing limit of $n-\Gpl_n(0)$ as $n\to\infty$.  As part of verifying this we take  a convergent subsequence of maximizing paths for $\Gpl_n(0)$ and   establish the existence of a path $\bar x_{0,\infty}$ such that 
\[  \psi(\w)= \sum_{k=0}^\infty(1-\w_{\bar x_k}) = \lim_{n\to\infty} (n - G_{0,\,\bar x_n}) . \] 
$\bar x_{0,\infty}$ must be a geodesic, that is,  any segment $\bar x_{m,n}$ for $0\le m<n<\infty$  must give the maximal  passage time $G_{\bar x_m,\,\bar x_n}$ because otherwise there is a better path.


  By  \cite[eqn.~(3) on p.~1028]{Dur-84} applied to oriented site percolation,  under $p_1>\opc$  the event  $\cup_{0\le k\le n}\{(k,n-k)\to\infty\}$ fails with probability at most $e^{-\gamma n}$ for a constant $\gamma>0$.  On this event $\psi\le (1-c)n$.   (Recall that $\w_x\ge c$ a.s.\ is part of assumption \eqref{2d-ass}.)   Consequently   $\psi$ has an exponential moment and in particular is almost surely  finite.   

\begin{theorem}
Under the assumptions of Theorem \ref{th:flat-B},   the Busemann function $B$ of the percolation cone  in \eqref{flat-B9}   is given by 
 \be\label{Ber-B}  B(\w,x,y)=(y-x)\cdot(e_1+e_2) +\psi(T_y\w) -\psi(T_x\w). \ee
\end{theorem} 

\begin{proof} Let $x'_{0,\infty}$ be a path that achieves the infimum over paths that start at  $x$, in the environment $\w$:  
\[   \psi(x,\w)=\psi(T_x\w)= \inf_{x_{0,\infty}: \, x_0=x}  \sum_{k=0}^\infty(1-\w_{x_k}) =\sum_{k=0}^\infty(1-\w_{x'_k}). \]

We argue that sequence $x'_n$ must satisfy \eqref{eq:vn-perc}.    To get a contradiction, suppose that   $x'_{n_i}\cdot e_1> \abs{x'_{n_i}}_1(\beta_{p_1}+\e)$ for some $\e>0$ and a  subsequence $n_i$.    Then by the shape theorem \eqref{lln5}  and the part of Theorem  \ref{th:flat-edge} that says $\gpp<1$ away from  the percolation cone,  
$\varlimsup  n_i^{-1} \Gpp_{x,\, x'_{n_i}}  \le 1-\delta$ for some $\delta>0$.   We have a contradiction:
\begin{align*}
\psi(x,\w) \ge  \sum_{k=0}^{n_i-1} (1-\w_{x'_k})  =  n_i-  \Gpp_{x, \,x'_{n_i}}   \nearrow \infty  \quad \text{ as $i\to\infty$.}  
\end{align*}

Now we can take the limit in \eqref{flat-B9} along $x'_n$.  Let $m'_n=\abs{x'_n-x}_1$ and $m''_n=\abs{x'_n-y}_1$. Take $n$ large enough so that $x'_n\ge x\vee y$ in coordinatewise order.  Use  $\Gpp_{y, \,x'_{n}} \le \Gpl_{m''_n}(0)\circ T_y$ to write  
\begin{align*}
 \Gpp_{x, \,x'_{n}}- \Gpp_{y, \,x'_{n}} &\ge 
   (x'_n-x)\cdot(e_1+e_2) -[  m'_n- \Gpp_{x, \,x'_{n}} ] \\
   &\qquad\qquad
    - (x'_n-y)\cdot(e_1+e_2) +[m''_n -  \Gpl_{m''_n}(0)\circ T_y ]
\end{align*}
and take a limit to get 
\[   B(\w,x,y)=\lim_{n\to\infty} ( \Gpp_{x, \,x'_{n}}- \Gpp_{y, \,x'_{n}}) 
\ge   (y-x)\cdot(e_1+e_2) -  \psi(T_x\w) + \psi(T_y\w).  \]  
Switch around $x$ and $y$ and use $B(\w,x,y)=-B(\w,y,x)$ to get the opposite inequality. The proof is complete.\qed
\end{proof}  

\begin{remark}[Weak disorder in the percolation cone] \label{rk:pc-wdo}   At inverse temperature $\beta\in(0,\infty)$, the point-to-line partition function of the {\it directed polymer model} is 
\[   Z_n =\sum_{x_{0,n-1}} 2^{-n+1} e^{\beta \sum_{k=0}^{n-1}\w_{x_k}}  \]
 where the sum is over admissible paths $x_{0,n-1}$ that start at $x_0=0$.  
The normalized partition function  $W_n=Z_n/\E Z_n$ is a positive martingale and has an a.s.\  limit $W_n\to W_\infty$.  By definition,  the model is in  {\it weak disorder} if $\P(W_\infty>0)=1$ and otherwise  in {\it strong disorder}   \cite[p.~208]{Hol-09}.  
It is known that the 1+1 dimensional directed polymer is in strong disorder at all positive $\beta$-values, as long as $\w_0$ has finite exponential moments.
(See \cite[Thm.~2.3(b)]{Com-Shi-Yos-03} and \cite[Thm.~1.1]{Car-Hu-02}).

The zero-temperature limit of the polymer  is the  point-to-line last-passage model:  $e^{\Gpl_n(0)}=\lim_{\beta\to\infty}  Z_n^{1/\beta}$, while   $\lim_{\beta\to\infty}  W_n^{1/\beta}=e^{G_n(0)-n}$.  Process  $e^{\Gpl_n(0)-n}$  is no longer a martingale  
but it is a supermartingale (because $\Gpl_n(0)\le \Gpl_{n-1}(0)+1$)  and, as we have seen, converges a.s.\ to $e^{-\psi}$.   
Above we observed that $\psi<\infty$   under $p_1>\opc$.  So  the martingale criterion suggests that the model is in weak disorder.   

We can observe two other markers of weak disorder.  The fluctuation exponent of $\Gpl_n(0)$ is zero, since without any normalization, 
$ \Gpl_n(0) -n\gpl(0) =  \Gpl_n(0) -n$ converges to a finite random quantity.  
Furthermore,  the centered cocycle $F(\w,x,y)=h(B)\cdot(x-y)-B(\w,x,y)$ that minimizes in Theorem \ref{thm:var-buse}\eqref{thm:var-buse:i} is a gradient:  
$F(\w,x,y)=\psi(T_x\w)-\psi(T_y\w)$.   It is in general true  in weak disorder that the    minimizer of  the variational formula for the quenched point-to-line free energy is a gradient \cite[Thm.~2.8]{Ras-Sep-Yil-15-}, though it is currently an open question  whether this is a characterization of weak disorder.  
\end{remark}

\medskip 

{\it Notice for the remainder of the paper.}   No assumptions  on the weights beyond \eqref{2d-ass}  and no regularity assumptions   on the shape function $\gpp$ are used, except when otherwise stated.

\section{Duality}
\label{sec:duality}

  By homogeneity we can represent  $\gpp$ by a single variable function.  A way of doing this that ties in naturally with the  queuing  theory arguments we use later  is   to   define  
  	\begin{align}
		\gppa(s)=\gpp(1,s)=\gpp(s,1)\quad \text{for}  \quad 0\le  s<\infty.    \label{eq:gbar}
	\end{align}
Function   $\gppa$ is real-valued, continuous and  concave.  Consequently 
one-sided derivatives $\gppa'(s\pm)$ exist and are monotone: $\gppa'(s_0+)\ge\gppa'(s_1-)\ge\gppa'(s_1+)$  for $0\le s_0<s_1$.  
Symmetry and homogeneity of $\gpp$ give  $\gppa(s)=s\gppa(s^{-1})$. 
 
\begin{lemma} 
\label{gppa-lm}
The derivatives satisfy  $\gppa'(s\pm)>\Ew$ for all $s\in\R_+$, 	
 	$\gppa'(0+)=\infty$, and    
	$\gppa'(\infty-)\equiv\ddd\lim_{s\to\infty}\gppa'(s\pm) =\gppa(0)=\Ew$.   
	 \end{lemma}

\begin{proof}
	The   shape universality at the boundary of $\R_+^2$ by   J.~Martin 
	\cite[Theorem 2.4]{Mar-04} says that 
		\begin{align}
		\label{eq:g-asym}
			\gppa(s)=\Ew+2\sigma\sqrt{ s}+o(\sqrt{s}\,)  \quad \text{as} \;  s\searrow 0.  
		\end{align}
	This gives $\gppa(0)=\Ew$ and 
  	$\gppa'(0+)=\infty$.
 	Lastly,  
		\[
			\gppa'(\infty-)=\lim_{s\to\infty}s^{-1}\gppa(s)=\lim_{s\to\infty}\gppa(s^{-1})=\gppa(0)=\Ew.
		\]
	Martin's asymptotic \eqref{eq:g-asym}  and $\gppa(s)=s\gppa(s^{-1})$ give 
		\begin{align}
		\label{eq:g-asym2}
			\gppa(s)=s\Ew+2\sigma\sqrt{ s}+o(\sqrt{s}\,)  \quad \text{as} \;  s\nearrow \infty.  
		\end{align}
	This is incompatible with having $\gppa'(s)=\Ew$ for $s\ge s_0$ for any
	$s_0<\infty$.  \qed
	\end{proof}

The lemma above has two 
important geometric consequences: 
\begin{align} 
\label{gg-1}&\text{every  linear segment of  $\gpp$   must lie in the interior $\ri\Uset$, \ \  and}  \\
\label{gg-2}&\text{the boundary $\{\xi\in\R_+^2:  \gpp(\xi)=1\}$  of the limit shape is asymptotic to the axes.}  
 \end{align}  

Define  
	\begin{align}
		\f(\alpha)=\sup_{s\ge0}\{\gppa(s)-s\alpha\} \quad \text{for}  \quad  \Ew< \alpha<\infty. \label{eq:flux}
	\end{align}

\begin{remark}[Queueing interpretations]\label{rk:q}  The quantities introduced in this section have natural interpretations in a queueing context.   The queueing connection  of LPP in terms of tandem service stations goes as follows.   Imagine a queueing system with customers labeled  $0,1,\dotsc,m$ and service stations labeled $0,1,\dotsc,n$.   The random weight $\w_{i,j}$ is the service time of customer $i$ at station $j$.  At time $t=0$ all customers  are  lined up at service station $0$.   Customers proceed through the system in order, obeying FIFO (first-in-first-out) discipline, and joining the queue at station $j+1$ as soon as service at station $j$ is complete.  Then for each $0\le k\le m$ and $0\le \ell\le n$,  
   $\Gpp_{(0,0),(k,\ell)}$ is the time when customer $k$ enters service at station $\ell$ and $\Gpp_{(0,0),(k,\ell)}+\w_{k,\ell}$ is the time when customer $k$ departs station $\ell$ and joins the end of the queue at station $\ell+1$.   Among the seminal references for these ideas are \cite{Gly-Whi-91,Mut-79}. 

In Section \ref{app:q} we  make use of a queueing system that operates in this  manner but is stationary in space and time, and the sequences of customers and stations are bi-infinite.  In this setting $\alpha\in(\Ew,\infty)$ is the mean interarrival (and interdeparture) time of customers at each queue, and it parametrizes the stationary distributions of the system.   $\f(\alpha)$ is the mean sojourn time, that is, the time between arrival and departure of a particular customer at a particular station.   
\end{remark}

\begin{lemma}\label{lm:f-properties}
	Function  $f$ is a strictly decreasing, continuous and convex involution 
	of the interval $(\Ew,\infty)$ onto itself, with  limits $\f(\Ew+)=\infty$ and $\f(\infty-)=\Ew$. That $f$ is an   involution means that $\f(\f(\alpha))=\alpha$.  
\end{lemma}

\begin{proof}
   	Asymptotics 
   	\eqref{eq:g-asym} and  \eqref{eq:g-asym2} imply  that   
	$\Ew< f(\alpha)<\infty$ for all $\alpha>\Ew$ and also that 
	the supremum in \eqref{eq:flux} is attained at some $s$. 
 	Furthermore, $\alpha<\beta$ implies $\f(\beta)=\gamma(s_0)-s_0\beta$ 
	with $s_0>0$ and $\f(\beta)<\gamma(s_0)-s_0\alpha\le\f(\alpha)$. 
	As a supremum of linear functions $\f$ is convex, and  
	hence  continuous on the open interval $(\Ew,\infty)$. 

We show how 	the   symmetry of $\gpp$ implies   that $\f$ is an involution.  
	By  concavity of $\gppa$,   
		\begin{align}
		\label{eq:f=g-sa}
			\f(\alpha)=\gppa(s)-s\alpha\quad\text{if and only if  $\alpha\in[\gppa'(s+),\gppa'(s-)]$}
		\end{align}
	and by Lemma \ref{gppa-lm} the intervals on the right cover $(\Ew,\infty)$. 
	Since $\f$ is strictly decreasing the above is the same as 
		\begin{align}
		\label{eq:a=g-f/s}
			\alpha=\gppa(s^{-1})-s^{-1}\f(\alpha)\quad\text{if and only if $\f(\alpha)\in[\f(\gppa'(s-)),\f(\gppa'(s+))]$.}
		\end{align}
	Differentiating $\gppa(s)=s\gppa(s^{-1})$ gives 
		\begin{align}
		\label{eq:gppa(1/s)}
			\gppa'(s\pm)=\gppa(s^{-1})-s^{-1}\gppa'(s^{-1}\mp).
		\end{align} 
	By \eqref{eq:f=g-sa} and  \eqref{eq:gppa(1/s)} 
	the condition in \eqref{eq:a=g-f/s} can be rewritten as
		 \begin{align}\label{eq:f=ga'}
			\f(\alpha)\in[\gppa(s)-s\gppa'(s-),\gppa(s)-s\gppa'(s+)]=[\gppa'(s^{-1}+),\gppa'(s^{-1}-)].
		 \end{align}
 Combining this with  \eqref{eq:f=g-sa} and  \eqref{eq:a=g-f/s} shows that $\alpha=\f(\f(\alpha))$.  
	The claim about the limits follows from $\f$ being a decreasing involution.\qed
\end{proof}
%
%

Extend these functions to the entire real line by  $\gppa(s)=-\infty$ when $s<0$ and $\f(\alpha)=\infty$ when $\alpha\le\Ew$.  Then convex duality gives 
	\be\label{eq:ga=f*}
		\gppa(s)=\inf_{\alpha>\Ew}\{\f(\alpha)+s\alpha\}.
	\ee

The natural bijection between $s\in(0,\infty)$ and $\xi\in\ri\Uset$ that goes together with \eqref{eq:gbar} is 
\be\label{s-xi} s=\xi\cdot e_1/\xi\cdot e_2. \ee  
Then direct differentiation, \eqref{eq:f=g-sa} and \eqref{eq:gppa(1/s)} give 
		\begin{align}
		\label{grad g-1}
			\nabla\gpp(\xi\pm) 
			=\bigl(\gppa'(s\pm),\gppa'(s^{-1}\mp)\bigr)
			=\bigl(\gppa'(s\pm),\f(\gppa'(s\pm))\bigr).
		\end{align}
Since  $\f$ is linear on $[\gppa'(s+),\gppa'(s-)]$,    we  get the following connection between the gradients of   $\gpp$  and the graph of $\f$:   for $\xi\in\ri\Uset$, 
		\be
		\label{f-9} 
			[\nabla\gpp(\xi+),\nabla\gpp(\xi-)]=\{(\alpha,\f(\alpha)):\alpha\in[\gppa'(s+),\gppa'(s-)]\}.  
		\ee
		
%

\smallskip 	

The next theorem details the duality between tilts $h$ and velocities $\xi$.  It is needed only for Section \ref{subsec:varsol} where we solve the variational formulas.  
 
\begin{theorem}\noindent\label{th:tilt-velocity}  
 	\begin{enumerate}[\ \ \rm(i)]
%
	\item\label{th:tilt-velocity:h->xi} Let $h\in\R^2$. There exists a unique $t=t(h)\in\R$ such that 
		\be
			h-t(e_1+e_2)\in-[\nabla\gpp(\xi+),\nabla\gpp(\xi-)]
		\label{eq:h-t}
		\ee
	 for some $\xi\in\ri\Uset$. 
	 The set  of  $\xi$  for which \eqref{eq:h-t} holds is a nonempty {\rm(}but possibly degenerate{\rm)} line segment 
 	$[\,\ximin(h),\ximax(h)]\subset\ri\Uset$.    If  $\ximin(h)\ne\ximax(h)$ then  $[\,\ximin(h),\ximax(h)]$ is a maximal linear segment of $\gpp$.   
	\item\label{th:tilt-velocity:xi-h} $\xi\in\ri\Uset$ and $h\in\R^2$ satisfy duality \eqref{eq:duality} if and only if     \eqref{eq:h-t} holds. 
	\end{enumerate}
\end{theorem}

\begin{proof}
	The graph $\{(\alpha,\f(\alpha)):\alpha>\Ew\}$ is strictly decreasing with 
	limits $f(\Ew+)=\infty$ and $f(\infty-)=\Ew$.  Since every 45 degree diagonal 
	intersects it at a unique point, the equation
	\begin{align}\label{eq:h-(a,t)}
			h=-(\alpha,\f(\alpha))+t(e_1+e_2) 
		\end{align}
defines  a bijection  $\R^2\ni h\longleftrightarrow (\alpha,t)\in(\Ew,\infty)\times\R$  illustrated in Figure \ref{fig:3.3}. 
Combining this with \eqref{f-9} shows that \eqref{eq:h-t} happens for a unique $t$ and for at least one $\xi\in\ri\Uset$.  

\begin{figure}[h]
	\begin{center}
		\begin{tikzpicture}[>=latex,  scale=.4]

			\draw[<->](0,15)--(0,0)--(15,0)node[right]{$\alpha$};
			\draw[dashed] (2,0)node[below]{$m_0$}--(2,15);
			\draw[dashed] (0,2)node[left]{$m_0$}--(15,2);
			\draw[semitransparent] (0,0) -- (13,13);
			
%
			\draw [line width=2.5pt, color= nicosred](3,6)--(4,4)--(6,3);
			\draw[domain=6:15,variable=\x,color=nicosred,line width=2.5pt]plot[smooth]({\x},{2+6/\x});
			\draw[domain=6:15,variable=\y,color=nicosred,line width=2.5pt]plot[smooth]({2+6/\y},{\y});
			
			\draw[semitransparent] (25/4,1)--(15,39/4);
			\draw(11,5.9)node[left]{$t(e_1+e_2)$}; 
			\draw[->, line width=2pt](13,31/4)node[right]{$-h$}--(8,11/4);
			\shade[ball color=red](8,11/4)circle(2.5mm);
			\shade[ball color=red](13,31/4)circle(2.5mm);
			\draw(11.5,11/4) node[above]{$(\alpha, f(\alpha))$}; 
			\shade[ball color=red](13,31/4)circle(2.5mm);
		\end{tikzpicture}
	\end{center}
	\caption{The graph of $\f$ and bijection \eqref{eq:h-(a,t)} between $(\alpha, t)$ and $h$.
			}
	\label{fig:3.3}
\end{figure}
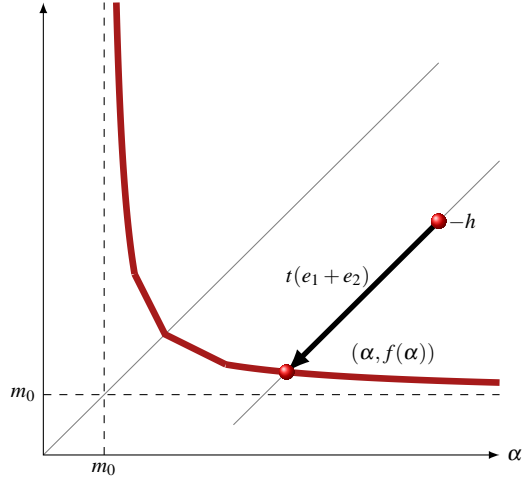

%
%
%
%
%

Once $h$ and $t=t(h)$ are given,  the geometry of the gradients (\eqref{grad g-1}--\eqref{f-9} and   limits \eqref{nabla-g-lim}) can be used to argue the claims about the $\xi$   that satisfy \eqref{eq:h-t}.    This proves part \eqref{th:tilt-velocity:h->xi}.  

That $h$ of the form \eqref{eq:h-t} is dual to $\xi$ follows readily from the fact that gradients are dual and  $\gpl(h+t(e_1+e_2))=\gpl(h)+t$ (this last from Definition \eqref{eq:g:p2l}).   

Note the following general facts for  any $q\in[\nabla\gpp(\zeta+),\nabla\gpp(\zeta-)]$.    By concavity   $\gpp(\eta)\le \gpp(\zeta)+q\cdot(\eta-\zeta)$ for all $\eta$.   Combining this with homogeneity gives $\gpp(\zeta)=q\cdot\zeta$.  Together with duality \eqref{h-xi}  we have 
	\begin{align}\label{g(h)=0}
	 \gpl(-q)=0  \; \text{ for } \; q\in\bigcup_{\zeta\in\ri\Uset}[\nabla\gpp(\zeta+),\nabla\gpp(\zeta-)]. 
	\end{align}  

It remains to show that if  $h$  is dual to $\xi$ then it satisfies \eqref{eq:h-t}.  Let $(\alpha, t)$ be determined by \eqref{eq:h-(a,t)}.  From the last two paragraphs  
\[  \gpl(h)=  \gpl(-\alpha, -\f(\alpha))+t =t. \]
  Let $s=\xi\cdot e_1/\xi\cdot e_2$  so that 
 	\[\gpp(\xi)+ h\cdot\xi=  \frac{\gppa(s)}{1+s}-\frac{\alpha s+\f(\alpha)}{1+s}  +t. \]
 Thus duality $\gpl(h)=\gpp(\xi)+ h\cdot\xi$ implies   $\gppa(s)=\alpha s+\f(\alpha)$ which happens if and only if $\alpha\in[\gppa'(s+),\gppa'(s-)]$.    \eqref{f-9} now implies \eqref{eq:h-t}. \qed
\end{proof}

\section{Stationary cocycles}\label{sec:cocycles}

In this section we describe  the   stationary cocycles, then show how these define stationary last-passage percolation processes and also solve the variational formulas for $\gpp(\xi)$ and $\gpl(h)$.   Assumption  \eqref{2d-ass} is in force but no other assumptions are made.

\subsection{Existence and properties of stationary cocycles} \label{sec:exist} 
   By appeal to queueing fixed points,   in Section \ref{app:q} we construct a family of cocycles  $\{B^\xi_{\pm}\}_{\xi\in\ri\Uset}$  on an extended space   $\Ombig=\Omega\times\Omega'  =\Omega\times \R^{\{1,2\}\times\cA_0\times\Z^2}$  where $\cA_0$ is a countable subset of the interval $(\Ew, \infty)$,  defined   in \eqref{q:A_0} below.      Generic elements  of $\Ombig$ are denoted by 
$\what=(\w, \w')$ where  $\w=(\w_x)_{x\in\Z^2}\in\Omega=\R^{\Z^2}$   is the original weight configuration  and $\w'=(\w^{i,\alpha}_x)_{i\in\{1,2\}, \,\alpha\in\cA_0, \, x\in\Z^2}$.    $\kSbig$ denotes the Borel $\sigma$-algebra of $\Ombig$, and in this context $\kS$ denotes   the sub-$\sigma$-algebra of $\kSbig$ generated by the projection $\what\mapsto\w$.    Spatial translations act   in the usual manner:  $(T_x\what)_y=\what_{x+y}$ for $x,y\in\Z^2$ where $\what_x=(\w_x,  \w'_x)=(\w_x,  (\w^{i,\alpha}_x)_{i\in\{1,2\}, \,\alpha\in\cA_0})$.  

Definition \ref{def:cK} of a cocycle  makes no reference to the LPP process.  
The key feature  that connects cocycles with the   last-passage weights is captured  in the next definition.   It was isolated in our previous paper \cite{Geo-Ras-Sep-13a-} that developed the cocycle set-up for general polymer and percolation models.  
This property is behind all our applications of the cocycles:    construction of stationary LPP in Section \ref{sec:stat-lpp} which in turn is used to prove  Busemann limits,  identification of minimizers of variational formulas in Section \ref{subsec:varsol}, and study of geodesics in \cite{Geo-Ras-Sep-15b-}.  
\begin{definition}
		\label{def:bdry-model}
			A stationary $L^1$ cocycle $B$ on $\Omhat$  {\rm recovers} 
			potential $V$ if 
				\be \label{eq:VB}
					V(\what)=\min_{i\in\{1,2\}} B(\what, 0,e_i)\quad\text{ for }\Pbig\text{-a.e.\ }\what.
				\ee
	 \end{definition}
  In our present case the  potential $V:\Omhat\to\R$  is  simply   $V(\what)=\w_0$, the last-passage weight at the origin.

    The next theorem gives the existence statement  and summarizes the properties of these cocycles.   
  This is the only place where our proofs   use the assumption $\P(\w_0\ge c)=1$,  and the only reason   is that the queueing results we reference have been proved only for $\w_0\ge 0$.    In part \eqref{th:cocycles:indep} below we use  this notation:  for  a finite or infinite set $I\subset\Z^2$,   $I^<=\{x\in\Z^2: x\not\ge z \;\forall z\in I\}$ is the set of lattice points that do not lie on a  ray from $I$ at an angle in $[0,\pi/2]$. For example, if $I=\{0,\dotsc,m\}\times\{0,\dotsc,n\}$ then $I^<=\Z^2\smallsetminus\Z_+^2$.



  \begin{theorem}\label{th:cocycles}
There exist real-valued Borel functions 
$B^\xi_+(\what,x,y)$ and  $B^\xi_-(\what,x,y)$ of  $(\what,\xi,x,y)\in\Ombig\times \ri\Uset\times\Z^2\times\Z^2$   and a translation invariant 
  Borel probability measure $\Pbig$ on  $(\Ombig, \kSbig)$ such that  the following properties hold.  
\begin{enumerate}[\ \ \rm(i)]  
\item\label{th:cocycles:indep} 
Under $\Pbig$, the marginal distribution of the configuration $\w$ is the i.i.d.\ measure $\P$ specified in assumption \eqref{2d-ass}. 
For each $\xi\in\ri\Uset$ and $\pm$, the  $\R^3$-valued  process   
$\{\psi^{\pm, \xi}_x\}_{x\in\Z^2}$
defined by   
\[ \psi^{\pm, \xi}_x(\what)=(\w_x, B^\xi_{\pm}(\what, x,x+e_1), B^\xi_{\pm}(\what, x,x+e_2))\] 
is separately ergodic under both translations $T_{e_1}$ and $T_{e_2}$.  
For any $I\subset\Z^2$,   the variables 
\[ \bigl \{ (\w_x, B^\xi_{+}(\what,x,x+e_i), B^\xi_{-}(\what,x,x+e_i)) : x\in I, \, \xi\in\ri\Uset, \, i\in\{1,2\} \bigr\}\] 
 are  independent of $\{ \w_x : x\in I^< \}$.   \\[-8pt]


 \item\label{th:cocycles:exist}  Each process 
 $B^\xi_{\pm}=\{B^\xi_{\pm}(x,y)\}_{x,y\in\Z^2}$ 
  is a stationary $L^1(\Pbig)$ cocycle {\rm(}Definition \ref{def:cK}{\rm)}  that recovers the  potential   {\rm(}Definition \ref{def:bdry-model}{\rm)}:
  \be\label{Bw-9}  \w_x=B^\xi_{\pm}(\what, x,x+e_1) \wedge B^\xi_{\pm}(\what, x,x+e_2)\qquad \text{$\Pbig$-a.s.}  \ee 
  The mean  vectors   $h_\pm(\xi)=h(B^\xi_{\pm})$ defined by \eqref{eq:EB}  satisfy
					\begin{align}\label{eq:h=grad} 
					-h_\pm(\xi)=
\bigl( \,\Ebig[B^\xi_{\pm}(x,x+e_1)]\,,\,  \Ebig[B^\xi_{\pm}(x,x+e_2)]\,\bigr) =\nabla\gpp(\xi\pm)\end{align}
				and are dual  to velocity $\xi$ as in \eqref{eq:duality}.   \\[-8pt]
				
	\item\label{th:cocycles:flat} 			 No two distinct cocycles have  a common tilt vector.  That is, if $h_+(\xi)=h_-(\zeta)$ then 
	\[  B^\xi_{+}(\what,x,y)=B^\zeta_{-}(\what,x,y)  
	\quad \forall\, \what\in\Ombig, \, x,y\in\Z^2  \] 
and similarly for all four combinations of $\pm$ and $\xi,\zeta$.  
 These equalities  hold    for all $\what$   without an   almost sure modifier    because they come directly   from the construction.   In particular,  if   $\xi\in\Diff$ then 
	\[  	 B^\xi_{+}(\what,x,y)= B^\xi_{-}(\what, x,y) = B^\xi(\what,x,y)\quad \forall \,\what\in\Ombig, \, x,y\in\Z^2, 
	 \]
	where the second equality defines the cocycle $B^\xi$.     \\[-8pt]
	
\item\label{th:cocycles:cont}  
There exists an event $\Ombig_0$ with $\Pbig(\Ombig_0)=1$ and such that {\rm(a)} and {\rm(b)} below  hold for all $\what\in\Ombig_0$, $x,y\in\Z^2$ and  $\xi,\zeta\in\ri\Uset$.  
 
{\rm(a)}  Monotonicity: if     $\xi\cdot e_1<\zeta\cdot e_1$ then  
					\be\begin{aligned}  
						B^\xi_{-}(\what, x,x+e_1) &\ge B^\xi_{+}(\what, x,x+e_1) \ge B^\zeta_{-}(\what, x,x+e_1) \\ \quad\text{and}\quad B^\xi_{-}(\what,x,x+e_2) &\le B^\xi_{+}(\what, x,x+e_2) \le B^\zeta_{-}(\what, x,x+e_2). 
					\end{aligned} \label{eq:monotone} 
					\ee
					
{\rm(b)}	 Right continuity: if $\xi_n\cdot e_1\searrow\zeta\cdot e_1$  then 
					\begin{align}\label{eq:cont}
						\lim_{n\to\infty}B^{\xi_n}_{\pm}(\what, x,y) = B^\zeta_{+}(\what, x,y)  . 
					\end{align}
					
\item\label{th:cocycles:cont-left}   Left continuity at a fixed $\zeta\in\ri\Uset$:   there exists an event $\Ombig^{(\zeta)}$ with $\Pbig(\Ombig^{(\zeta)})=1$ and such that for any sequence $\xi_n\cdot e_1\nearrow\zeta\cdot e_1$  
	\begin{align}\label{eq:cont-left}
						\lim_{n\to\infty}B^{\xi_n}_{\pm}(\what, x,y) = B^\zeta_{-}(\what,x,y)\qquad \text{for $\what\in\Ombig^{(\zeta)}, x,y\in\Z^d$.}   
					\end{align}

Limits \eqref{eq:cont} and \eqref{eq:cont-left}   hold also   in $L^1(\Pbig)$ due to inequalities \eqref{eq:monotone}.  
\end{enumerate} 
\end{theorem}

 The conditional expectations $\Ebig[ B^\xi_{\pm}( x, y)\vert\kS]$ are cocycles that are  functions of $\w$ alone, but the conditioning may destroy crucial property \eqref{Bw-9} that relates the cocycles to the percolation.
Some mild additional regularity on $\gpp$ guarantees that  all cocycles  $\what\mapsto B^\xi_{\pm}(\what, x, y)$   are in fact $\kS$-measurable.  This theorem is proved in Section \ref{sec:busemann}.  

\begin{theorem}\label{thm:kS}  
Assume that  $\gpp$ is differentiable at the endpoints of its linear segments {\rm(}if any{\rm)}.   Then all cocycles $\{B^\xi_{\pm}(x,y)\}_{\xi\in\ri\Uset, \, x,y\in\Z^2}$ from  Theorem \ref{th:cocycles}  are  measurable with respect to {\rm(}the completion of{\rm)} $\kS$.  
\end{theorem}

\smallskip

\begin{remark}\label{rm:cocycles} 
 The  cocycle construction of Theorem \ref{th:cocycles} undertaken  in Section \ref{app:q} utilizes   a   countable dense subset $\Uset_0$ of $\Uset$ such that, for $\xi\in\Uset_0$,    nearest-neighbor cocycle values are coordinate projections   $B^\xi_{\pm}(\what, x,x+e_i)=\w^{i, \gppa'(s\pm)}_x$   
 where $s$ is defined by \eqref{s-xi}.
 $\Uset_0$   contains all points of nondifferentiability and endpoints of linear segments of $\gpp$.  
  For   $\zeta\in(\ri\Uset)\smallsetminus\Uset_0$   we define $B^\zeta=B^\zeta_\pm$ through right limits from  $\{B^\xi_{\pm}\}_{\xi\in\Uset_0}$.     This is behind  Theorem \ref{th:cocycles}(\ref{th:cocycles:flat})--\eqref{th:cocycles:cont}.    Monotonicity \eqref{eq:monotone} simultaneously for all $B^\xi_\pm$  outside a single null set will be convenient for constructing geodesics in   \cite{Geo-Ras-Sep-15b-}. 
   
     The cocycle properties and \eqref{Bw-9} can also be arranged to hold simultaneously  for all $B^\xi_\pm$ outside a single null set, if so desired.  
  But for  left and right limits  to  agree at a particular $\zeta$ we have to allow for the  exceptional $\Pbig$-null event $(\Ombig^{(\zeta)})^c$  that is specific to $\zeta$.   Thus left  limit \eqref{eq:cont-left} is not claimed for all $\zeta$ outside a single null set.  
  
Other conventions are possible in the construction.  We could extend $B^\xi_\pm$ from $\Uset_0$ so that    $B^\xi_+$ is  right-continuous and $B^\xi_-$   left-continuous in $\xi$.  Monotonicity \eqref{eq:monotone} would still hold outside a single null set,  but   $B^\xi_+=B^\xi_-$ would be only almost surely true for a given $\xi\in\EP$, instead of identically true.  

\end{remark}

\subsection{Stationary last-passage percolation} \label{sec:stat-lpp}

Fix a  cocycle $B(\what, x,y)=B^\xi_+(\what, x,y)$ or $B^\xi_-(\what, x,y)$  from Theorem \ref{th:cocycles}.  Fix a point $v\in\Z^2$ that will serve as an origin.   
By part \eqref{th:cocycles:indep} of Theorem \ref{th:cocycles}, the weights $\{\w_x: x\le v-e_1-e_2\}$ are independent of  $\{ B(v-(k+1)e_i, v-ke_i): k\in\Z_+, \, i\in\{1,2\}\}$.  These define a stationary last-passage percolation process in the third quadrant relative to the  origin $v$, in the following sense.   Define  passage times $\Gne_{u,v}$ that  use the cocycle as 
edge weights   on the 
north and  east boundaries   and  weights  $\w_x$  in the bulk  $x\le v-e_1-e_2 $: 	\begin{align}
	\label{eq:Gnecorner}
		\begin{split}
		\Gne_{u,v}&=B(u,v)\quad\text{for $u\in\{v-ke_i: k\in\Z_+, \, i\in\{1,2\}\}$}\\
	\qquad 	\text{and} \qquad 		\Gne_{u,v}
			&=\w_u+\Gne_{u+e_1,v}\vee \Gne_{u+e_2,v}\qquad \text{for } \  u\le v-e_1-e_2\,.
		\end{split}
	\end{align}
It is immediate from  recovery 
$\w_x=B(x,x+e_1)\wedge B(x,x+e_2)$     and additivity of $B$ that 
\[\Gne_{u,v}=  B(u,v)  \qquad \text{for all $u\le v$.  } \] 
Process $\{\Gne_{u,v}: u\le v\}$ is stationary in the sense that the increments
\be\label{stat-lpp5} \Gne_{x,v}-\Gne_{x+e_i,v} = B(x,x+e_i)  \ee
are stationary under lattice translations and, as the equation above reveals,  do not depend on the choice of the origin $v$ (as long as we stay southwest of the origin).

 \begin{remark}   In the exactly solvable cases where $\w_x$ is either exponential or geometric,  more is known.  Given the stationary cocycle,  define weights 
 \[Y_x=B(x-e_1,x)\wedge B(x-e_2,x).  \]
 Then the  weights   $\{Y_x\}$ have the same i.i.d.\ distribution as the original weights $\{\w_x\}$.  Furthermore,  $\{Y_x: x\ge v+e_1+e_2\}$ are independent of  $\{ B(v+ke_i, v+(k+1)e_i): k\in\Z_+, \, i\in\{1,2\}\}$.   Hence a  stationary last-passage percolation process can be defined   in the first quadrant with cocycles on the south and west boundaries: 
\begin{align*} 	
		\begin{split}
		\Gsw_{v, x}&=B(v,x)\quad\text{for $x\in\{v+ke_i: k\in\Z_+, \, i\in\{1,2\}\}$}\\
	\qquad 	\text{and} \qquad 		\Gsw_{v,x}
			&=Y_x+\Gsw_{v, x-e_1}\vee \Gsw_{v, x-e_2}\qquad \text{for } \  x\ge v+e_1+e_2\,.
		\end{split}
	\end{align*}

   This feature appears in \cite{Bal-Cat-Sep-06}  as the ``Burke property'' of the exponential last-passage model.  It also works for the log-gamma polymer in positive temperature \cite{Geo-etal-15-, Sep-12-corr}.  We do not know presently if this works in the general last-passage case. 
 \end{remark}
  
\subsection{Solution to the variational formulas}\label{subsec:varsol}

In this section we construct 
  minimizers for   variational formulas  \eqref{eq:g:K-var}--\eqref{eq:gpp:K-var} in terms of the cocycles from Theorem \ref{th:cocycles}.    The theorems below are   versions of    Theorem \ref{thm:var-buse}   that  we can state without extra regularity assumptions on the shape function $\gpp$.     The proof of Theorem \ref{thm:var-buse} has to wait till  Section \ref{sec:busemann},  where  we   identify the  minimizing cocycles used below  as Busemann functions under regularity assumptions on $\gpp$.      
  
    Recall from Theorem \ref{th:tilt-velocity} that $\gpp$ is linear over each line segment   $[\,\ximin(h),\ximax(h)]$ and hence, 
by  Theorem \ref{th:cocycles}\eqref{th:cocycles:flat}, 
cocycles $B^\xi$   coincide for all $\xi\in\,]\ximin(h),\ximax(h)[$.


\begin{theorem}\label{th:var-sol}
	Let $\{ B^\xi_{\pm}\}$ be the cocycles given  in  Theorem \ref{th:cocycles}.  Fix $h\in\R^2$. Let $t(h)$, $\ximin(h)$, and $\ximax(h)$ be as in Theorem \ref{th:tilt-velocity}. One has the following three cases.
		\begin{enumerate}[\ \ \rm(i)]
			\item\label{th:var-sol:flat} $\ximin(h)\ne\ximax(h)$:   For any {\rm(}and hence all{\rm)} $\xi\in\; ]\ximin(h),\ximax(h)[\,$  let 
				\begin{align}
				\label{eq:F-solution}
					F^h(\what, x,y)=h(\xi)\cdot(x-y)-B^\xi(\what, x,y). 
				\end{align}
	Then for $\Pbig$-almost every $\what$
			\begin{align}
			\label{eq:solution}
				\gpl(h)=\max_{i=1,2}\{\w_0+h\cdot e_i+F^h(\what, 0,e_i)\} = t(h).
			\end{align}
		\item\label{th:var-sol:diff} $\ximin(h)=\ximax(h)=\xi\in\Diff$: \eqref{eq:solution} holds for $F^h$ defined as in \eqref{eq:F-solution}. 
		
		\item\label{th:var-sol:notdiff} $\ximin(h)=\ximax(h)=\xi\not\in\Diff$: Let $\theta\in[0,1]$ be such that \[ h-t(h)(e_1+e_2)=\theta h_-(\xi)+(1-\theta)h_+(\xi)\] and define 
			\begin{align*}
				&F^{\xi\pm}(\what,x,y)=h_\pm(\xi)\cdot(x-y)-B^\xi_{\pm}(\what,x,y)\quad\text{and}\\
				&F^h(\what,x,y)=\theta F^{\xi-}(\what,x,y)+(1-\theta)F^{\xi+}(\what,x,y).
			\end{align*}
				If $\theta\in\{0,1\}$,   \eqref{eq:solution} holds again almost surely.  
For all $\theta\in[0,1]$ we have 
		\be \label{eq:solution2}	\begin{aligned}
							\gpl(h)&=
							\P\text{-}\esssup_{\w} \max_{i=1,2} \,\{\w_0+h\cdot e_i+\Ebig[ F^h(0,e_i)\vert\kS] \,\}  \\
							&=\Pbig\text{-}\esssup_{\what}\,\max_{i=1,2}\,\{\w_0+h\cdot e_i+F^h(\what, 0,e_i) \} = t(h).
			\end{aligned}\ee
	\end{enumerate}
	 In particular, in all cases \eqref{th:var-sol:flat}--\eqref{th:var-sol:notdiff},    $\Ebig[ F^h(x,y)\vert\kS]\in \Coc_0(\Omega)$ is a minimizer in \eqref{eq:g:K-var}. 
\end{theorem}

Here are qualitative descriptions of the cases above. 

(i) The graph  of $\f$ has a corner at the point $(\alpha, \f(\alpha))$ where it crosses the $45^o$ line through $-h$.  Correspondingly,  
$\gpp$ is linear on $[\ximin(h),\ximax(h)]$ with gradient $\nabla\gpp(\xi)=(\alpha, \f(\alpha))$ at  interior points $\xi\in\; ]\ximin(h),\ximax(h)[$. 

(ii)  The unique   $\xi$ dual  to $h$ lies in $\EP$. 

(iii)   The unique   $\xi$ dual  to $h$ is exposed but not in $\Diff$.  

\begin{proofof}{of Theorem \ref{th:var-sol}}
 Let $B$ be one of   $B^\xi, B^\xi_\pm$  and define centered cocycle $F$  by   \eqref{FF}.    	 By \eqref{eq:h=grad} and \eqref{g(h)=0},  $\gpl(h(B))=0$.
Then directly from definitions  \eqref{Gnh} and  \eqref{eq:g:p2l},    $\gpl(h)=(h-h(B))\cdot e_j$ for $j\in\{1,2\}$ for  any $h\in\R^2$ that 
satisfies $(h-h(B))\cdot(e_2-e_1)=0$.    Hence by recovery \eqref{Bw-9}, for these same $h$-values,  for $\Pbig$-a.e.\ $\what$ and $j\in\{1,2\}$, 
\be \label{eq:Kvar:min}	\begin{aligned}
	\gpl(h)&\;=\;\max_{i\in\{1,2\}}\, \{ \w_0+h\cdot e_i+F(\what, 0,e_i) \}
	\;=\;(h-h(B))\cdot e_j .
	\end{aligned}\ee
  (This situation is developed for  general models in  Theorem 3.4 of \cite{Geo-Ras-Sep-13a-}.)


		In cases \eqref{th:var-sol:flat} and \eqref{th:var-sol:diff}, \eqref{eq:solution}  comes from   \eqref{eq:Kvar:min} combined with \eqref{eq:h-t}  and \eqref{eq:h=grad}. 
	  The same   works in case \eqref{th:var-sol:notdiff} when $\theta\in\{0,1\}$.
	
	Consider  case \eqref{th:var-sol:notdiff}.  
	Using  Theorem \ref{th:tilt-velocity} and  $\gpl(h_\pm(\xi))=0$ gives 
		\begin{align*}
			\gpl(h)&=\gpp(\xi)+h\cdot\xi\\
				  &=t(h)+\theta(\gpp(\xi)+h_-(\xi)\cdot\xi)+(1-\theta)(\gpp(\xi)+h_+(\xi)\cdot\xi)\\
			          &=t(h)+\theta\gpl(h_-(\xi))+(1-\theta)\gpl(h_+(\xi)) 
			          =t(h).
		\end{align*}
By \eqref{Bw-9},   $\Pbig$-almost surely 
		\be\label{aux-67}  \min\{\theta B^\xi_-(0,e_1)+(1-\theta)B^\xi_+(0,e_1)\,,\,\theta B^\xi_-(0,e_2)+(1-\theta)B^\xi_+(0,e_2)\}\ge\w_0.\ee 
Since $\Ebig[ F^h(x,y)\vert\kS]$ is a member of $\Coc_0(\Omega)$,   we can use 	\eqref{eq:g:K-var}  to verify \eqref{eq:solution2}: 
\be\label{aux-68} \begin{aligned}
\gpl(h)&\le  \P\text{-}\esssup_{\w} \max_{i=1,2}\, \{\w_0+h\cdot e_i+\Ebig[ F^h(0,e_i)\vert\kS] \,\}  \\
&\le \Pbig\text{-}\esssup_{\what} \max_{i=1,2} \,\{\w_0+h\cdot e_i+F^h(\what, 0,e_i)\}   \le t(h) = \gpl(h).
\end{aligned}\ee 	
The last inequality above is just a rearrangement of 	\eqref{aux-67}.  	
The same inequalities show that  $\Ebig[ F^h(0,e_i)\vert\kS]$  minimizes in \eqref{eq:g:K-var}  in cases \eqref{th:var-sol:flat} and \eqref{th:var-sol:diff}.   
\qed\end{proofof}

We state also the corresponding theorem for the point-to-point case. Using the duality \eqref{eq:duality} of $h_\pm(\xi)$ and  $\xi$,  it follows as the theorem above.   

\begin{theorem}\label{thm:var-p2p}
	Let $\xi\in\ri\Uset$.  Then  for $\Pbig$-a.e.~$\what$
			\begin{align}
			\label{eq:solution-pp}
			\begin{split}
				\gpp(\xi)=\max_{i=1,2} \,\{\w_0-B^\xi_{\pm}(\what, 0,e_i)-h_\pm(\xi)\cdot\xi \}  = -h_\pm(\xi)\cdot\xi.  
			\end{split}
			\end{align}
For any 	$\theta\in[0,1]$,   cocycle $\Ebig[  \theta B^\xi_{-}+(1-\theta) B^\xi_{+}\vert\kS]\in\Coc(\Omega)$ is  a minimizer in \eqref{eq:gpp:K-var}. 
\end{theorem}

\section{Busemann functions from cocycles}
\label{sec:busemann}

In this section we prove the  existence of   Busemann functions.   As before,     \eqref{2d-ass} is a standing assumption.   No regularity assumptions are made on $\gpp$, hence the results below are in terms of inequalities for  limsup and liminf.     Under additional regularity assumptions,  the sharper theorems claimed in earlier sections  are derived as corollaries.

Recall the line segment $\Uset_\xi=[\ximin, \ximax]$ with $\ximin\cdot e_1\le \ximax\cdot e_1$ from \eqref{eq:sector1}--\eqref{eq:sector2}
and the cocycles $B^\xi_{\pm}$ constructed on the extended space $\OBPbig$ in Theorem \ref{th:cocycles}.   

\begin{theorem}\label{th:construction} 
Fix a possibly degenerate segment   $[\zeta,\eta]\subset\ri\Uset$.  Assume    either that there is an   exposed  point $\xi$  such that  $\xi=\ximin=\ximax$ and  $[\zeta,\eta]=\{\xi\}=[\,\ximin,\ximax\,]$,  or that  $[\zeta,\eta]$  a maximal nondegenerate  linear segment of $\gpp$ in which case  $[\zeta,\eta]=[\,\ximin,\ximax\,]$ for any $\xi\in]\zeta,\eta[$.
Then there exists an event  $\Ombig_0$ with $\Pbig(\Ombig_0)=1$ such that for each 
$\what\in\Ombig_0$ and  for any  sequence $v_n\in\Z_+^2$ that satisfies 
\begin{align}\label{eq:vn66}
\abs{v_n}_1\to\infty\quad\text{and}\quad\ximin\cdot e_1\le\varliminf_{n\to\infty}\frac{v_n\cdot e_1}{\abs{v_n}_1}\le\varlimsup_{n\to\infty}\frac{v_n\cdot e_1}{\abs{v_n}_1}\le\ximax\cdot e_1,
\end{align}
  we have for all $x\in\Z^2$ 
\be\label{bu:G1}  \begin{aligned}   B^{\ximax}_{+}(\what, x,x+e_1)&\le \varliminf_{n\to \infty} \big( \Gpp_{x, v_n}(\w) - \Gpp_{x+e_1, v_n}(\w) \big) \\ &\le \varlimsup_{n\to \infty} \big( \Gpp_{x, v_n}(\w) - \Gpp_{x+e_1, v_n}(\w)  \big) \le B^{\ximin}_{-}(\what,  x,x+e_1)  
\end{aligned}  \ee
and 
\be\label{bu:G2}  
\begin{aligned}  B^{\ximin}_{-}(\what, x,x+e_2)  &\le \varliminf_{n\to \infty} \big( \Gpp_{x, v_n}(\w) - \Gpp_{x+e_2, v_n}(\w) \big)\\
&\le \varlimsup_{n\to \infty} \big( \Gpp_{x, v_n}(\w) - \Gpp_{x+e_2, v_n}(\w) \big) \le B^{\ximax}_{+}(\what, x,x+e_2).  
\end{aligned}  \ee
\end{theorem}


Let us emphasize that the sequence $v_n$ is allowed to  depend on $\what\in\Ombig_0$.  
The interesting cases are of course the ones where we have a limit.   The next corollary follows immediately because if 
$\xi, \ximin, \ximax\in\Diff$  then by Theorem \ref{th:cocycles}\eqref{th:cocycles:flat}
$B_\pm^{\ximin}=B^{\xi}=B_\pm^{\ximax}$ for all $\xi\in\, [\ximin\,, \ximax]\,$.     

\begin{corollary}\label{cor:buse} 
 Assume  $\xi, \ximin, \ximax\in\Diff$ .
Then there exists an event  $\Ombig_0$ with $\Pbig(\Ombig_0)=1$ such that for each 
$\what\in\Omega_0$,   for any  sequence $v_n\in\Z_+^2$ that satisfies \eqref{eq:vn66}, 
and  for all $x,y\in\Z^2$, 
\be \label{eq:grad:coc} 
							B^\xi(\what,x,y) = \lim_{n\to \infty} \big( \Gpp_{x, v_n}(\w) - \Gpp_{y, v_n}(\w)  \big). 
						\ee 
				The limit implies that $\P$-a.s.\	$B^\xi(\what,x,y)$ is  a function of $(\w, x, y)$.  
						
						In particular, if $\gpp$ is differentiable everywhere on $\ri\Uset$,  then  for each direction $\xi\in\ri\Uset$ there is an event of full $\Pbig$-probability on which   limit \eqref{eq:grad:coc}  holds for any sequence $v_n/\abs{v_n}_1\to\xi$ with $\abs{v_n}_1\to\infty$.  
 \end{corollary}

Before the  proof of Theorem \ref{th:construction}, we  complete the proofs of some earlier theorems.

\begin{proofof}{of Theorem \ref{thm:buse}}   Limit \eqref{eq:grad:coc1}    is in Corollary \ref{cor:buse}.    Equation \eqref{EB=Dg} follows from \eqref{eq:h=grad}.

To prove the point-to-line limit \eqref{p2lB-1}    recall the duality from Theorem \ref{th:tilt-velocity}.   Observe that
\be\label{p2lB-5} 
\begin{aligned}
\zeta\in[\ximin, \ximax] &\;\Longleftrightarrow\;  \gpp(\zeta)=\gpp(\xi)+\nabla\gpp(\xi)\cdot(\zeta-\xi)  \\
&\;\Longleftrightarrow\;   \gpp(\zeta)=\gpp(\xi)-h\cdot(\zeta-\xi)   
\;\Longleftrightarrow\;  \zeta\in[\,\ximin(h), \ximax(h)\,] .
\end{aligned}\ee
 
Let $z\in\{e_1,e_2\}$.   Pick ($\w$-dependent)  $u_n, v_n\in\Z^2_+$  so  that  $\abs{u_n}_1=\abs{v_n}_1=n$ and 
\[  \Gpl_{n}(h)=\Gpp_{0,u_n}+h\cdot u_n 
\quad\text{and}\quad 
\Gpl_{n-1}(h)\circ T_z=\Gpp_{z,v_n}+h\cdot(v_n-z).  \]
Fix $\w$ so that limits \eqref{eq:g:p2l} and \eqref{lln5} hold both for $\w$ and $T_z\w$.   Then taking $n\to\infty$ along suitable subsequences shows that all  limit points of $u_n(\w)/n$ and $v_n(\w)/n$ satisfy duality \eqref{eq:duality} and so lie in $[\,\ximin(h), \ximax(h)\,]$.  By \eqref{p2lB-5}   these sequences satisfy  \eqref{eq:vn66}. 
Let $n\to\infty$  in the inequalities 
\[\Gpp_{0,v_n}-\Gpp_{z,v_n}+h\cdot z \le \Gpl_{n}(h)-\Gpl_{n-1}(h)\circ T_z \le \Gpp_{0,u_n}-\Gpp_{z,u_n}+h\cdot z,\]
and use \eqref{eq:grad:coc} to get the conclusion.  
\qed
\end{proofof}

\begin{proofof}{of Theorem \ref{thm:var-buse}}
The theorem follows from  
Theorems \ref{th:var-sol} and  \ref{thm:var-p2p}
because the Busemann function $B^\xi$ is the  cocycle $B^\xi$ from Theorem \ref{th:cocycles}. 
\qed\end{proofof}



\begin{proofof}{of Theorem \ref{thm:kS}}
Under the assumption of differentiability at endpoints of linear segments, every $\xi\in\Diff$ satisfies Corollary \ref{cor:buse} and so $B^\xi(\cdot, x,y)$ is $\kS$-measurable.  Any other point $\zeta\in\ri\Uset$ is a limit from both left and right of $\Diff$-points,  and  so by parts \eqref{th:cocycles:cont} and \eqref{th:cocycles:cont-left} of Theorem \ref{th:cocycles},  cocycles $B^\zeta_{\pm}$  are a.s.\ limits of $\kS$-measurable cocycles.   
\qed\end{proofof}

 The remainder of this section proves Theorem \ref{th:construction}.   
 We begin with   a general comparison lemma.   This idea  goes back at least to \cite{Alm-98, Alm-Wie-99}.  
With  arbitrary real weights $\{\wt Y_x\}_{x\in\Z^2}$ define  last-passage times   
 	\[\wt \Gpp_{u,v}=\max_{x_{0,n}}\sum_{k=0}^{n-1}\wt Y_{x_k}.\]
The maximum is over up-right paths from $x_0=u$ to $x_n=v$ with $n=|v-u|_1$.  The convention is $\wt \Gpp_{v,v}=0$.  
For  $x\le v-e_1$ and $y\le v-e_2$  denote the increments  by 
	\[	\wt I_{x,v} = \wt \Gpp_{x,v} - \wt \Gpp_{x+e_1,v} \qquad  \text{ and } \qquad  \wt J_{y,v} = \wt \Gpp_{y,v} - \wt \Gpp_{y+e_2,v}\,. \]
	
\begin{lemma}\label{lm:new:comp0}
	For  $x\le v-e_1$ and $y\le v-e_2$ 
		\be
		\label{eq:new:comp0}
			\wt I_{x,v+e_2} \ge \wt I_{x,v} \ge \wt I_{x,v+e_1} 
			\qquad  \text{ and } \qquad 
			\wt J_{y, v+e_2} \le \wt J_{y,v} \le \wt J_{y, v+e_1}\,.
		\ee
\end{lemma}

\begin{proof}
	Let $v=(m,n)$. 
	The proof goes   by an induction argument. 
	Suppose $x = (k,n)$ for some $k < m$. Then on the north boundary
		\begin{align*}
			\wt I_{(k,n), (m, n+1)} 
			&= \wt \Gpp_{(k,n), (m, n+1)} - \wt \Gpp_{(k +1,n), (m, n+1)} \\
			&= \wt Y_{k,n} + \wt \Gpp_{(k+1,n), (m, n+1)} \vee \wt \Gpp_{(k,n+1), (m, n+1)} - \wt \Gpp_{(k +1,n), (m, n+1)}\\
			&\ge \wt Y_{k,n}= \wt \Gpp_{(k,n), (m, n)} - \wt \Gpp_{(k+1,n), (m, n)} = \wt I_{(k,n), (m,n)}\,.
		\end{align*}
	On the east boundary, when $y= (m,\ell)$ for some $\ell < n$
		\begin{align*}
			\wt J_{(m,\ell),(m,n+1)}&=\wt \Gpp_{(m,\ell),(m,n+1)} - \wt \Gpp_{(m,\ell+1),(m,n+1)}\\
			&=\wt Y_{m,\ell}  = \wt \Gpp_{(m,\ell),(m,n)}-\wt \Gpp_{(m,\ell+1),(m,n)}  = \wt J_{(m,\ell),(m,n)}\,. 
		\end{align*}
	These inequalities start the induction. Now let   $u\le v- e_1- e_2$.
	Assume by induction  that \eqref{eq:new:comp0} holds for $x=u+e_2$ and $y=u+e_1$.  
		\begin{align*}
			\wt I_{u,v+e_2}
			&=\wt \Gpp_{u,v+e_2}-\wt \Gpp_{u+e_1,v+e_2}
			= \wt Y_u + (\wt \Gpp_{u+e_2,v+e_2}-\wt \Gpp_{u+e_1,v+e_2})^+ \\
			&=\wt Y_u+(\wt I_{u+e_2,v+e_2} - \wt J_{u+e_1,v+e_2})^+  \\
			&\ge  \wt Y_u+(\wt I_{u+e_2,v} - \wt J_{u+e_1,v})^+ = \wt I_{u,v}\,.  
		\end{align*} 
		The last equality comes by repeating  the first three equalities with $v$ instead of $v+e_2$. 
	A similar argument works for $\wt I_{u,v}\ge \wt I_{u,v+e_1}$ and a symmetric argument works for the $\wt J$ inequalities. 	
\qed	
\end{proof}

The estimates needed for the proof of Theorem \ref{th:construction} come from coupling $\Gpp_{u,v}$ with the stationary LPP described in Section \ref{sec:stat-lpp}.  For the next two lemmas fix $\zeta\in\ri\Uset$ and a  
  cocycle $B(\what, x,y)=B^\zeta_\pm(\what, x,y)$ from Theorem \ref{th:cocycles}.  Let $r=\zeta\cdot e_1/\zeta\cdot e_2$ so that $\alpha=\gppa'(r\pm)$ satisfies
\be\label{bu:a}  \alpha=\Ebig[B(x,x+e_1)]  \quad\text{and}\quad  \f(\alpha)=\Ebig[B(x,x+e_2)]. \ee
 As in \eqref{eq:Gnecorner} define 
 	\begin{align}
	\label{Gnecorner5}
		\begin{split}
		\Gne_{u,v}&=B(u,v)\qquad\text{for $u\in\{v-ke_i: k\in\Z_+, \, i\in\{1,2\}\}$}\\[3pt]
	\qquad 	\text{and} \qquad 		\Gne_{u,v}
			&=\w_u+\Gne_{u+e_1,v}\vee \Gne_{u+e_2,v}\qquad \text{for } \  u\le v-e_1-e_2\,.
		\end{split}
	\end{align}

Let  $\Gne_{u,v}(A)$ denote a maximum over paths restricted to the set $A$.  In particular, below we use 
  \[  \Gne_{0,\,v}(v-e_i\in x_\centerdot) = \max_{x_\centerdot\,: \, x_{\abs{v}_1-1}=v-e_i} \sum_{k=0}^{\abs{v}_1-1}  \wt Y_{x_k}
  \]
  where the maximum is restricted to  paths that go through the point $v-e_i$, and the 
 weights are from \eqref{Gnecorner5}:  $\wt Y_x=\w_x$ for $x\le v-e_1-e_2$ while 
 $\wt Y_{v-ke_i}=B(v-ke_i, v-(k-1)e_i)$.


Figure \ref{fig:7.5}  makes the limits of the next lemma obvious.   But  a.s.\ convergence requires some technicalities  because the north-east  boundaries themselves are translated as the limit is taken.  

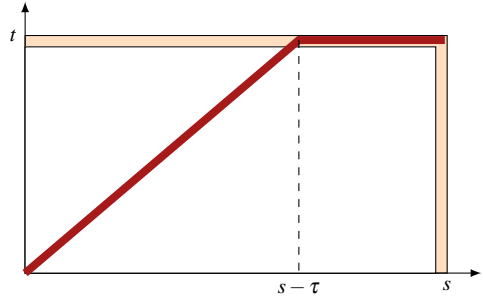
\begin{figure}[h]
	\begin{center}
		\begin{tikzpicture}[>=latex, scale=0.6]

\draw[<->] (0,6)--(0,0)--(10,0); 

			\fill[color=orange, nearly transparent](0,0)rectangle(9.25,5.25); 
			\draw(0,0)rectangle(9.25,5.25); 
			\fill[color=white](0,0)rectangle(9,5);
			\draw(0,0)rectangle(9,5);
			
			\draw(9.25,0)node[below]{$s$};
			\draw(0,5.25)node[left]{$t$};
			\draw[color=nicosred, line width=3pt](0,0)--(6,5.15)--(9.2, 5.15); 
			
			\draw[dashed] (6,5.15)--(6,0)node[below]{$s-\tau$};

%
%
%
%
%
%
%
%
%
%
%
%
%

		\end{tikzpicture}
	\end{center}
	\caption{\small Illustration of \eqref{eq:hor:lln}.  Forcing the last step to be $e_1$ restricts the maximization   to paths that hit the north boundary instead of the east boundary.  The path from $0$ to $(s-\tau,t)$ contributes $\gpp(s-\tau,t)$ and the remaining segment of length $\tau$ on the north boundary  contributes $\alpha \tau$. 
			}
	\label{fig:7.5}
\end{figure}

 \begin{lemma}
	Assume \eqref{2d-ass}. Fix reals $0<s,t<\infty$. Let $v_n\in\Z_+^2$ be such that    $v_n/\abs{v_n}_1\to(s,t)/(s+t)$ as $n\to\infty$ and  $\abs{v_n}_1\ge\eta_0n$ for some constant $\eta_0>0$. Then we have the following almost sure  limits:
		\be\label{eq:hor:lln}
			\abs{v_n}_1^{-1}\,\Gne_{0,\,v_n}(v_n-e_1\in x_\centerdot) \; \underset{n\to\infty}{ \longrightarrow }\;
			(s+t)^{-1}\sup_{0 \le \tau \le s}\{ \alpha\tau + \gpp(s-\tau,t) \}  
		\ee

and 
 		\be\label{eq:vert:lln}
			\abs{v_n}_1^{-1}\,\Gne_{0,\,v_n}(v_n-e_2\in x_\centerdot) \; \underset{n\to\infty}{\longrightarrow} \;
			(s+t)^{-1}\sup_{0 \le \tau \le t}\{ f(\alpha)\tau + \gpp(s, t-\tau) \}. 
		\ee
 \end{lemma}

\begin{proof}
	We prove \eqref{eq:hor:lln}. Fix $\e>0$, 
	  let  $M=\fl{\e^{-1}}$,  and 
		\[
			q^n_j = j\Bigl\lfloor\frac{\e\abs{v_n}_1 s}{s+t}\Bigr\rfloor \, \text{ for }0\le j \le M-1, 
				\text{ and } q^n_{M}= v_n\cdot e_1.
		\]
		For large enough $n$ it is the case that  $q^n_{M-1}<v_n\cdot e_1$.  

Suppose    a maximal path for $\Gne_{0,\,v_n}(v_n-e_1\in x_\centerdot)$ enters 
	 the north boundary from  the bulk at the point 
  $v_n-(\ell,0)$  with $q^n_{j} < \ell\le q^n_{j+1}$.  By superadditivity 
		\begin{align*}
			&\Gne_{0,\,v_n}(v_n-e_1\in x_\centerdot) 
= \Gpp_{0, \,v_n-(\ell,1)} + \w_{v_n-(\ell,1)} + 
B( v_n-(\ell,0), v_n)      
\\
			&\quad 
	\le \Gpp_{0, \,v_n-(q^n_j,1)}   +  q^n_j\alpha - \sum_{k=q^n_j+1}^{\ell-1} \bigl(\w_{v_n-(k,1)} -\Ew\bigr)  +(\ell-1-q^n_j)  \Ew \\
&\qquad\qquad  + \; \;  
\bigl( B( v_n-(\ell,0), v_n) -\ell\alpha\bigr)      
+ (\ell-q^n_j)\alpha . 	
		\end{align*}
 The two main  terms come  right after the inequality above  and the rest are errors.  	
		
Define the centered cocycle 
$   F(x,y)= h(B)\cdot(x-y)-B(x,y) $   so that 
\[   
B( v_n-(\ell,0), v_n) -\ell\alpha= F(0, v_n-(\ell,0))  -F(0, v_n).  
\] 
The potential-recovery property \eqref{eq:VB}   $\w_0=B(0,e_1)\wedge B(0,e_2)$ gives 
\[  F(0, e_i) \le \alpha\vee\f(\alpha) -\w_0\qquad\text{for $i\in\{1,2\}$.}   \]  		
The i.i.d.\ distribution of $\{\w_x\}$ and  $\E(\abs{\w_0}^{\pp})<\infty$ with $\pp>2$  are  strong enough to guarantee that Lemma  \ref{th:Atilla}   from Appendix \ref{app:aux} applies and gives 
\be\label{b:F-erg}      \lim_{N\to\infty} \; \frac1N \; \max_{x\ge 0\,:\,\abs{x}_1\le N} \abs{F(\what,0,x)} =0\qquad
\text{for a.e.\ $\what$.}    	
\ee	
		
 Collect the bounds for all the intervals $(q^n_j, q^n_{j+1}]$ and let $C$ denote a constant.  Abbreviate $S^n_{j,m}=\sum_{k=q^n_j+1}^{q^n_j+m} \bigl(\w_{v_n-(k,1)} -\Ew\bigr)$. 
	\be	\begin{aligned}
			&\Gne_{0,\,v_n}(v_n-e_1\in x_\centerdot)  
				\;\le\;\max_{0\le j \le M-1 }\!\!\Big\{  \Gpp_{0, \,v_n-(q^n_j,1)}   +  q^n_j\alpha + C(q^n_{j+1}-q^n_j)  \\
	&\quad +   \max_{0\le m< q^n_{j+1}-q^n_j} \abs{S^n_{j,\,m}} 
+  \max_{q^n_{j} < \ell\le q^n_{j+1}} F(0, v_n-(\ell,0))  -F(0, v_n) 
\Big\}.
		\end{aligned}\label{G:bnd} \ee 
Divide through by $\abs{v_n}_1$ and let $n\to\infty$.   Limit   \eqref{lln5}  gives convergence of the $G$-term on the right.  
  We claim that the terms on the second line of \eqref{G:bnd} vanish.  Limit  \eqref{b:F-erg}    takes care of the $F$-terms. 		
 Combine Doob's maximal inequality for martingales with 	Burkholder's inequality 
 \cite[Thm.~3.2]{Bur-73}	to obtain, for $\delta>0$,  
\begin{align*}
&\P\Bigl\{  \,\max_{0\le m< q^n_{j+1}-q^n_j} \abs{S^n_{j,m}} \ge \delta\abs{v_n}_1 \Bigr\} \le \frac{ \E\bigl[ \abs{S^n_{j, \,q^n_{j+1}-q^n_j}}^{\pp}\bigr] }{\delta^{\pp}\abs{v_n}_1^{\pp} } \\
&\qquad\qquad\qquad
 \le  \frac{C}{\delta^{\pp}\abs{v_n}_1^{\pp} } \, { \E\biggl[ \,\biggl\lvert \,\sum_{i=1}^{q^n_{j+1}-q^n_j} \bigl(\w_{i,0} -\Ew\bigr)^2\biggr\rvert^{\pp/2} \,  \biggr] }  
\le \frac{C}{\abs{v_n}_1^{\pp/2} } . 
\end{align*} 	
Thus Borel-Cantelli takes 	care of the $S^n_{j,m}$-term on the second line of \eqref{G:bnd}.  (This is the place where the assumption $\abs{v_n}_1\ge \eta_0n$ is used.)   We have the upper bound
\begin{align*}
\varlimsup_{n\to\infty}   \abs{v_n}_1^{-1} \Gne_{0,\,v_n}(v_n-e_1\in x_\centerdot)  
    \le 
 (s+t)^{-1} \max_{0\le j \le M-1 }\big[\gpp(s-sj\e, t) +sj\e \alpha + C\e s \big].
 \end{align*} 
 Let $\e\searrow 0$ to complete the proof of the upper bound.

  To get the  matching lower bound let  the supremum $\sup_{\tau\in [0,s]} \{ \tau\alpha +  \gpp(s-\tau , t)\}$ be attained at $\tau^* \in [0,s]$.  
	With $m_n=\abs{v_n}_1/(s+t)$ we have 
 		\begin{align*}
	\Gne_{0,\,v_n}(v_n-e_1\in x_\centerdot) &\ge 	\Gpp_{0,v_n-(\fl{m_n\tau^*}\vee1,1)}\; +\; \w_{v_n-(\fl{m_n\tau^*}\vee1,1)} \\[2pt]  & \qquad  
	+  \; B(v_n- (\fl{m_n\tau^*}\vee1,0), \, v_n). 
		\end{align*}
Use again the cocycle $F$ from above,  and let $n\to\infty$ to get 
		\begin{align*}
\varliminf_{n\to\infty}  		\abs{v_n}_1^{-1}\Gne_{0,\,v_n}(v_n-e_1\in x_\centerdot)
		\ge  (s+t)^{-1}[\gpp(s-\tau^*,t) + \tau^* \alpha] . 
		\end{align*}
This completes  the proof of  \eqref{eq:hor:lln}.
\qed\end{proof}

 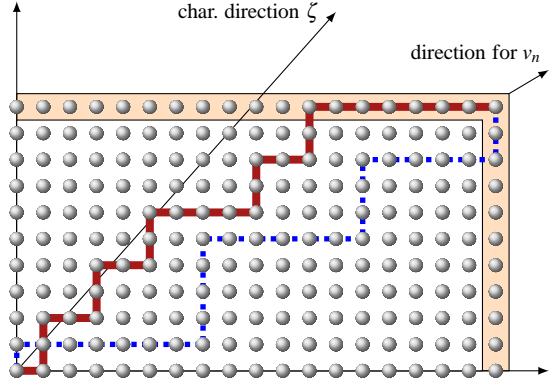
\begin{figure}[h]
	\begin{center}
		\begin{tikzpicture}[ >=latex, scale=0.7]

			\draw[<->] (0,7)--(0,0)--(10,0); 

			\fill[color=orange, nearly transparent](0,0)rectangle(9.25,5.25); 
			\draw(0,0)rectangle(9.25,5.25); 
			\fill[color=white](0,0)rectangle(8.75,4.75);
			\draw(0,0)rectangle(8.75,4.75);

			\draw[->] (0,0)--(6,6.8) node[left]{\small char.~direction $\zeta$\phantom{a}};

			\draw[line width=3pt, color=nicosred](9,5)--(5.5,5)--(5.5,4)--(4.5,4)--(4.5,3)--(2.5, 3)--(2.5, 2)--(1.5,2)--(1.5,1)--(.5,1)--(.5,0)--(0,0); 
			
			\draw[line width=2pt, dotted, color=blue](9,5)--(9,4)--(6.5,4)--(6.5,2.5)--(3.5, 2.5)--(3.5,1.5)--(3.5,0.5)--(.5,.5)--(0,.5)--(0,0); 
			
			 \foreach \x in {0,...,18}{
              			 \foreach \y in {0,...,10}{
					\shade[ball color=really-light-gray](\x*0.5,\y*0.5)circle(1.3mm);  
								}
							} 
			\draw[->] (9.25, 5.25)--(10,5.7)node[above left]{\small direction for $v_n$}; 
			


			
			

					\end{tikzpicture}
	\end{center}
	\caption{\small  Illustration of Lemma \ref{lm:exit:con}.  With $\alpha$-boundaries geodesics tend to go in the $\alpha$-characteristic direction $\zeta$.  If   $v_n$ converges in a direction  below $\zeta$,  maximal paths to $v_n$ tend  to hit  the north boundary. The dotted path that hits  the east boundary is unlikely to  be maximal
	for large $n$. }
	\label{fig:7.6}
\end{figure}

Continue with the stationary LPP defined by \eqref{Gnecorner5} in terms of a cocycle $B=B^\zeta_{\pm}$, with $r=\zeta\cdot e_1/\zeta\cdot e_2$ and $\alpha$ as in \eqref{bu:a}.    Let us call the direction $\zeta$ {\sl characteristic} for $\alpha$.  
  The next lemma shows 	that in stationary LPP a maximizing path  to a point   below the characteristic direction    will eventually hit    the north boundary before the east boundary.  (Illustration in Figure \ref{fig:7.6}.)    We omit  the entirely analogous result  and proof for  a point above the characteristic line.      

\begin{lemma} \label{lm:exit:con}    Let $s\in(r, \infty)$.  Let   $v_n\in\Z_+^2$  be such that   $v_n/\abs{v_n}_1\to(s,1)/(1+s)$ and $\abs{v_n}_1\ge\eta_0n$ for some constant $\eta_0>0$.
Assume that $\gppa'(r+)>\gppa'(s-)$.  
	Then
 	$\Pbig$-a.s.\ there exists a random $n_0 < \infty$ such that for all $n \ge n_0$, 
		\be\label{eq:move:e1}
		\Gne_{0,\,v_n} = 
		\Gne_{0, v_n}(v_n-e_1\in x_\centerdot). 
		\ee
\end{lemma}

\begin{proof} 
	The right derivative at $\tau=0$ of 
$\alpha \tau + \gpp(s-\tau, 1) = \alpha \tau +  \gppa(s- \tau)$	
	equals 
		\[
			\alpha - \gppa'(s-) > \alpha - \gppa'(r+) \ge 0.
		\]
	The last inequality above follows from the assumption on $r$.  
	Thus we can find  $\tau^*\in(0,r)$ such that 
		\be\label{eq:hor:ineq}
			\alpha \tau^*+\gpp(s-\tau^*, 1) > \gpp(s,1).
		\ee
		
	  To produce a contradiction let $A$ be the event  on which $\Gne_{0,\,v_n} = \Gne_{0, v_n}(v_n-e_2\in x_\centerdot)$ for infinitely many $n$ and assume $\Pbig(A)>0$.    Let $m_n=\abs{v_n}_1/(1+s)$.   
	On $A$ we have for infinitely many $n$ 
		\begin{align*} 
			&\abs{v_n}^{-1}	\Gne_{0,\,v_n}(v_n-e_2\in x_\centerdot) \; = \; \abs{v_n}^{-1} \Gne_{0, v_n}\\[3pt]
			&\qquad \ge \abs{v_n}^{-1} B(v_n-(\fl{m_n\tau^*}+1)e_1, v_n)
			+ \abs{v_n}^{-1}\Gpp_{0,v_n-(\fl{m_n\tau^*}+1,1)}\\
			&\qquad\qquad +\abs{v_n}^{-1}\w_{v_n -(\fl{m_n\tau^*}+1,1)}\notag.
		\end{align*}
		Apply \eqref{eq:vert:lln} to the leftmost quantity.  Apply limits \eqref{lln5} and \eqref{b:F-erg} and stationarity and integrability of $\w_x$  to the expression on the right. 
	  Both extremes  of the above inequality converge almost surely.   
Hence on the event $A$ the inequality is preserved to the limit and yields (after multiplication by $1+s$) 
		\begin{align*}
		\begin{split}
		\sup_{0\le \tau \le 1} \{ \f(\alpha) \tau + \gpp(s,1-\tau)\}  
		\ge \alpha \tau^* + \gpp(s-\tau^*,1).
		\end{split}
		\end{align*}
The   supremum of the left-hand side is achieved at $\tau = 0$ because the right derivative  equals 
\[  \f(\alpha) -\gppa'(\tfrac{1-\tau}s-) \le \f(\alpha) -\gppa'(r^{-1}-)  \le 0
\]
where the  first   inequality comes from  $s^{-1}<r^{-1}$ and the second from   \eqref{eq:f=ga'}. Therefore
		\[  \gpp(s,1) \ge \alpha \tau^* + \gpp(s-\tau^*, 1) \]
	which contradicts \eqref{eq:hor:ineq}. Consequently $\Pbig(A)=0$ and \eqref{eq:move:e1} holds for $n$ large.
\qed
\end{proof}

\begin{proofof}{of Theorem \ref{th:construction}} The proof  goes in two steps.\medskip

{\bf Step 1.} First consider a fixed $\xi=(\frac{s}{1+s}, \frac{1}{1+s})\in\ri\Uset$ and a fixed sequence $v_n$ such 
  that $v_n/\abs{v_n}_1\to \xi$ and $\abs{v_n}_1\ge \eta_0 n$ for some $\eta_0>0$. 
  We prove that  the last inequality of \eqref{bu:G1} holds almost surely.  
	Let $\zeta=(\frac{r}{1+r}, \frac{1}{1+r})$ satisfy  $\zeta\cdot e_1<\ximin\cdot e_1$ so that  $\gppa'(r+)>\gppa'(s-)$ and Lemma \ref{lm:exit:con} can be applied. 
	Use cocycle  $B^{\zeta+}$ from Theorem \ref{th:cocycles} to define   last-passage times    $\Gne_{u, v}$ as in \eqref{Gnecorner5}.   Furthermore,  define last-passage times    $\Gn_{u, v}$ that use cocycles only on the north boundary and bulk weights elsewhere: 
		\begin{align*}
		\begin{split}
			\Gn_{v-ke_1,v}=B^{\zeta+}(v-ke_1,v) \,,
			&\qquad \Gn_{v-\ell e_2,v}=\sum_{j=1}^{\ell} \w_{v-je_2}\,,\\
			\quad\text{and}\qquad 	\Gn_{u,v}
			&=\w_u+\Gn_{u+e_1,v}\vee\Gn_{u+e_2,v}\quad \text{ for } \ u\le v-e_1-e_2\,.
		\end{split}
	\end{align*}
For large  $n$  we have 
		\begin{align*}
			\Gpp_{x,v_n} -  \Gpp_{x+e_1,v_n} &\le   \Gn_{x,v_n+e_2}-  \Gn_{x+e_1,v_n+e_2} \\
			 &  =  \Gne_{x,v_n+e_1+e_2}(v_n+ e_2\in x_\centerdot)-  \Gne_{x+e_1,v_n+e_1+e_2}(v_n+ e_2\in x_\centerdot)\\
		&  =  \Gne_{x,v_n+e_1+e_2} -  \Gne_{x+e_1,v_n+e_1+e_2}
		= B^\zeta_+(x, x+e_1).   		\end{align*}
The first inequality above is the first inequality of \eqref{eq:new:comp0}.  The first equality above  is obvious.  The second  equality is Lemma \ref{lm:exit:con} and the last     equality is \eqref{stat-lpp5}. 	Thus 
		\[\varlimsup_{n\to \infty} \big(\Gpp_{x,v_n} -  \Gpp_{x+e_1,v_n}\big) \le B^\zeta_+(x,x+e_1).\]

	Let $\zeta\cdot e_1$  increase to $\ximin\cdot e_1$.  Theorem \ref{th:cocycles}\eqref{th:cocycles:cont} implies
		\[\varlimsup_{n\to \infty}\big( \Gpp_{x,v_n} -  \Gpp_{x+e_1,v_n}\big) \le B^{\ximin}_-(x,x+e_1).\]
		
	An analogous  argument gives the matching lower bound (first inequality of \eqref{bu:G1}) by taking $\zeta\cdot e_1>\ximax\cdot e_1$ and by reworking Lemma \ref{lm:exit:con} for the case where the direction of $v_n$ is above the characteristic direction $\zeta$.  
	Similar  reasoning  works for vertical increments  $\Gpp_{x,v_n} -  \Gpp_{x+e_2,v_n}$.\medskip

{\bf Step 2.} We prove the full statement of Theorem \ref{th:construction}. 
Let $\eta_\ell$ and $\zeta_\ell$ be two sequences in $\ri\Uset$ such that
$\eta_\ell\cdot e_1<\ximin\cdot e_1\le\ximax\cdot e_1<\zeta_\ell\cdot e_1$, $\eta_\ell\to\ximin$, and $\zeta_\ell\to\ximax$. Let $\Ombig_0$ be the event on 
  which  limits \eqref{eq:cont} hold for directions $\ximax$ and $\ximin$ (with sequences $\zeta_\ell$ and $\eta_\ell$, respectively) and
\eqref{bu:G1} holds  for each  direction $\zeta_\ell$ with sequence $\fl{n\zeta_\ell}$, and for each    direction $\eta_\ell$ with  sequence $\fl{n\eta_\ell}$.  
$\Pbig(\Ombig_0)=1$ by Theorem \ref{th:cocycles}\eqref{th:cocycles:cont} and Step 1.

Fix any $\what\in\Ombig_0$ and a sequence   $v_n$ as in \eqref{eq:vn66}.  Abbreviate $a_n=\abs{v_n}_1$.  For each $\ell$ reason as follows.   For large  $n$  
\[\fl{a_n\eta_\ell}\cdot e_1<v_n\cdot e_1<\fl{a_n\zeta_\ell}\cdot e_1
\quad\text{and}\quad 
\fl{a_n\eta_\ell}\cdot e_2>v_n\cdot e_2>\fl{a_n\zeta_\ell}\cdot e_2.
\]
By repeated application of  the first inequality of \eqref{eq:new:comp0},  
	\[\Gpp_{x,\fl{a_n\zeta_\ell}}-\Gpp_{x+e_1,\fl{a_n\zeta_\ell}}\le\Gpp_{x,v_n}-\Gpp_{x+e_1,v_n}\le\Gpp_{x,\fl{a_n\eta_\ell}}-\Gpp_{x+e_1,\fl{a_n\eta_\ell}}.\]
Take $n\to\infty$ and apply \eqref{bu:G1} to the sequences $\fl{a_n\zeta_\ell}$ and  $\fl{a_n\eta_\ell}$.  This works because $\fl{a_n\zeta_\ell}$  is a subset of $\fl{n\zeta_\ell}$ that escapes to infinity.  Thus for  $\what\in\Ombig_0$
	\begin{align*}
	B^{\overline\zeta_\ell}_+(\what,x,x+e_1)&\le\varliminf_{n\to\infty}(\Gpp_{x,v_n}(\w)-\Gpp_{x+e_1,v_n}(\w))\\
	&\le\varlimsup_{n\to\infty}(\Gpp_{x,v_n}(\w)-\Gpp_{x+e_1,v_n}(\w))
	\le B^{\underline\eta_\ell}_-(\what,x,x+e_1).
	\end{align*}
Take $\ell\to\infty$ and apply \eqref{eq:cont} to arrive at \eqref{bu:G1} as stated.
  \eqref{bu:G2} follows similarly. 
\qed\end{proofof}

\section{Cocycles from queuing fixed points}
\label{app:q} 

This section proves Theorem \ref{th:cocycles}.  At the end of the section we address briefly the exactly solvable case.   By shifting the variables $\{\w_x, B^\xi_{\pm}(x,x+e_i)\}$    in Theorem \ref{th:cocycles} if necessary, we can assume without loss of generality that $\P\{\w_0\ge0\}=1$.   Then the weights $\w_x$ can represent service times in  queueing theory.    We switch   to queuing terminology    to enable the reader to relate this section  to  queueing literature.

	Consider an infinite sequence of  $\;\cdot/\text{G}/1/\infty/\text{FIFO}$   queues in tandem.  That is, each  queue or service station (these terms are used interchangeably)   has a  general service time distribution (the law of $\w_x$ under $\P$), a single server,   unbounded room for customers waiting to be served,  and customers obey first-in-first-out discipline.  The   service stations are  
	indexed by $k\in\Z_+$ and a bi-infinite sequence  of customers is indexed by $n\in\Z$.  Customers enter the system at station 0 and move from station to station in order.   The server at station $k$ serves one customer at a time.  Once the service of customer  $n$ is complete at station $k$, customer $n$ moves to the  back of the queue at station $k+1$ and customer $n+1$ enters service at station $k$ if they were   already waiting in the queue.   If the queue at station $k$ is empty after the departure of customer $n$, then server $k$ remains idle until customer $n+1$ arrives.   Each customer retains their integer label as they move through the system.  
	
    The system has two independent  inputs:  (i) a stationary, ergodic,  arrival process $A^0=\{\A_{n,0}\}_{n\in\Z}$ and (ii)   i.i.d.\ service times $\{\S_{n,k}\}_{n\in\Z,k\in\Z_+}$  with distribution  $\S_{0,0}\overset{d}=\w_0$ under $\P$.   $A_{n,0}\ge 0$ is the time between the arrival of customer $n$ and customer $n+1$ at queue $0$.  $S_{n,k}\ge 0$ is the amount  of time the service of customer $n$ takes at station $k$.      Assume 
\be\label{EA>ES}	 E[\S_{0,0}]=\Ew<E[\A_{0,0}]<\infty. \ee

The development begins with the waiting times at station $0$.  Define  the stationary, ergodic process $\{W_{n,0}\}_{n\in\Z}$   by  	
 		\begin{align}
		\label{def:W}
   			\W_{n,0} = \Big( \sup_{j\le n-1} \sum_{i=j}^{n-1} ( \S_{i,0} - \A_{i,0} ) \Big)^+ . 
   		\end{align}
By the ergodic theorem and \eqref{EA>ES},  			
		 $
			\W_{n,0}<\infty$ a.s.\ $ \forall n\in\Z 		 
		$. 
   $\{W_{n,0}\}$ satisfies  Lindley's equation: 
	\begin{align}
   			&\W_{n+1,0} 
			=(\W_{n,0}+\S_{n,0}-\A_{n,0})^+.\label{q:lin}
  		 \end{align}
The  interpretation of $\W_{n,0}$ as the waiting time of customer $n$ is natural from this equation.     If  $W_{n,0}+S_{n,0}< A_{n,0}$ then customer $n$ leaves  station $0$ before customer $n+1$ arrives, and consequently customer $n+1$ has no wait  and $W_{n+1,0}=0$.  In the complementary case customer $n+1$   waits time $W_{n+1,0}=W_{n,0}+S_{n,0}-A_{n,0}$ before entering service at station $0$.


		
	\begin{lemma}  $n^{-1}W_{n,0}\to 0$ almost surely  as $n\to\infty$.  
	\label{t:lm-W_n} \end{lemma}

	\begin{proof}   Abbreviate $\U_n=\S_{n,0}-\A_{n,0}$.   For $a\ge 0$ and $\e>0$  define 
	\begin{align*}
W^\e_0(a)=a \quad\text{and}\quad 
 	W^\e_{n+1}(a)  = \bigl(   W^\e_n(a)  + U_n -E(U_0) +\e \bigr)^+\quad\text{for $n\ge 0$}.
	\end{align*}
	Check inductively  that 
	\[    W^\e_n(0)  = \Bigl( \;\max_{0\le m<n}  \sum_{k=m}^{n-1}  [U_k -E(U_0) +\e] \Bigr)^+ .\]
	Consequently 
	\[   W^\e_n(a)\ge  W^\e_n(0)  \ge  \sum_{k=0}^{n-1}  [U_k -E(U_0) +\e]
	\longrightarrow \infty   \quad\text{ as $n\to\infty$.} \]  
	Thus $W^\e_n(a)>0$ for large $n$ which implies, from its definition, that 
	for large $n$ 
	\[    W^\e_{n+1}(a)  =    W^\e_n(a)  + U_n -E(U_0) +\e.  \] 
	Another application of the ergodic theorem gives   $n^{-1}   W^\e_n(a)  \to \e$  $\P$-a.s.\ as $n\to\infty$.  

	Now for the conclusion.   Since $W_{0,0}=W^\e_0(W_{0,0}) $, we can check inductively that 
	\begin{align*}   W_{n+1,0}&=(\W_{n,0}+U_n)^+\le \bigl(W^\e_n(\W_{0,0})+U_n\bigr)^+ \\
	&\le  \bigl(W^\e_n(\W_{0,0})+U_n-E(U_0) +\e \bigr)^+=W^\e_{n+1}(\W_{0,0}). 
	\end{align*}
	From this,  $0\le n^{-1}W_{n,0}\le n^{-1}   W^\e_n(W_{0,0}) \to\e$, and we let $\e\searrow0$. 
	\qed\end{proof} 			
		
	The  stationary and ergodic process  $A^1=\{\A_{n,1}\}_{n\in\Z}$ of inter-departure times from queue 0 (equivalently, inter-arrival times at queue $1$)   is  defined by 
		\begin{align}
		\A_{n,1} = ( \W_{n,0} + \S_{n,0} - \A_{n,0} )^- + \S_{n+1,0}.
		\label{def:A}
		\end{align}
To see that this definition agrees with common sense, consider the two possible scenarios.  If $W_{n+1,0}>0$ from \eqref{q:lin} then customer $n+1$ is already waiting   and  goes immediately into service after the departure of customer $n$.  The time $A_{n,1}$ between the departures of customers $n$ and $n+1$	is then exactly the service time $\S_{n+1,0}$ of customer $n+1$.  In the complementary case $W_{n+1,0}=0$,  server $0$ is idle for time $\A_{n,0} - \S_{n,0} - \W_{n,0}$  before  customer $n+1$ arrives.  The time between the departures is this idle time plus the service time $\S_{n+1,0}$.  

Combining  equations \eqref{q:lin} and \eqref{def:A}  	 and iterating gives 
		\[
			\W_{1,0}+\S_{1,0}+\sum_{i=1}^n\A_{i,1}=\W_{n+1,0}+\S_{n+1,0}+\sum_{i=1}^n\A_{i,0} \qquad\text{for $n\ge 1$.} 
		\]
This and  Lemma \ref{t:lm-W_n} imply  $E[\A_{0,1}]=E[\A_{0,0}]$.
 	(In the queueing literature, this has been observed in \cite{Loy-62}.) 
	
These steps are repeated successively at each station $k=1,2,3,\dotsc$. 
   The stationary, ergodic  arrival process  $A^k=\{\A_{n,k}\}_{n\in\Z}$  at station $k$  is the departure process from station $k-1$.   $\A_{n,k}$ is  the inter-arrival time between customers $n$ and $n+1$ at station $k$, or, equivalently, the inter-departure time between customers $n$ and $n+1$ from  station $k-1$.  $A^k$  is  independent of the service times  $\{\S_{n,j}\}_{n\in\Z, \, j\ge k}$ because $A^k$ was constructed as a function of the given initial arrival process $A^0$ and the service times $\{\S_{n,j}\}_{n\in\Z, \,0\le j\le k-1}$.   $W_{n,k}$ is the waiting time of customer $n$ at queue $k$, that is, the time between the arrival of customer $n$ at station $k$ and the beginning of their service at station $k$.  These are  defined by 
\begin{align}
		\label{def:Wk}
   			\W_{n,k} = \Big( \sup_{j\le n-1} \sum_{i=j}^{n-1} ( \S_{i,k} - \A_{i,k} ) \Big)^+, \qquad n\in\Z . 
   		\end{align}	
Properties  $\W_{n,k}<\infty$, Lemma \ref{t:lm-W_n},   and $E[\A_{n,k}]=E[\A_{0,0}]$ are preserved along the way. 	
 The total time customer $n$ spends at station $k$ is  
   	the sojourn time  $\W_{n,k}+\S_{n,k}$.  


This procedure constructs the  process $\{ \A_{n,k}, \S_{n,k}, \W_{n,k}\}_{ n\in\Z, k\in\Z_+}$  that satisfies the   following system of equations:   
		\begin{align}
			\begin{split}\label{eq:ne-induction:q}
 				&\W_{n+1,k}+\S_{n+1,k}=\S_{n+1,k}+(\W_{n,k}+\S_{n,k}-\A_{n,k})^+,\\
				&\A_{n,k+1} = (\W_{n,k}+\S_{n,k}-\A_{n,k})^- + \S_{n+1,k},\\
				&\S_{n+1,k}=(\S_{n+1,k}+\W_{n+1,k})\wedge \A_{n,k+1}.
			\end{split}
		\end{align} 
		The appearance of $\S_{n+1,k}$ on both sides of the first line is intentional.  
	The third equation follows  from the first two by taking the minimum of either  side.    Subtracting the second line  from the first line in 
	  \eqref{eq:ne-induction:q} gives   the ``conservation law''
		\begin{align}
		\label{eq:closed-loop:q}
			\W_{n+1,k}+\S_{n+1,k}+A_{n,k}=\W_{n,k}+\S_{n,k}+\A_{n,k+1}.
		\end{align}

 As a product of an ergodic process and an i.i.d.\ one, the input   
$\{\A_{n,0}, \S_{n,k}\}_{n\in\Z, k\in\Z_+}$  is   stationary and ergodic under   translations of the $n$-index.   	
Consequently so is  the entire queuing system.
  A {\sl fixed point}   is a distribution of   $\{A_{n,0}\}_{n\in\Z}$ such that the system is also stationary   under translations of the $k$-index.   
 
%

The next  four statements  
summarize the situation with fixed points, quoted  from articles \cite{Mai-Pra-03, Pra-03}. 
Fix the  i.i.d.\ distribution $\P$ of the service times $\{\S_{n,k}\}_{n\in\Z,\,k\in\Z_+}$.   Given a stationary ergodic probability measure $\mu$ on $\R^\Z$, consider the  random variables 
		$
			\{\A_{n,0},\,\S_{n,0},\,\W_{n,0},\,\A_{n,1}\}_{n\in\Z}
		$
	where $\{\A_{n,0}\}_{n\in\Z}$ $\sim$ $\mu$ are independent of   $\{\S_{n,0}\}_{n\in\Z}$ $\sim$ $\P$, $\W_{n,0}$ is defined via \eqref{def:W}, 
	and $\A_{n,1}$ is defined via \eqref{def:A}.
	Let $\Phi(\mu)$ denote the distribution of the sequence  $\{\A_{n,1}\}_{n\in\Z}$.  $\Phi$ is the {\sl queueing operator} whose fixed points are the focus now.  

	Let  $\M_e^\alpha(\R^\Z)$ be the space of translation-ergodic probability measures $\mu$ on $\R^\Z$ with marginal mean  $E^\mu[\A_{0,0}]=\alpha$.   From the discussion above we know that $\Phi$ maps  $\M_e^\alpha(\R^\Z)$ into $\M_e^\alpha(\R^\Z)$.    We are interested in ergodic fixed points, so we define
		\begin{align*}
			\cA=\{\alpha\in(\Ew,\infty):\exists\mu\in\M_e^\alpha(\R^\Z)\text{ such that } \Phi(\mu)=\mu\}.
		\end{align*}

         		\begin{theorem} \label{th:Pra-03}  \text{\rm\cite[Thm.~1]{Pra-03}} \,
			Let $\alpha\in\cA$. Then there exists a unique $\mu^\alpha\in\M_e^\alpha(\R^\Z)$ with $\Phi(\mu^\alpha)=\mu^\alpha$.
			Furthermore, let $\A^0=\{\A_{n,0}\}_{n\in\Z}$ be ergodic with mean 
			$E[\A_{0,0}]=\alpha$ and $\{\S_{n,k}\}_{n\in\Z,k\in\Z_+}$  
			i.i.d.\ with distribution $\P$ and independent of   $\A^0$.
			Let $A^k=\{\A_{n,k}\}_{n\in\Z}$, $k\in\N$,  be defined via inductions 
			\eqref{def:W} and \eqref{eq:ne-induction:q}.
			Then as $k\to\infty$ the distributions of $A^k$ converge weakly to $\mu^\alpha$.  
		\end{theorem}
 
		\begin{theorem}   \text{\rm\cite[Thm.~5.1 and 6.4 and Lm.\ 6.3(a)]{Mai-Pra-03}} 
			\, Set $\cA$ is a nonempty, closed subset of $(\Ew, \infty)$, $\inf\cA=\Ew$, and $\sup\cA=\infty$. If $\alpha<\beta$ are both in $\cA$ then $\mu^\alpha\le\mu^\beta$ in the usual sense of  
	  stochastic ordering.  
		\end{theorem}
		
				\begin{lemma} \text{\rm\cite[Lm.\ 6.3(b)]{Mai-Pra-03}}  \,
		\label{lm:EJ=f}
			Let $\alpha\in\cA$,  $A^0\sim\mu^\alpha$, and $\{S_{n,k}\}\sim\P$ independent of $A^0$.  
			Define $\W_{n,0}$ via \eqref{def:W}. 
			Then 
\be\label{q:500} E^{\mu^\alpha\otimes\P}[\W_{0,0}+\S_{0,0}]=\f(\alpha).\ee
		\end{lemma}
          
Suppose $\alpha\in(\Ew,\infty)\cap\cA^c$.  Let 
		\[
			\amin=\sup\bigl( \cA\cap(\Ew,\alpha]\bigr) \in\cA\quad\text{and}\quad\amax=\inf\bigl(\cA\cap[\alpha,\infty)\bigr)\in\cA, 
		\]
$t=(\amax-\alpha)/(\amax-\amin)$ and   $\mu^\alpha=t\mu^{\amin}+(1-t)\mu^{\amax}$. 
	Now $\mu^{\alpha}$ is a mean $\alpha$ fixed point of $\Phi$.   This fixed point is again attractive, in the following sense. 
 
		\begin{theorem} \text{\rm\cite[Prop.\ 6.5]{Mai-Pra-03}} \,
		\label{th:attractive}
			Let $\alpha>\Ew$.
			Let $\A^0=\{\A_{n,0}\}_{n\in\Z}$ be ergodic with mean 
			$E[\A_{0,0}]=\alpha$ and $\{\S_{n,k}\}_{n\in\Z,k\in\Z_+}$  
			i.i.d.\ with distribution $\P$ and independent of the $\{\A_{n,0}\}$-process.
			Let $\{\A_{n,k}:n\in\Z,k\in\N\}$ be defined via inductions 
			\eqref{def:W} and \eqref{eq:ne-induction:q}.
			Then as $k\to\infty$  the Ces\`aro mean of the distributions of $\{\A_{n,k}\}_{n\in\Z}$ converges weakly to $\mu^\alpha$.
		\end{theorem}

Based on the development in \cite{Mai-Pra-03} we deduce auxiliary facts.  

\hbox{}

\hbox{}
 

\begin{lemma}\noindent\label{q:lm5} 
\begin{enumerate}[\ \ \rm(a)]
\item\label{q:lm5:a}  Let $\amin< \amax$ be points in $\cA$ such that  $(\amin, \amax)\subset\cA^c$.  Then $\f$ is linear on the interval  $[\amin, \amax]$.  
\item\label{q:lm5:b}  Let $\xi\in\Diff$, $s=\xi\cdot e_1/\xi\cdot e_2$ and $\alpha=\gppa'(s)$.  Then $\alpha\in\cA$.  
\end{enumerate}
\end{lemma}

\begin{proof}  Part \eqref{q:lm5:a}.   Let $0<t<1$ and  $\alpha= t\amin+(1-t)\amax$.  In the notation of \cite{Mai-Pra-03}, consider a sequence of tandem queues
$(\A^k, \S^k, \W^k, \A^{k+1})_{k\in\Z_+}$ where the initial arrival process  $\A^0=\{A_{n,0}\}_{n\in\Z}$ is ergodic with mean $E(A_{n,0})=\alpha$,  the service times   $\{S^k\}_{k\in\Z_+}=\{\S_{n,k}\}_{n\in\Z, k\in\Z_+}$ are independent of $\A^0$ and i.i.d.\ $\P$-distributed, and the remaining variables are defined iteratively.    Let $(\wh\A, \wh\S, \wt\W, \wt\D)$ denote a weak limit point of the Ces\`aro averages of the distributions of $(\A^k, \S^k, \W^k, \A^{k+1})$.  (The tightness argument is on p.~2225 of \cite{Mai-Pra-03}.)   Then, as shown in \cite[eqn.~(29)]{Mai-Pra-03} in the course of the proof of their Theorem 5.1,   $\wt\W=\Psi(\wh\A, \wh\S)$ where the mapping  $\Psi$ encodes definition \eqref{def:W}.   By Theorem \ref{th:attractive}   \cite[Prop.~6.5]{Mai-Pra-03}    the distribution of $\wh\A$ is  $t\mu^{\amin}+(1-t)\mu^{\amax}$.   By \cite[Theorem 4.1]{Mai-Pra-03}, 
\be\label{MP-lln5} n^{-1}\sum_{k=0}^{n-1} \W_{0,k} \to M(\alpha)\equiv  \f(\alpha)-\Ew
\qquad\text{almost surely.}\ee
  Combine these facts as follows.   First  
\begin{align*}
E(\wt\W_0)&=E[\Psi(\wh\A, \wh\S)_0]= t E^{\mu^{\amin}\otimes\P}[\Psi(\wh\A, \wh\S)_0] + (1-t) E^{\mu^{\amax}\otimes\P}[\Psi(\wh\A, \wh\S)_0]\\
  &=   tM(\amin)+(1-t)M(\amax)
\end{align*}
where the last equality comes from \cite[Lemma 6.3(b)]{Mai-Pra-03} restated as Lemma \ref{lm:EJ=f} above.    The weak limit, combined with the law of large numbers \eqref{MP-lln5}
and dominated convergence, gives, for any $c<\infty$ and along a subsequence, 
\begin{align*}
E(\wt\W_0\wedge c)&= \lim_{n\to\infty}  n^{-1}\sum_{k=0}^{n-1} E(\W_{0,k}\wedge c) \le  \lim_{n\to\infty}  E\Bigl[\,\Bigl(  n^{-1}\sum_{k=0}^{n-1} \W_{0,k}\Bigr) \wedge c\Bigr] = M(\alpha)\wedge c\\
&\le M(\alpha).  
\end{align*}  
Letting $c\nearrow\infty$ gives 
\[    tM(\amin)+(1-t)M(\amax) \le  M(\alpha).   \]
Since $M$ is convex and $\f$ differs from $M$ by a constant, this implies 	
$\f(\alpha)=  t\f(\amin)+(1-t)\f(\amax)$ and completes the proof of part \eqref{q:lm5:a}. 	

\smallskip

Part \eqref{q:lm5:b}.    If  $\alpha\in\cA^c$,  there exist $\amin< \amax$   in $\cA$ such that  $\alpha\in (\amin, \amax)\subset\cA^c$.    By part \eqref{q:lm5:a}  $f$ is linear on  $[\amin, \amax]$.  Basic convex analysis   implies that $\gppa$ has multiple tangent slopes at $s$ and hence cannot be differentiable  at $s$. 
%
\qed
\end{proof}

	Given $\alpha>\Ew$,  let $\{\A_{n,0}\}_{n\in\Z}$ $\sim$ 
	  $\mu^\alpha$ and i.i.d.\ $\{\S_{n,k}\}_{n\in\Z,k\in\Z_+}$ 
	$\sim$  $\P$  be independent. 
	Define $\{\W_{n,k},\A_{n,k+1}\}_{n\in\Z,k\in\Z_+}$ using \eqref{def:Wk} 
	and \eqref{eq:ne-induction:q}.
	Because $\Phi(\mu^\alpha)=\mu^\alpha$,  
	$\{\A_{n,k},\S_{n,k},\W_{n,k}\}_{n\in\Z,k\in\Z_+}$ is stationary in $n$ and  $k$.
	Extend this to the   stationary process $(\A, \S, \W)=\{\A_{n,k},\S_{n,k},\W_{n,k}\}_{n,k\in\Z}$ indexed by  $\Z^2$.    
Define  another 	$\Z^2$-indexed stationary process $(\Ap, \Sp, \Wp)$ by  
\be\label{q:til}  ( \Ap_{i,j},\Sp_{i,j}, \Wp_{i,j}) = 
(\W_{j-1,i+1}+\S_{j-1,i+1} , \, \S_{j,i}, \, \A_{j-1,i+1}-\S_{j,i} ), \quad (i,j)\in\Z^2.  \ee

\begin{lemma}\label{q:erg-lm} 
Suppose  $\alpha\in\cA$.  Then     the process $(\A, \S, \W)$ is ergodic   under   translation $T_{e_1}$,  and also ergodic under  $T_{e_2}$.   Furthermore,  $\f(\alpha)\in\cA$.   $(\Ap, \Sp, \Wp)$ is a stationary queueing system where $\{\Ap_{n,0}\}_{n\in\Z}$ has distribution  $\mu^{\f(\alpha)}$, and is also ergodic under both $T_{e_1}$  and    $T_{e_2}$. 
\end{lemma} 

\begin{proof}
  The queueing construction gives $T_{e_1}$-ergodicity of $\{\A_{n,k},\S_{n,k},\W_{n,k}\}_{n\in\Z, \,k\ge\ell}$ for any $\ell\in\Z$.   
Let $B$ be a $T_{e_1}$-invariant event  of the full process $\{\A_{n,k},\S_{n,k},\W_{n,k}\}_{n, k\in\Z}$.    Write $\cG_\ell$ for the $\sigma$-algebra generated by $\{\A_{n,k},\S_{n,k},\W_{n,k}\}_{n\in\Z, k\ge\ell}$.   The conditional expectations 
   $E(\one_B\vert\cG_\ell)$  are $T_{e_1}$-invariant, hence a.s.\ constant by the ergodicity proved thus far.   $E(\one_B\vert\cG_\ell)\to\one_B$  almost surely as $\ell\to-\infty$, and  consequently $\one_B$ is a.s.\ constant.   This completes the proof of ergodicity under $T_{e_1}$.   

To get ergodicity under $T_{e_2}$ we transpose, and that leads us to look at $( \Ap, \Sp,\Wp)$ of \eqref{q:til}. 
To see that   $( \Ap, \Sp,\Wp)$ is another queueing system with the same i.i.d.\ service time distribution  $\Sp_{i,j}=\S_{j,i}$, we need to check three items.

(i)  Independence of $\{\Ap_{i,\ell}\}_{i\in\Z}$ and $\{\Sp_{i,j}\}_{i\in\Z, j\ge \ell}$, for each $\ell\in\Z$.  This follows from the structure of equations \eqref{eq:ne-induction:q} and the independence of the $\{S_{i,j}\}$.  

(ii)  $\Ap_{i, j+1}=(\Wp_{ij}+\Sp_{ij} -\Ap_{ij}  )^- +\Sp_{i+1, j} $.  
This follows from the top equation of \eqref{eq:ne-induction:q}.  

(iii)  The third point needed is
 \be\label{q:700} \Wp_{k+1,j}=\Bigl( \;\sup_{n:\,n\le k}\, \sum_{i=n}^{k} (\Sp_{ij}-\Ap_{ij})  \Bigr)^+. \ee This needs a short argument.  Fix $k,j$.   The middle equation of \eqref{eq:ne-induction:q} gives 
\be\label{q:701}  \Wp_{ij}=(\Wp_{i-1,j} +\Sp_{i-1, j} -\Ap_{i-1, j} )^+  \ee
which  can be iterated to give  
\[  \Wp_{k+1,j}=\Bigl(  \Bigl\{\Wp_{\ell j} +  \sum_{i=\ell}^{k} (\Sp_{ij}-\Ap_{ij})  \Bigr\}  \;\vee\; 
\Bigl\{ \max_{n:\ell<n\le k}  \sum_{i=n}^{k} (\Sp_{ij}-\Ap_{ij})  \Bigr\} \,  \Bigr)^+  
\qquad \text{for $\ell\le k$.}  \]
Thus \eqref{q:700} follows if $\Wp_{\ell j}=0$ for some $\ell\le k$.  Suppose on the contrary that  $\Wp_{ij}>0$ for all $i\le k$.   Apply \eqref{q:701}  to all $\Wp_{ij}$ for $n<i\le k$ and divide by $\abs n$ to get   
\[   \frac{\Wp_{kj}}{\abs n}= \frac{\Wp_{nj}}{\abs n} +  \frac1{\abs n}\sum_{i=n}^{k-1} (\Sp_{ij}-\Ap_{ij})  \]
which is the same as 
 \be\label{q:705}  \frac{\A_{j-1,k+1}}{\abs n} -  \frac{\S_{jk}}{\abs n}  \; = \;  \frac{\A_{j-1,n+1}}{\abs n} -  \frac{\S_{jn}}{\abs n} +  \frac1{\abs n}\sum_{i=n}^{k-1} (\S_{ji}-\W_{j-1,i+1}-\S_{j-1,i+1}) .  \ee
Let $n\to-\infty$.   The i.i.d.\ property of the $\{\S_{ij}\}$ and  Theorem 4.1 of  \cite{Mai-Pra-03} quoted as \eqref{MP-lln5} above,  combined with \eqref{q:500} from above,  give  the  limit  in probability 
 \be\label{q:707}   \lim_{n\to-\infty}  \frac1{\abs n}\sum_{i=n}^{k-1}  \Ap_{ij} =     \lim_{n\to-\infty}  \frac1{\abs n}\sum_{i=n}^{k-1} (\W_{j-1,i+1}+\S_{j-1,i+1}) =  f(\alpha). \ee
 The four leftmost  terms  of \eqref{q:705}  vanish as $n\to-\infty$ (by stationarity and finite expectations). 
Hence  letting $n\to-\infty$ in \eqref{q:705} along  a suitable subsequence  leads to 
$0=\Ew-f(\alpha)<0$  (the last inequality from Lemma \ref{lm:f-properties}). 
This contradiction  verifies \eqref{q:700}. 

At this point we have shown  that  the stationary process $\{\Ap_{n,0}\}_{n\in\Z}$ is a fixed point for $\Phi$ with the deterministic  pathwise limit \eqref{q:707}.  By Prop.~4.4 of  \cite{Mai-Pra-03} the process $\{\Ap_{n,0}\}_{n\in\Z}$ must be ergodic. 
%
%
  We have shown that $f(\alpha)\in\cA$.   The   part of the lemma already  proved gives the ergodicity of
the process  
\[  \{ \Ap_{ij},\Sp_{ij}, \Wp_{ij}\} = 
\{ \W_{j-1,i+1}+\S_{j-1,i+1} , \, \S_{j,i}, \, \A_{j-1,i+1}-\S_{j,i} \}  \]
under translations of the index $i$.  Since ergodicity is preserved by mappings that respect translations,  a suitable mapping of the right-hand side  above gives the $T_{e_2}$-ergodicity of $\{ \A_{nk},\S_{nk}, \W_{nk}\}$. 
\qed
\end{proof} 

\begin{proofof}{of Theorem \ref{th:cocycles}}
We begin by constructing a convenient 
\label{A_0} 
   countable subset $\cA_0$ of $\cA$.   Let $\Uset_0$ be a countable dense subset of $\ri\Uset$ such that $\Uset_0$ contains 
 \begin{itemize}
 \item[(i)] all (at most countably many)  points of nondifferentiability of $\gpp$, 
 \item[(ii)]  all endpoints of nondegenerate intervals on which $\gpp$ is linear (recall \eqref{gg-1}),    and 
\item[(iii)] 
    a  countable dense subset of points of differentiability of $\gpp$.  
 \end{itemize}     
     Then put 
 \be\label{q:A_0} \cA_0=\bigl\{ \gppa'(s\pm):   \bigl(\tfrac{s}{1+s}, \tfrac{1}{1+s}\bigr)\in\Uset_0\bigr\}. \ee    $\cA_0\subset\cA$ by virtue of Lemma \ref{q:lm5}\eqref{q:lm5:b} and the closedness of $\cA$ in $(\Ew,\infty)$.  
  

We construct a coupled 
process  \[\{(\A^\alpha, \S,  \W^\alpha): \alpha\in\cA_0\}= \{ \A^\alpha_{n,k}, \S_{n,k}, \W^\alpha_{n,k}:  (n,k)\in\Z^2, \alpha\in\cA_0\}\] 
whose 
 distribution  $\wb P$ on $\R^{\cA_0\times\Z^2}\times\R^{\Z^2}\times\R^{\cA_0\times\Z^2}$  is invariant under translations of the $(n,k)$ index over $\Z^2$  and that has    the following properties.   For each $(n,k)\in\Z^2$ we have the inequalities 
\be \label{q:ineq13}     \A^\alpha_{n,k}\le \A^\beta_{n,k}   \quad\text{and}\quad 
     \W^\alpha_{n,k}\ge \W^\beta_{n,k}  \quad\text{ for $\alpha<\beta$}.    
 \ee 
 For each  $\alpha\in\cA_0$ the marginal process  $(\A^\alpha, \S,  \W^\alpha)$ is a stationary queueing process of the type described in  Lemma \ref{q:erg-lm},  stationary and ergodic under both translations,  with $\{ \A_{n,0}^\alpha\}_{n\in\Z}\sim\mu^\alpha$ and    $\{\S_{0,k}+\W_{0,k}^\alpha\}_{k\in\Z}\sim\mu^{\f(\alpha)}$.   $\wb P$ comes  from a subsequential  Ces\`aro limit of queueing processes followed by Kolmogorov extension to the full lattice. 
  
  Start by taking for each $\alpha\in\cA_0$ the  constant   initial inter-arrival process   $\A^\alpha_{n,0}=\alpha$.  Use the iterative equations to construct 
   $(\A^{\alpha, k}, \S^k, \W^{\alpha, k})= \{ \A^\alpha_{n,k}, \S_{n,k}, \W^\alpha_{n,k}\}_{n\in\Z}$ for $k\in\Z_+$.  Use the same version of the service times $\{S_{n,k}\}$ for each $\alpha\in\cA_0$.   Inequalities  \eqref{q:ineq13}     are true for the $A$-processes at $k=0$ by construction.  They are propagated for all $k\in\Z_+$ by equations \eqref{def:Wk} and \eqref{eq:ne-induction:q}.  
   
 Let $Q^\ell$ be the joint  distribution of 
  $\{(\A^{\alpha, \ell+k}, \S^{\ell+k}, \W^{\alpha, \ell+k}): \alpha\in\cA_0, k\ge0\}$, a probability measure on the countable product space $\R^{\cA_0\times\Z\times\Z_+}\times\R^{\Z\times\Z_+}\times\R^{\cA_0\times\Z\times\Z_+}$.   By the development in Sections 5 and 6 of \cite{Mai-Pra-03},  sequence  $\{Q^\ell\}$ is tight and for each $\alpha\in\cA_0$,    $\A^{\alpha, k}$ converges weakly to $\mu^\alpha$ as $k\to\infty$ and   the   process  $(\A^{\alpha, \ell+k}, \S^{\ell+k}, \W^{\alpha, \ell+k})_{k\ge0}$  converges weakly to a stationary queueing process as $\ell\to\infty$.  In order to get a limit where the joint distribution between different $\alpha$-values is also invariant  under shifts of the $k$-index, we perform one averaging:  let $\wb Q^n=n^{-1}\sum_{\ell=0}^{n-1} Q^\ell$.    The  effect of shifting the $k$-index is 
  $\wb Q^n\circ T_{e_2}^{-1}=n^{-1}\sum_{\ell=1}^{n} Q^\ell$.   Thus any limit point   $\wb Q$  of $\{ \wb Q^n\}$  is  invariant under shifts of the $k$-index.    This  invariance  extends  $\wb Q$ to negative $k$-values and gives  an invariant  measure $\wb P$  indexed by $(n,k)\in\Z^2$.   The almost sure  inequalities  \eqref{q:ineq13}   are preserved in this construction.   $\wb P$ has the   properties described below \eqref{q:ineq13}.

Define the following mapping from the coordinates $\{(\A^\alpha, \S, \W^\alpha): \alpha\in\cA_0\}$ to the coordinates   $\{ (\w_x)_{x\in\Z^2},  (\w^{i,\alpha}_x)_{i\in\{1,2\}, \,\alpha\in\cA_0, \, x\in\Z^2}\}$ of the space 
 $\Ombig=\Omega\times \R^{\{1,2\}\times\cA_0\times\Z^2}$:   for $(n,k)\in\Z^2$ and  $\alpha\in\cA_0$, 
 \begin{align}\label{q:map} 
( \om_{n,k}, \, \w^{1,\alpha}_{n,k}, \, \w^{2,\alpha}_{n,k})
=  (\S_{-n,-k},\,\A^\alpha_{-n-1,-k+1},\,\W^\alpha_{-n,-k}+\S_{-n,-k})  . 
 \end{align}
Let $\Pbig$ be the distribution induced on $\Ombig$ by this mapping, from the joint distribution $\wb P$  of the coupled stationary queueing processes.  

The probability space $\OBPbig$ of Theorem \ref{th:cocycles}  has now been constructed.  
For $\xi\in\Uset_0$ and $i=1,2$   define the functions   $B^\xi_{\pm}(\what, x,x+e_i)$ as the following coordinate projections: 
\be\label{q:B3}     B^\xi_{\pm}(\what,x,x+e_i)= \w^{i,\gppa'(s\pm)}_x \qquad\text{for $s=\xi\cdot e_1/\xi\cdot e_2$.}   
\ee
The set $\cA_0$ was constructed to ensure  $\gppa'(s\pm)\in\cA_0$ for each $\xi\in\Uset_0$ so these functions  are well-defined.   For each $\xi\in\Uset_0$ and $x,y\in\Z^d$, follow the same deterministic procedure to extend these functions to $B^\xi_{\pm}(\what,x,y)$ for all $x,y\in\Z^d$.  For example, through these steps:

$\bf\bullet$  Set  $B^\xi_{\pm}(\what,x,x)=0$.   For $x\le y$ fix a path $x=x_0,x_1,\dotsc,x_m=y$ such that all $e_1$ steps come before $e_2$ steps and set $B^\xi_{\pm}(\what,x,y)=\sum_{i=0}^{m-1}  B^\xi_{\pm}(\what, x_i,x_{i+1})$.   If $y\le x$,   set $B^\xi_{\pm}(\what,x,y)=-B^\xi_{\pm}(\what,y,x)$.

$\bf\bullet$   If there is no path between $x$ and $y$, let $z=(x\cdot e_1\wedge y\cdot e_1\,, \, x\cdot e_2\wedge y\cdot e_2)$  and   set $B^\xi_{\pm}(\what, x,y)=B^\xi_{\pm}(\what, x,z)+B^\xi_{\pm}(\what,z, y)$.

\smallskip 

The remainder  of the proof of Theorem \ref{th:cocycles} consists of two steps: (a) verification that the processes 	$B^\xi_{\pm}(x,y)$  thus defined  for $\xi\in\Uset_0$ satisfy all the properties required by Theorem \ref{th:cocycles} and  (b) definition of  $B^\xi_{\pm}(x,y)$ for {\sl all} $\xi\in\ri\Uset$  through right  limits followed by  another verification of the required properties.    

\medskip

In part \eqref{th:cocycles:indep} of Theorem \ref{th:cocycles} the   stationarity and ergodicity of each process 
\[\psi^{\pm, \xi}_x(\what)=(\w_x, B^\xi_{\pm}(x,x+e_1), B^\xi_{\pm}(x,x+e_2))\]  under both translations $T_{e_1}$ and $T_{e_2}$  are a consequence  of Lemma \ref{q:erg-lm}. 
The independence claim follows from the   queuing construction.  For any given $(n,k)$ and  $m<-k$,  $\A^\alpha_{-n-1,-k}$  is  
a function of $(\{ \A^\alpha_{i,m}\}_{ i\le -n-1}, \{\S_{i,j}\}_{ i\le -n, \, m\le j\le -k-1})$ and 
$\W^\alpha_{-n-1,-k}$  is  
a function of $(\{ \A^\alpha_{i,m}\}_{ i\le -n-2},$  $ \{\S_{i,j}\}_{ i\le -n-1, \, m\le j\le -k})$.   
These inputs are independent of $\S_{-n,-k}$. 


Part \eqref{th:cocycles:exist} of Theorem \ref{th:cocycles} requires the cocycle properties. 
For   a given   $\xi\in\Uset_0$, conservation law \eqref{eq:closed-loop:q} 
translates into the  $\Pbig$-almost sure property 
\[	B^\xi_{\pm}(x,x+e_2) + B^\xi_{\pm}(x+e_2,x+e_1+e_2) =  B^\xi_{\pm}(x,x+e_1) + B^\xi_{\pm}(x+e_1,x+e_1+e_2). 	  \]	
  Thus     $\{B^\xi_{\pm}(x,y)\}_{x,y\in\Z^2}$ is  additive.   Stationarity came in the previous paragraph and integrability is the integrability of the $\A$, $\S$ and $\W$ variables.  

 The mean vectors satisfy 
\begin{align*}
h_\pm(\xi) &=  -\, \bigl( \,\Ebig[B^\xi_{\pm}(0,e_1)]\,, \, \Ebig[B^\xi_{\pm}(0, e_2)]\, \bigr) 
=  - \bigl( E[A^{\gppa'(s\pm)}_{0,0}],  E[W^{\gppa'(s\pm)}_{0,0}+S_{0,0}]\,\bigr)\\
&= -\bigl(   \gppa'(s\pm) , \f(\gppa'(s\pm))\bigr)  = -\nabla \gpp(\xi\pm).  
\end{align*} 
The fact that one-sided gradients satisfy the duality \eqref{eq:duality} is basic convex analysis. 

The bottom equation of \eqref{eq:ne-induction:q}  translates into the potential-recovery property 
\[\w_x= B^\xi_{\pm}(x,x+e_1) \wedge  B^\xi_{\pm}(x,x+e_2) \qquad \text{$\Pbig$-a.s.}   \]
Part \eqref{th:cocycles:exist} of Theorem \ref{th:cocycles} has been verified for $B^\xi_{\pm}(\what,x,y)$ for $\xi\in\Uset_0$. 

Part \eqref{th:cocycles:flat} of Theorem \ref{th:cocycles} is the equality of   cocycles that share the mean vector.  This is clear from  \eqref{q:B3}  and  the construction   because $h_\pm(\xi)$ determines $\gppa'(s\pm)$.  

For the inequalities of part \eqref{th:cocycles:cont},  let  $s=\xi\cdot e_1/\xi\cdot e_2$  and $t=\zeta\cdot e_1/\zeta\cdot e_2$ for $\xi, \zeta\in\Uset_0$.   Then 
 $\xi\cdot e_1<\zeta\cdot e_1$ implies $s<t$. By concavity $\gppa'(s-)\ge \gppa'(s+)\ge \gppa'(t-)$    and the first inequality of \eqref{q:ineq13} gives 
 $A^{\gppa'(s-)}_{n,k}\ge A^{\gppa'(s+)}_{n,k}\ge A^{\gppa'(t-)}_{n,k}$ which translates into the first inequality of \eqref{eq:monotone}.   Similarly the second 
 inequality of \eqref{q:ineq13} gives  the second  inequality of \eqref{eq:monotone}.   
   Let  $\Ombig_1$ be  the event on which  
   inequalities \eqref{eq:monotone} hold for all countably many  $\xi, \zeta\in\Uset_0$.  
   
For $\zeta\in\Uset_0$ define 
$Y(\what, \zeta, x,x +e_1)=\sup_{\xi\in\Uset_0: \,\xi\cdot e_1>\zeta\cdot e_1} B^\xi_{\pm}(x,x+e_1)$.   Then  for any sequence $\xi_n\in\Uset_0$ such that 
   $\xi_n\cdot e_1\searrow\zeta\cdot e_1$, by monotonicity  
\be\label{q:Blim}  \lim_{n\to\infty}B^{\xi_n}_{\pm}(x,x+e_1)  = Y(\what, \zeta, x,x +e_1) \le  B^\zeta_{+}(x,x+e_1)   \quad\text{for all $\what\in\Ombig_1$.}  
\ee
Monotonicity of the family of cocycles gives a bound that justifies dominated convergence, and hence 
\be\label{B-mon-6}  \Ebig\bigl[ \, \lim_{n\to\infty}  B^{\xi_n}_{\pm}(x,x+e_1)\bigr] =    \lim_{n\to\infty}  \gppa'(s_n\pm) =  \gppa'(t+)  =   \Ebig\bigl[  B^\zeta_{+}(x,x+e_1)\bigr] .
\ee
Equality of expectations forces a.s.\ equality in \eqref{q:Blim}.  We now have a.s.\ right continuity \eqref{eq:cont} for the case $(x,y)=(x,x+e_1)$. Analogously deduce a.s.\  right continuity \eqref{eq:cont} for  $(x,y)=(x,x+e_2)$, and  a.s.\ left continuity \eqref{eq:cont-left}  for  $(x,y)=(x,x+e_i)$.   Then a.s.\ left and right  continuity  follow for all $(x,y)$ by the construction of  $B^\xi_{\pm}(x,y)$  in terms of the nearest-neighbor values $B^\xi_{\pm}(x,x+e_i)$.  

Let $\Ombig_0$ be the full $\Pbig$-measure subset of $\Ombig_1$  on which   limits \eqref{eq:cont}  and  \eqref{eq:cont-left}  hold for all  $ \zeta\in\Uset_0$ when $\xi_n\to\zeta$ in $\Uset_0$.   
  
\medskip 
 
 Theorem \ref{th:cocycles}  has now been verified for  $B^\xi_{\pm}$   for $\xi\in\Uset_0$.   The next step  defines $B^\zeta_{\pm}$ for   $\zeta\in(\ri\Uset)\smallsetminus\Uset_0$. Since all points of nondifferentiability of $\gpp$ were included in $\Uset_0$,  $\zeta\in\Diff$.  Then  we define $B^\zeta_{\pm}$ as equal and denote it by $B^\zeta$.   We choose  right limits for the definition.  So for   $\zeta\in(\ri\Uset)\smallsetminus\Uset_0$ set  
\be\label{q:Blim3} \begin{aligned}  B^\zeta(\what, x,x+e_1) =B^\zeta_\pm(\what,x,x+e_1) =  \sup_{\xi\in\Uset_0 \,: \,  \xi\cdot e_1>\zeta\cdot e_1}B^\xi_\pm(\what,x,x+e_1 ) \\   
 B^\zeta(\what, x,x+e_2) =B^\zeta_\pm(\what, x,x+e_2) =  \inf_{\xi\in\Uset_0\,: \,  \xi\cdot e_1>\zeta\cdot e_1}B^\xi_\pm(\what, x,x+e_2 ).  
\end{aligned} \ee 

On the  event $\Ombig_0$ of full $\Pbig$-probability defined above, 
definition \eqref{q:Blim3} extends inequalities  \eqref{eq:monotone}  and right limits \eqref{eq:cont}   to all $\xi, \zeta\in\ri\Uset$, nearest-neighbor edges $(x,x+e_i)$  and sequences  $\xi_n\cdot e_1\searrow\zeta\cdot e_1$.    Extend the nearest-neighbor values $B^\zeta(x,x+e_i)$ to all $B^\zeta(x,y)$ by the procedure used earlier after \eqref{q:B3}.  Then right limits  \eqref{eq:cont}  work for all $B^\zeta(x,y)$ and sequences  $\xi_n\cdot e_1\searrow\zeta\cdot e_1$.

Fix $\zeta\in(\ri\Uset)\smallsetminus\Uset_0$.  We argue that outside  a single $\Pbig$-null set   specific to $\zeta$, we get  the left limit   \eqref{eq:cont-left}.   Define 
\begin{align*}
Y(\what, \zeta, x,x +e_1)&=\inf_{\xi\in\Uset_0 \,: \,  \xi\cdot e_1<\zeta\cdot e_1} B^\xi_{\pm}(\what, x,x+e_1)\\
Y(\what, \zeta, x,x +e_2)&=\sup_{\xi\in\Uset_0 \,: \,  \xi\cdot e_1<\zeta\cdot e_1} B^\xi_{\pm}(\what, x,x+e_2). 
\end{align*}
The left limit $Y(\what, \zeta,x,x+e_i) = \lim_{n\to\infty}B^{\zeta_n}_\pm(\what, x,x+e_i) $  happens for $\what\in\Ombig_0$ by monotonicity, for any sequence   $\zeta_n\cdot e_1\nearrow\zeta\cdot e_1$ in $\ri\Uset$. 
Now set  
\[   \Ombig^{(\zeta)} = \{ \what\in\Ombig_0:   Y(\what, \zeta, x,x +e_i)=B^\zeta(\what, x,x +e_i) \; \forall x\in\Z^2, i=1,2\}  .\]  
The monotonicity argument with coinciding expectations used above in \eqref{B-mon-6} implies that $\Pbig(\Ombig^{(\zeta)})=1$.  
From nearest-neighbor values the limits extend to all $(x,y)$ by the construction, and so on the full measure event $\Ombig^{(\zeta)}$ we have  \eqref{eq:cont-left}.

We turn to verifying the remaining claims of  Theorem \ref{th:cocycles} for the fully  defined processes $B^\xi_\pm(x,y)$.   First, right-continuity in $\xi$ is enough to make $B^\xi_+(\what,x,y)$ a jointly Borel function of  $(\what,\xi,x,y)$.    Since $B^\xi_-$ replaces the value of $B^\xi_+$ with a different Borel function of $(\what,x,y)$  only at the countably many  $\xi\in(\ri\Uset)\smallsetminus\Diff$,  $B^\xi_-(\what,x,y)$ is also jointly Borel.  

Part \eqref{th:cocycles:indep}.   
Stationarity and the independence claim are  preserved by almost sure limits but ergodicity is not.   To verify the ergodicity of  $\psi^{\zeta}_x(\what)=(\w_x, B^\zeta(\what,x,x+e_1), B^\zeta(\what,x,x+e_2))$  under both translations $T_{e_1}$ and $T_{e_2}$ we return to the queuing picture.  The limit \eqref{q:Blim3} can also be taken in the queueing processes.   First $\cA_0\ni\alpha_n=\gppa'(s_n-)\nearrow \gppa'(t)=\beta$.    Since $\cA$ is closed,  $\beta\in\cA$.  Hence there is a stationary queueing process $(\A^\beta, \S, \W^\beta)$ that satisfies Lemma \ref{q:erg-lm} and that we can include in the coupling with  the queueing processes indexed by  $\cA_0$.   The coordinatewise  monotone a.s.\ limit $\lim_{n\to\infty}(\A^{\alpha_n}, \S, \W^{\alpha_n})$ must coincide with $(\A^\beta, \S, \W^\beta)$ by the same reasoning used above:  there are inequalities, namely  $\lim_{n\to\infty}\A^{\alpha_n}_{m,k}\le \A^{\beta}_{m,k}$ and $\lim_{n\to\infty}\W^{\alpha_n}_{m,k}\ge \W^{\beta}_{m,k}$,  
 but the expectations agree and hence force a.s.\ agreement.   The continuous mapping \eqref{q:map}  transports the distribution of 
$\{  (\S_{-n,-k},\,\A^\beta_{-n-1,-k+1},\,\W^\beta_{-n,-k}+\S_{-n,-k}) : n,k\in\Z\}$  to the process 
$\{ (\w_x, B^\zeta(x,x+e_1),B^\zeta(x,x+e_2)): x\in\Z^2\}$,  which   thereby inherits from Lemma \ref{q:erg-lm}  the   ergodicity claimed in part \eqref{th:cocycles:indep} of  Theorem \ref{th:cocycles}.  

The cocycle properties and expectations in part \eqref{th:cocycles:exist} are preserved by the  limits.   The identities of part \eqref{th:cocycles:flat} continue to hold without null sets because if vector $h(\zeta)$ is not unique to $B^\zeta$, then 
$\zeta$ lies in the interior of some  linear segment  $]\xi',\xi''[$  of $\gpp$ with $\xi',\xi''\in\Uset_0$ and $\xi'\cdot e_1<\xi''\cdot e_1$.   The construction  (\eqref{q:B3} and \eqref{q:Blim3}) then implies that 
$B^\zeta= B^{\xi''}_-=B^{\xi'}_+$ for all $\zeta\in ]\xi',\xi''[$.  
 The   inequalities and limits of parts \eqref{th:cocycles:cont}--\eqref{th:cocycles:cont-left} were discussed above.   This completes the proof of    Theorem \ref{th:cocycles}. 
 \qed\end{proofof}

\subsection{Exactly solvable models}\label{sec:solv}
 
We describe briefly how the calculations work in  the exactly solvable geometric case discussed in Section \ref{subsec:solv}.     (The exponential case is completely analogous.) The weights $\{\w_x\}$  are  i.i.d.\   with $\P(\w_x=k)=(1- {\Ew}^{-1})^{k-1} {\Ew}^{-1}$ for $k\in\N$, mean $\Ew=\E(\w_0)>1$ and variance $\sigma^2=\Ew(\Ew-1)$.    



   With  i.i.d.\ geometric service times $\{\S_{n,0}\}$  with mean $\Ew$, let the initial arrival process $\{A_{n,0}\}$ be  i.i.d.\ geometric with mean $\alpha$.  Let $J_n=S_{n,0}+W_{n,0}$.   Then equations \eqref{q:lin}    and  \eqref{def:A} show that   the process $\{(A_{n,1}, J_{n+1})\}_{sn\in\Z}$ is an irreducible  aperiodic  Markov chain with transition probability 
\be\label{g:trans}  \begin{aligned}
&P( A_{n,1}=b, J_{n+1}=j\,\vert\, A_{n-1,1}=a, J_{n}=i)\\
&\qquad\qquad 
=  P\bigl\{  (A_{n,0}-i)^+ +S_{n+1,0}=b,  \,  (i-A_{n,0})^+ +S_{n+1,0}= j\bigr\}.  
 \end{aligned}\ee
The equations also show that $(A_{n,0}, S_{n+1,0})$ is independent of $(A_{n-1,1}, J_{n})$.    Since  the process $\{(A_{n,1}, J_{n+1})\}_{n\in\Z}$ is  stationary, its marginal must be the unique invariant distribution of transition \eqref{g:trans}, namely 
\[   P(A_{n-1,1}=k, J_{n}=j) =   (1-{\alpha}^{-1})^{k-1}  {\alpha}^{-1} \cdot   (1- {\f(\alpha)}^{-1})^{j-1}  {\f(\alpha)}^{-1} \qquad \text{for $k,j\in\N$,}  
\]
with $\f(\alpha) = \Ew\frac{\alpha-1}{\alpha-\Ew} $. 
This shows that i.i.d.\ mean $\alpha$ geometric is a queuing  fixed point. 
Next solve for $\gppa(s)=\inf_{\alpha>\Ew}\{\alpha s+\f(\alpha)\}$.  The unique minimizing $\alpha$ in terms of $s=\xi\cdot e_1/\xi\cdot e_2$ is 
$\alpha = \Ew+\sigma  \sqrt{\xi\cdot e_2/\xi\cdot e_1}$ 
  which defines the  bijection between  $\xi\in\ri\Uset$  and $\alpha\in(\Ew,\infty)$.  From this  
\[\f(\alpha) = \Ew\frac{\alpha-1}{\alpha-\Ew}  = \Ew+\sigma \sqrt{\xi\cdot e_1/\xi\cdot e_2}.  
 \]
Finding $\gamma(s)$ gives   \eqref{eg:gpp} via \eqref{eq:gbar}  and then $\gamma'(s)$ gives   \eqref{geom:B} via \eqref{grad g-1} and \eqref{eq:h=grad}. 

The terms in the sum $J_n=S_{n,0}+W_{n,0}$ are independent, so we can  also find the distribution of the waiting time: 
\[  P(W_{n,0}=0)= \frac{\alpha-\Ew}{\alpha-1}, \quad 
P(W_{n,0}=k)= \frac{\Ew-1}{\alpha-1} \cdot   \bigl(1- {\f(\alpha)}^{-1}\,\bigr)^{k-1}  {\f(\alpha)}^{-1}  \quad (k\ge 1). 
\]

\appendix

\section{Ergodic theorem for cocycles}  
\label{app:aux}

Cocycles satisfy a uniform ergodic theorem. The following is a special case of Theorem 9.3 of  \cite{Geo-etal-15-}. Note that a one-sided bound suffices for a hypothesis.   Recall   Definition \ref{def:cK} for the space $\Cor$ of centered cocycles.   

\begin{theorem}\label{th:Atilla}
Assume $\P$ is ergodic under the transformations  $\{T_{e_i}:i\in\{1,2\}\}$. 
Let $F\in\Cor$.   Assume there exists  a function $V$ such that  for $\P$-a.e.\ $\w$  
\be \label{cL-cond}
\varlimsup_{\e\searrow0}\;\varlimsup_{n\to\infty} \;\max_{x: \abs{x}_1\le n}\;\frac1n \sum_{0\le k\le\e n} 
\abs{V(T_{x+ke_i}\w)}=0\qquad\text{for $i\in\{1,2\}$ }\ee  
and 
$\max_{i\in\{1,2\}} F(\w,0,e_i)\le V(\w)$.  
Then  
\[\lim_{n\to\infty}\;\max_{\substack{x=z_1+\dotsm+z_n\\z_{1,n}\in\{e_1, e_2\}^n}} \;\frac{\abs{F(\w,0,x)}}n=0 \qquad\text{for   $\P$-a.e.\ $\w$.}\]
\end{theorem}
If the process  $\{V(T_x\w):x\in\Z^2\}$ is  i.i.d.,  then a sufficient condition for \eqref{cL-cond} is  $\E(\abs{V}^p)<\infty$ for some $p>2$  \cite[Lemma A.4]{Ras-Sep-Yil-13}

 \begin{acknowledgements}
The authors thank Yuri Bakhtin and  Michael Damron for useful discussions and two anonymous referees for valuable comments.  
N.\ Georgiou was partially supported by a Wylie postdoctoral fellowship at the University of Utah and the Strategic
Development Fund (SDF) at the University of Sussex.
F.\ Rassoul-Agha and N.\ Georgiou were partially supported by National Science Foundation grant DMS-0747758.
F.\ Rassoul-Agha was partially supported by National Science Foundation grant DMS-1407574 and by Simons Foundation grant 306576.
T.\ Sepp\"al\"ainen was partially supported by  National Science Foundation grants DMS-1306777 and DMS-1602486, by Simons Foundation grant 338287, and by  the Wisconsin Alumni Research Foundation.
\end{acknowledgements}



\bibliographystyle{spmpsci-nourlnodoi}      
\bibliography{firasbib2010}   

\begin{thebibliography}{10}
\providecommand{\url}[1]{{#1}}
\providecommand{\urlprefix}{URL }
\expandafter\ifx\csname urlstyle\endcsname\relax
  \providecommand{\doi}[1]{DOI~\discretionary{}{}{}#1}\else
  \providecommand{\doi}{DOI~\discretionary{}{}{}\begingroup
  \urlstyle{rm}\Url}\fi

\bibitem{Alm-98}
Alm, S.E.: A note on a problem by {W}elsh in first-passage percolation.
\newblock Combin. Probab. Comput. \textbf{7}(1), 11--15 (1998)

\bibitem{Alm-Wie-99}
Alm, S.E., Wierman, J.C.: Inequalities for means of restricted first-passage
  times in percolation theory.
\newblock Combin. Probab. Comput. \textbf{8}(4), 307--315 (1999).
\newblock Random graphs and combinatorial structures (Oberwolfach, 1997)

\bibitem{Arm-Sou-12}
Armstrong, S.N., Souganidis, P.E.: Stochastic homogenization of
  {H}amilton-{J}acobi and degenerate {B}ellman equations in unbounded
  environments.
\newblock J. Math. Pures Appl. (9) \textbf{97}(5), 460--504 (2012)

\bibitem{Auf-Dam-13}
Auffinger, A., Damron, M.: Differentiability at the edge of the percolation
  cone and related results in first-passage percolation.
\newblock Probab. Theory Related Fields \textbf{156}(1-2), 193--227 (2013)

\bibitem{Bak-07}
Bakhtin, Y.: Burgers equation with random boundary conditions.
\newblock Proc. Amer. Math. Soc. \textbf{135}(7), 2257--2262 (electronic)
  (2007)

\bibitem{Bak-13}
Bakhtin, Y.: The {B}urgers equation with {P}oisson random forcing.
\newblock Ann. Probab. \textbf{41}(4), 2961--2989 (2013)

\bibitem{Bak-15-}
Bakhtin, Y.: Inviscid burgers equation with random kick forcing in noncompact
  setting  (2014).
\newblock Preprint ({\tt arXiv:1406.5660})

\bibitem{Bak-Cat-Kha-14}
Bakhtin, Y., Cator, E., Khanin, K.: Space-time stationary solutions for the
  {B}urgers equation.
\newblock J. Amer. Math. Soc. \textbf{27}(1), 193--238 (2014)

\bibitem{Bak-Kha-10}
Bakhtin, Y., Khanin, K.: Localization and {P}erron-{F}robenius theory for
  directed polymers.
\newblock Mosc. Math. J. \textbf{10}(4), 667--686, 838 (2010)

\bibitem{Bal-Cat-Sep-06}
Bal{{\'a}}zs, M., Cator, E., Sepp{{\"a}}l{{\"a}}inen, T.: Cube root
  fluctuations for the corner growth model associated to the exclusion process.
\newblock Electron. J. Probab. \textbf{11}, no. 42, 1094--1132 (electronic)
  (2006)

\bibitem{Bur-73}
Burkholder, D.L.: Distribution function inequalities for martingales.
\newblock Ann. Probability \textbf{1}, 19--42 (1973)

\bibitem{Car-Hu-02}
Carmona, P., Hu, Y.: On the partition function of a directed polymer in a
  {G}aussian random environment.
\newblock Probab. Theory Related Fields \textbf{124}(3), 431--457 (2002)

\bibitem{Cat-Gro-06}
Cator, E., Groeneboom, P.: Second class particles and cube root asymptotics for
  {H}ammersley's process.
\newblock Ann. Probab. \textbf{34}(4), 1273--1295 (2006)

\bibitem{Cat-Pim-12}
Cator, E., Pimentel, L.P.R.: Busemann functions and equilibrium measures in
  last passage percolation models.
\newblock Probab. Theory Related Fields \textbf{154}(1-2), 89--125 (2012)

\bibitem{Cat-Pim-13}
Cator, E., Pimentel, L.P.R.: Busemann functions and the speed of a second class
  particle in the rarefaction fan.
\newblock Ann. Probab. \textbf{41}(4), 2401--2425 (2013)

\bibitem{Coh-Elk-Pro-96}
Cohn, H., Elkies, N., Propp, J.: Local statistics for random domino tilings of
  the {A}ztec diamond.
\newblock Duke Math. J. \textbf{85}(1), 117--166 (1996)

\bibitem{Com-Shi-Yos-03}
Comets, F., Shiga, T., Yoshida, N.: Directed polymers in a random environment:
  path localization and strong disorder.
\newblock Bernoulli \textbf{9}(4), 705--723 (2003)

\bibitem{Cor-12}
Corwin, I.: The {K}ardar-{P}arisi-{Z}hang equation and universality class.
\newblock Random Matrices Theory Appl. \textbf{1}(1), 1130,001, 76 (2012)

\bibitem{Dam-Han-14}
Damron, M., Hanson, J.: Busemann functions and infinite geodesics in
  two-dimensional first-passage percolation.
\newblock Comm. Math. Phys. \textbf{325}(3), 917--963 (2014)

\bibitem{Dur-84}
Durrett, R.: Oriented percolation in two dimensions.
\newblock Ann. Probab. \textbf{12}(4), 999--1040 (1984)

\bibitem{Dur-Lig-81}
Durrett, R., Liggett, T.M.: The shape of the limit set in {R}ichardson's growth
  model.
\newblock Ann. Probab. \textbf{9}(2), 186--193 (1981)

\bibitem{E-etal-00}
E, W., Khanin, K., Mazel, A., Sinai, Y.: Invariant measures for {B}urgers
  equation with stochastic forcing.
\newblock Ann. of Math. (2) \textbf{151}(3), 877--960 (2000)

\bibitem{Fer-Mar-Pim-09}
Ferrari, P.A., Martin, J.B., Pimentel, L.P.R.: A phase transition for
  competition interfaces.
\newblock Ann. Appl. Probab. \textbf{19}(1), 281--317 (2009)

\bibitem{Fer-Pim-05}
Ferrari, P.A., Pimentel, L.P.R.: Competition interfaces and second class
  particles.
\newblock Ann. Probab. \textbf{33}(4), 1235--1254 (2005)

\bibitem{Geo-Ras-Sep-13b-}
Georgiou, N., Rassoul-Agha, F., Sepp\"al\"ainen, T.: Stationary cocycles for
  the corner growth model  (2014).
\newblock Preprint ({\tt arXiv:1404.7786})

\bibitem{Geo-Ras-Sep-15b-}
Georgiou, N., Rassoul-Agha, F., Sepp\"al\"ainen, T.: Geodesics and the
  competition interface for the corner growth model.
\newblock Probab. Theory Relat. Fields  (2016).
\newblock To appear ({\tt arXiv:1510.00860})

\bibitem{Geo-Ras-Sep-13a-}
Georgiou, N., Rassoul-Agha, F., Sepp\"al\"ainen, T.: Variational formulas and
  cocycle solutions for directed polymer and percolation models.
\newblock Comm. Math. Phys.  (2016).
\newblock To appear ({\tt arXiv:1311.0316})

\bibitem{Geo-etal-15-}
Georgiou, N., Rassoul-Agha, F., Sepp\"al\"ainen, T., Y{\i}lmaz, A.: Ratios of
  partition functions for the log-gamma polymer.
\newblock Ann. Probab. \textbf{43}(5), 2282--2331 (2015)

\bibitem{Gly-Whi-91}
Glynn, P.W., Whitt, W.: Departures from many queues in series.
\newblock Ann. Appl. Probab. \textbf{1}(4), 546--572 (1991)

\bibitem{Hoa-Kha-03}
Hoang, V.H., Khanin, K.: Random {B}urgers equation and {L}agrangian systems in
  non-compact domains.
\newblock Nonlinearity \textbf{16}(3), 819--842 (2003)

\bibitem{Hof-05}
Hoffman, C.: Coexistence for {R}ichardson type competing spatial growth models.
\newblock Ann. Appl. Probab. \textbf{15}(1B), 739--747 (2005)

\bibitem{Hof-08}
Hoffman, C.: Geodesics in first passage percolation.
\newblock Ann. Appl. Probab. \textbf{18}(5), 1944--1969 (2008)

\bibitem{Hol-09}
den Hollander, F.: Random polymers, \emph{Lecture Notes in Mathematics}, vol.
  1974.
\newblock Springer-Verlag, Berlin (2009).
\newblock Lectures from the 37th Probability Summer School held in Saint-Flour,
  2007

\bibitem{How-New-01}
Howard, C.D., Newman, C.M.: Geodesics and spanning trees for {E}uclidean
  first-passage percolation.
\newblock Ann. Probab. \textbf{29}(2), 577--623 (2001)

\bibitem{Itu-Kha-03}
Iturriaga, R., Khanin, K.: Burgers turbulence and random {L}agrangian systems.
\newblock Comm. Math. Phys. \textbf{232}(3), 377--428 (2003)

\bibitem{Joc-Pro-Sho-98}
Jockusch, W., Propp, J., Shor, P.: Random domino tilings and the arctic circle
  theorem  (1998).
\newblock {\tt arXiv:math/9801068}

\bibitem{Joh-00}
Johansson, K.: Shape fluctuations and random matrices.
\newblock Comm. Math. Phys. \textbf{209}(2), 437--476 (2000)

\bibitem{Joh-06}
Johansson, K.: Random matrices and determinantal processes.
\newblock In: Mathematical statistical physics, pp. 1--55. Elsevier B. V.,
  Amsterdam (2006)

\bibitem{Lic-New-96}
Licea, C., Newman, C.M.: Geodesics in two-dimensional first-passage
  percolation.
\newblock Ann. Probab. \textbf{24}(1), 399--410 (1996)

\bibitem{Loy-62}
Loynes, R.M.: The stability of a queue with non-independent interarrival and
  service times.
\newblock Proc. Cambridge Philos. Soc. \textbf{58}, 497--520 (1962)

\bibitem{Mai-Pra-03}
Mairesse, J., Prabhakar, B.: The existence of fixed points for the
  {$\cdot/GI/1$} queue.
\newblock Ann. Probab. \textbf{31}(4), 2216--2236 (2003)

\bibitem{Mar-02}
Marchand, R.: Strict inequalities for the time constant in first passage
  percolation.
\newblock Ann. Appl. Probab. \textbf{12}(3), 1001--1038 (2002)

\bibitem{Mar-04}
Martin, J.B.: Limiting shape for directed percolation models.
\newblock Ann. Probab. \textbf{32}(4), 2908--2937 (2004)

\bibitem{Mut-79}
Muth, E.J.: The reversibility property of production lines.
\newblock Management Sci. \textbf{25}(2), 152--158 (1979/80)

\bibitem{New-95}
Newman, C.M.: A surface view of first-passage percolation.
\newblock In: Proceedings of the {I}nternational {C}ongress of
  {M}athematicians, {V}ol.\ 1, 2 ({Z}{\"u}rich, 1994), pp. 1017--1023.
  Birkh{\"a}user, Basel (1995)

\bibitem{Pim-07}
Pimentel, L.P.R.: Multitype shape theorems for first passage percolation
  models.
\newblock Adv. in Appl. Probab. \textbf{39}(1), 53--76 (2007)

\bibitem{Pim-13-}
Pimentel, L.P.R.: Duality between coalescence times and exit points in
  last-passage percolation models.
\newblock Ann. Probab.  (2015).
\newblock To appear ({\tt arXiv:1307.7769})

\bibitem{Pra-03}
Prabhakar, B.: The attractiveness of the fixed points of a {$\cdot/GI/1$}
  queue.
\newblock Ann. Probab. \textbf{31}(4), 2237--2269 (2003)

\bibitem{Ras-Sep-14}
Rassoul-Agha, F., Sepp{{\"a}}l{{\"a}}inen, T.: Quenched point-to-point free
  energy for random walks in random potentials.
\newblock Probab. Theory Related Fields \textbf{158}(3-4), 711--750 (2014)

\bibitem{Ras-Sep-Yil-13}
Rassoul-Agha, F., Sepp{\"a}l{\"a}inen, T., Y{\i}lmaz, A.: Quenched free energy
  and large deviations for random walks in random potentials.
\newblock Comm. Pure Appl. Math. \textbf{66}(2), 202--244 (2013)

\bibitem{Ras-Sep-Yil-15-}
Rassoul-Agha, F., Sepp{\"a}l{\"a}inen, T., Y{\i}lmaz, A.: Variational formulas
  and disorder regimes of random walks in random potentials.
\newblock Bernoulli  (2016).
\newblock To appear ({\tt arXiv:1410.4474})

\bibitem{Ros-81}
Rost, H.: Nonequilibrium behaviour of a many particle process: density profile
  and local equilibria.
\newblock Z. Wahrsch. Verw. Gebiete \textbf{58}(1), 41--53 (1981)

\bibitem{Sep-98-mprf-2}
Sepp{{\"a}}l{{\"a}}inen, T.: Coupling the totally asymmetric simple exclusion
  process with a moving interface.
\newblock Markov Process. Related Fields \textbf{4}(4), 593--628 (1998).
\newblock I Brazilian School in Probability (Rio de Janeiro, 1997)

\bibitem{Sep-98-mprf-1}
Sepp{{\"a}}l{{\"a}}inen, T.: Hydrodynamic scaling, convex duality and
  asymptotic shapes of growth models.
\newblock Markov Process. Related Fields \textbf{4}(1), 1--26 (1998)

\bibitem{Sep-12-corr}
Sepp{{\"a}}l{{\"a}}inen, T.: Scaling for a one-dimensional directed polymer
  with boundary conditions.
\newblock Ann. Probab. \textbf{40}(1), 19--73 (2012).
\newblock Corrected version available at {\tt arXiv:0911.2446}

\bibitem{Wut-02}
W{{\"u}}thrich, M.V.: Asymptotic behaviour of semi-infinite geodesics for
  maximal increasing subsequences in the plane.
\newblock In: In and out of equilibrium ({M}ambucaba, 2000), \emph{Progr.
  Probab.}, vol.~51, pp. 205--226. Birkh{\"a}user Boston, Boston, MA (2002)

\end{thebibliography}


\end{document}